\documentclass[11pt]{amsart}

\usepackage{amsmath, amsthm, amssymb}

\usepackage{lmodern}
\usepackage{enumerate}
\usepackage{palatino}                       
\usepackage[bigdelims,vvarbb]{newpxmath}    
\usepackage[scaled=0.95]{inconsolata}       
\usepackage[T1]{fontenc}                    
\usepackage{ stmaryrd }
\usepackage{comment}
\usepackage{quiver}
\usepackage[abs]{overpic}		  
\usepackage{import}
\usepackage{transparent}
\usepackage{adjustbox}
\usepackage{esvect}
\usepackage{mathtools}

\usepackage{xcolor}
\definecolor{lred}{RGB}{226, 106, 106}
\definecolor{nred}{RGB}{237, 28, 36}
\definecolor{lblue}{RGB}{52, 152, 219}
\definecolor{nblue}{RGB}{0, 174, 239}
\definecolor{lyellow}{RGB}{232, 197, 91}
\definecolor{dgreen}{RGB}{0, 148, 68}
\definecolor{l1yellow}{RGB}{217, 224, 33}
\definecolor{lgrey}{RGB}{179, 179, 179}
\definecolor{indigo}{rgb}{0.29, 0.0, 0.51}  
\usepackage[colorlinks, urlcolor=indigo, linkcolor=indigo, citecolor=indigo]{hyperref}
\usepackage[capitalize,noabbrev]{cleveref}
\usepackage[hcentering, vcentering, total={5.9in, 8.2in}]{geometry}  

\usepackage{tikz-cd}
\usetikzlibrary{cd}
\tikzset{
  symbol/.style={
    draw=none,
    every to/.append style={
      edge node={node [sloped, allow upside down, auto=false]{$#1$}}
    },
  },
}

\newtheorem{thm}{Theorem}[section]

\newtheorem{lem}[thm]{Lemma}

\theoremstyle{definition}
\newtheorem{definition}[thm]{Definition}

\newtheorem{lemma}[thm]{Lemma}
\newtheorem{proposition}[thm]{Proposition}

\theoremstyle{remark}
\newtheorem{notation}[thm]{Notation}
\newtheorem{remark}[thm]{Remark}

\newtheorem{con}[thm]{Construction}

\numberwithin{equation}{section}

\newcommand{\R}{\mathbb{R}}
\newcommand{\Z}{\mathbb{Z}}

\newcommand{\RP}{\mathbb{RP}}
\newcommand{\sm}{\setminus}
\newcommand{\ol}{\overline}

\makeatletter
\newcommand{\sbullet}{%
  \hbox{\fontfamily{lmr}\fontsize{.4\dimexpr(\f@size pt)}{0}\selectfont\textbullet}}

\makeatother

\DeclareMathOperator\Id{Id}

\DeclareMathOperator\pt{pt}

\DeclareMathOperator\Homeo{Homeo}
\DeclareMathOperator\Diff{Diff}
\DeclareMathOperator\TOP{TOP}
\DeclareMathOperator\PL{PL}

\DeclareMathOperator\GL{GL}

\DeclareMathOperator\Image{Im}

\DeclareMathOperator\CAT{CAT}
\DeclareMathOperator\DIFF{DIFF}
\DeclareMathOperator\constant{constant}
\DeclareMathOperator\Wh{Wh}
\DeclareMathOperator\St{St}
\DeclareMathOperator\pr{pr}
\DeclareMathOperator\Dic{Dic}

\title{Pseudo-isotopy versus isotopy for homeomorphisms of 4-manifolds\\  
}

\author[Daniel Galvin]{Daniel Galvin}
\address{Max Planck Institute for Mathematics, Bonn, Germany}
\email{galvin@mpim-bonn.mpg.de}
\author[Isacco Nonino]{Isacco Nonino}
\address{School of Mathematics and Statistics, University of Glasgow}
\email{Isacco.Nonino@glasgow.ac.uk} 

\begin{document}
	
	\begin{abstract}
		We define obstructions which obstruct topological pseudo-isotopies from being isotopic to isotopies in dimension four.  These match the smooth obstructions of Hatcher-Wagoner for smooth pseudo-isotopies, and accordingly are valued in certain Whitehead groups.  We show that our obstructions are fully realisable, and we use these realisations to build homeomorphisms of $Y\times S^1$ for many 3-manifolds $Y$ that are pseudo-isotopic to the identity but not isotopic to the identity.
	\end{abstract}
	
	\maketitle
	
	\section{Introduction}
 
	Understanding the space of pseudo-isotopies, denoted $ \mathcal{P}^{\CAT}(X)$, of a $\CAT$ manifold $X^n$ has been a relevant topic in algebraic topology for many years (where $\CAT$ stands for either the smooth or topological category, denoted $\DIFF$ or $\TOP$). Roughly speaking, \emph{pseudo-isotopies} are $\CAT$ isomorphisms $M \times I \to M \times I$ that are not necessarily level preserving (see \cref{definition:pseudo_isotopy}).
 
 Work of Cerf \cite{Cerf} showed that if $M$ is smooth and simply connected, $n \ge 5$, then a smooth pseudo-isotopy is always smoothly isotopic to a smooth isotopy. In particular, if $f$ is smoothly pseudo-isotopic to $\Id$, then it is also smoothly isotopic to $\Id$.  Hatcher and Wagoner \cite{hatcher_wagoner_1973} and Igusa \cite{Igusa} extended Cerf's method to a two-stage obstruction theory, given by maps $\Sigma$ and $\Theta$, that in the non simply connected case ($n \ge 6$) decide whether a smooth pseudo-isotopy can be smoothly isotoped to an isotopy.  We extend these obstructions to the topological category in dimension four, building on the work of Burghelea-Lashof and Pedersen \cite{BergheleaLashofBook, pedersen_1977}, who extended the obstructions to the topological category in high dimensions.

\begin{thm}[Topological invariants]\label{thm:main}
    Let $X$ be a compact, topological 4-manifold.  Then there exists homomorphisms \[\Sigma^{\TOP} \colon \pi_0(\mathcal{P}^{\TOP}(X,\partial X))\to \Wh_2(\pi_1(X))\] and \[\Theta^{\TOP}\colon \ker\Sigma^{\TOP}\to \Wh_1(\pi_1(X); \Z/2 \times \pi_2(X))/\chi\]
    
    such that $\Sigma^{\TOP}$ and $\Theta^{\TOP}$ vanish on pseudo-isotopies topologically isotopic to isotopies, and such that the following diagrams commute.
    \[
    \begin{tikzcd}[cramped]
        \pi_0(\mathcal{P}^{\TOP}(X,\partial X)) \arrow[r,"\Sigma^{\TOP}"] & \Wh_2(\pi_1(X)) &\ker\Sigma^{\TOP} \arrow[r,"\Theta^{\TOP}"] & \Wh_1(\pi_1(X); \Z/2 \times \pi_2(X))/\chi \\
        \pi_0(\mathcal{P}^{\DIFF}(X,\partial X)) \arrow[ur,"\Sigma"'] \arrow[u] & & \ker\Sigma \arrow[ur,"\Theta"'] \arrow[u] &
    \end{tikzcd}
    \]
    Here the vertical maps are the maps induced by forgetting the smooth structure, and $\chi$ denotes the image of the map $\chi\colon K_3(\Z [\pi_1X]) \to \Wh_1(\pi_1(X);\Z/2\times \pi_2(X))$.
\end{thm}

The invariants are valued in corresponding K-theoretic Whitehead groups, whose definitions are recalled in \Cref{sbs:def_sigma} and \Cref{sbs:def_theta}, along with the definitions of the original smooth Hatcher-Wagoner invariants.  The definition of our topological invariants requires careful investigation of the differences between the $\DIFF$ and $\TOP$ categories in dimension four, together with several ad hoc steps to make sure $\Sigma^{\TOP}$ and $\Theta^{\TOP}$ are compatible with their smooth counterparts.  Furthermore, we show that these invariants are realisable in the strongest possible sense.  Here we follow the work of Singh \cite{singh2022pseudoisotopies}, but we obtain a stronger realisation result due to working in the topological category.

\begin{thm}[Realisation]\label{thm:realisation}
    Let $X$ be a compact, topological 4-manifold with good fundamental group.  Then given $ x\in \Wh_2(\pi_1(X))$ or $y\in \Wh_1(\pi_1(X); \Z/2 \times \pi_2(X))/\chi(K_3\mathbb{Z}[\pi_1X])$ there exists a pseudo-isotopy $F\colon X\times I\to X\times I$ with $\Sigma^{\TOP}(F)=x$ or $\Sigma^{\TOP}(F)=0$ and $\Theta^{\TOP}(F)=y$.
\end{thm}

By `good' in \Cref{thm:realisation} we mean in the sense of Freedman-Quinn (see \cite[Chapter 2.9]{freedman_quinn_1990} or \cite[\S19]{behrens_kalmar_kim_powell_ray_2021} for a definition).  It is known that the set of good groups includes elementary amenable groups, as well as groups of sub-exponential growth \cite{freedman_teichner_1995,krushkal_quinn_2000}.  These are the fundamental groups for which the extension of Freedman's disc embedding theorem apply (see \cite[Chapter 5]{freedman_quinn_1990} and/or \cite[\S19]{behrens_kalmar_kim_powell_ray_2021}).

We also show that these invariants satisfy certain properties, which are analogues of the properties that the smooth Hatcher-Wagoner invariants possess.  Firstly, our invariants satisfy naturality for certain inclusions of codimension zero submanifolds, which we make precise now.

\begin{proposition}
 \label{lemma:proposition_naturality_topological}
    Let $X = Y \cup_{W} Z$, where $W$ is a (connected) codimension-0 submanifold of $\partial Y$. Let $F$ be a pseudo-isotopy of $X$ that satisfies $F|_{Z \times I}=\Id$.  Let $i_{Y,X}\colon Y\to X$ be the inclusion map. Then $\Sigma^{\TOP}(F)=(i_{Y,X})_*\Sigma^{\TOP}(F|_Y)$. If $F$ lies in the kernel of $\Sigma^{\TOP}$, then we have $\Theta^{\TOP}(F)=(i_{Y,X})_*\Theta^{\TOP}(F|_Y)$.  
\end{proposition}

We also prove a duality formula for pseudo-isotopies.  This mirrors the duality formula that the Hatcher-Wagoner invariants satisfy \cite[Part I Chapter VIII; Part II 4.4]{hatcher_wagoner_1973}.

\begin{thm}\label{thm:duality}
    Let $X$ be a compact, topological 4-manifold with $k_1(X)=0$, and let $F$ be a pseudo-isotopy.  Then
    \[
    \Sigma^{\TOP}(F)= \ol{\Sigma^{\TOP}(F)}
    \]
    and, if $F\in \ker\Sigma^{\TOP}$, then
    \[
    \Theta^{\TOP}(F)= \ol{\Theta^{\TOP}(F)},
    \]
    where the overline denotes the standard involution on the respective Whitehead group (see \Cref{sec:duality}).
\end{thm}

Using \Cref{thm:realisation} and the properties listed above, we use our invariants to produce homeomorphisms of 4-manifolds which are pseudo-isotopic to the identity but not isotopic to the identity.  In particular, the homeomorphisms produced are homotopic to the identity.

\begin{thm}\label{thm:homeos}
    Let $Y^3$ be a 3-manifold whose first $k$-invariant $k_1(Y)$ is trivial and with $\pi_1(Y)$ good and not ambivalent.  Then there exists a homeomorphism $f\colon Y\times S^1\to Y\times S^1$ which is pseudo-isotopic to the identity but not isotopic to the identity.  In particular, $f$ is homotopic but not isotopic to the identity.
\end{thm}

The intersection of the conditions in the statement of \Cref{thm:homeos} is a little opaque, so we give some examples of 3-manifolds that satisfy it (for the definition of an ambivalent group see \cref{def:ambivalent}) .  For the detailed statements see \Cref{sbs:YxI}.  \Cref{thm:homeos} applies to lens spaces $L(p,q)$ for $p\geq 3$, all tetrahedral manifolds, all but one octahedral manifold, all icosahedral manifolds except the Poincar\'{e} homology sphere, and many prism manifolds.  It also applies to certain circle bundles over tori, including the 3-torus.  In particular, we classify exactly for which elliptic 3-manifolds \Cref{thm:homeos} applies.  

\begin{remark}
    Ohta-Watanabe \cite{ohta_watanabe_2023} have also constructed interesting diffeomorphisms of $Y\times S^1$ for certain spherical 3-manifolds, in particular for lens spaces and the Poincar\'{e} homology sphere.  Due to the overlap, it seems useful to compare these automorphisms a little.  Our \Cref{thm:homeos} does work for lens spaces, but does not work for the Poincar\'{e} homology sphere.  Furthermore, our homeomorphisms survive suspension (see \Cref{definition:suspension_map}), whereas the diffeomorphisms constructed by Ohta-Watanabe do not.  Hence, Ohta-Watanabe's construction seems to be orthogonal to ours.
\end{remark}

\subsection{Background}
The start of pseudo-isotopy theory was Cerf's proof that (in the simply connected case) smooth pseudo-isotopy implies smooth isotopy in dimensions greater than or equal to five \cite{Cerf}.  As already stated, work of Hatcher-Wagoner and Igusa extended this to the non-simply-connected case by showing that Cerf's method extends to a two-stage obstruction theory for pseudo-isotopies \cite{hatcher_wagoner_1973,Igusa}.  It should be noted that some of the Hatcher-Wagoner theory works in dimension four, in particular the definitions of the invariants $\Sigma$ and $\Theta$.  Work of Burghelea-Lashof and Pedersen extended this to the topological category in dimensions greater than or equal to five \cite{burghelea_lashof_1974}.  We give more details on this extension in \Cref{sec:def_top_obstructions}. 

We now briefly recap the background for 4-manifolds.  Perron and Quinn proved that topological pseudo-isotopy implies topological isotopy for topological, simply-connected 4-manifolds \cite{perron_1986,quinn_1986} (Quinn's proof was recently corrected in \cite{gghkp_2023}).  Ruberman showed that smooth pseudo-isotopy does not imply smooth isotopy in the simply-connected case using gauge theory \cite{Ruberman}, and many other examples have since followed.  Budney-Gabai and Watanabe have both constructed diffeomorphisms smoothly pseudo-isotopic to the identity but not smoothly isotopic to the identity in the non-simply-connected case \cite{budney_gabai_2019,watanabe_2020} which are detected using invariants coming from configuration spaces, and in fact Budney-Gabai have constructed ones that are not even topologically isotopic to the identity \cite{budney_gabai_2023}.  Singh showed that both $\Sigma$ and $\Theta$ were stably smoothly surjective in dimension 4 \cite{singh2022pseudoisotopies}, and used $\Theta$ to construct diffeomorphisms which are smoothly pseudo-isotopic to the identity but not smoothly isotopic to the identity for some 4-manifolds, in particular for $S^1\times S^1\times S^2$.  It would be very interesting to have a more complete understanding of how these different invariants interact with one another; the results of \cite{fernandez_gay_hartman_kosanovic_2024} suggest that the barbell diffeomorphisms constructed by Budney-Gabai can, in some cases, also be detected by $\Theta$.  

We also note that Kwasik \cite{kwasik_1987} states a version of \Cref{thm:realisation} and a version of \Cref{thm:homeos} which applies in the special case that $Y$ is the 3-torus.  However, he implicitly assumes that the Hatcher-Wagoner invariants $\Sigma$ and $\Theta$ are also defined in the topological category without any adjustments.  Although we recognize the general strategy that he presents, we point out that the first step of \cite[Proposition 3.2]{kwasik_1987} which is to ``allow topological pseudo-isotopies" and compute their $\Sigma$ invariant has a fundamental flaw. That is, topological pseudo-isotopies do not lie in the domain of Hatcher-Wagoner's map $\Sigma$. Hence, in order to state such a theorem one needs to first
define an invariant for topological pseudo-isotopies. There is much work that has to be done in that direction---this is the content of \Cref{sec:def_top_obstructions}. Even assuming such an invariant exists, there is no reason to expect that the computations work out exactly as in the smooth version. In fact, \Cref{sec:realisation_sigma} and \Cref{sec:realisation_theta} deal with exactly  this problem.  To compute our invariants, we have to make use of recent developments, i.e.\ \cite{singh2022pseudoisotopies,cha_kim_2023} and so we believe that the computation itself cannot be ignored.  All in all, it seems helpful for there to be an independent proof of these results, and our careful definition of $\Sigma^{\TOP}$ and $\Theta^{\TOP}$, together with a deep investigation of their properties, allows us to produce a more general result than the one sketched in \cite{kwasik_1987}.

\begin{remark}
    It is also worth stressing that in dimension $\ge 6$ the Hatcher-Wagoner and Igusa obstruction fits in an exact sequence which, in the $\TOP$ category and $\dim=4$ does not hold even for the case $\pi_1(M)\cong \Z$ (by the previous-mentioned work of \cite{budney_gabai_2023}). As such, we work with the two maps $\Sigma^{\TOP}$ and $\Theta^{\TOP}$, and not the exact sequence.
\end{remark}

\subsection{Outline}

We now briefly outline the contents of the paper and the structure in which we will prove our results.

The first part of this paper \Cref{sec:def_top_obstructions} is dedicated to the construction of the $\TOP$ version of $\Sigma$ and $\Theta$, where we follow the sketch from~\cite{BergheleaLashofBook}[Appendix 2].
We start by increasing the dimension of our manifold using a suspension map $S^+$, which is described in \Cref{sec:background} and was already introduced in \cite{hatcher_wagoner_1973}. Then we use Pedersen's \cite{pedersen_1977} work to reduce the problem to the 3-handle skeleton (see \Cref{def:handle_skeleton}) of the suspended manifold. We will show that restricting to the 3-handle skeleton does not lose us any information regarding the structure of $\pi_0 \mathcal{P}$. Restricting the problem to the 3-handle skeleton allows to pass to the smooth category, where we can utilize the work of Hatcher-Wagoner and Igusa, which we recall in \Cref{sec:background}, and hence define our invariants. Thanks to a careful comparison with the 2-handle skeleton we will conclude that the definition of our invariants is independent of the choice of smooth structure, and hence prove that our invariants are well defined.

\begin{remark}\label{rem:blp}
    We remark that the definition of the $\TOP$ obstruction in dimension $n\ge 6$ presented in \cite{BergheleaLashofBook}---which we use after suspending the manifold---is sketched and not fully described. Even though it is not the goal of this paper, in \Cref{sec:def_top_obstructions} we try to give extra details on how to derive a $\TOP$ analogue of the $\DIFF$ and $\PL$ statements in \cite{BergheleaLashofBook}.
\end{remark}

We then, in \Cref{sec:properties}, prove some properties of our invariants which  will be crucial later on.  In particular, we show that, for smooth pseudo-isotopies, our topological invariants match the smooth invariants of Hatcher-Wagoner.  Due to the circuitous nature of the definition of our topological invariants, this is not immediately clear.  This will complete the proof of \Cref{thm:main}.  We then prove \Cref{lemma:proposition_naturality_topological}, that our invariants are natural with respect to inclusion of certain codimension zero submanifolds.  Both of these properties will be key in proving our realisation theorem (\Cref{thm:realisation}).

We then move onto the realisation part of the paper.  This starts in \Cref{sec:handle_decompositions} by carefully defining a notion of \emph{allowed one-parameter families of topological handle decompositions} (\Cref{def:one-parameter_families}).  This allows us to circumvent a key problem, which is the absence of a topological version of Cerf's functional theory, and we emphasise that it is not the scope of this paper to develop such a theory.  Instead we content ourselves with this `ad-hoc' definition which will suffice for proving our realisation theorem.  Importantly, these families have the property that restricting to the far end of the family yields a pseudo-isotopy.

In \Cref{sec:realisation_sigma} and \Cref{sec:realisation_theta} we will use these allowed one-parameter families to produce candidate pseudo-isotopies for realising our invariants.  This realisation procedure will mimic the surjectivity theorems of Hatcher-Wagoner \cite{hatcher_wagoner_1973} in high-dimensions, modified for our purposes.  This is where we will invoke the disc-embedding theorem of Freedman \cite{freedman_1982, freedman_quinn_1990, behrens_kalmar_kim_powell_ray_2021} to allow us to produce these one-parameter families.  Again, due to the circuitous nature of the definition of our topological invariants, it is then a difficult problem to compute them on these candidate pseudo-isotopies.  Here we will invoke the properties that we established in \Cref{sec:properties} and recent theorems by Singh and Cha-Kim \cite{singh2022pseudoisotopies,cha_kim_2023} to allow us to compute these invariants, and hence prove \Cref{thm:realisation}.

We then turn to using \Cref{thm:realisation} to construct homeomorphisms of certain 4-manifolds which are pseudo-isotopic to the identity but not isotopic to the identity.  To do this, we have to address the issue that our invariants are invariants of pseudo-isotopies, not of homeomorphisms.  To conclude that the homeomorphisms produced by \Cref{thm:realisation} (by restricting to the far end of the pseudo-isotopies) are not isotopic to the identity, we need to control the \emph{inertial pseudo-isotopies}, i.e.\ pseudo-isotopies $F\colon X\times I\to X\times I$ such that $F\vert_{\partial (X\times I)}=\Id$.  We do this by proving a duality formula for pseudo-isotopies, \Cref{thm:duality}, which we do in \Cref{sec:duality}.

We then use \Cref{thm:duality} to prove \Cref{thm:interesting_homeomorphisms}, which is relatively straightforward after all of the machinery has been developed.  To finish, we spend some time investigating which 3-manifolds satisfy the conditions of \Cref{thm:interesting_homeomorphisms}, including completely answering the question for finite fundamental group 3-manifolds.  This involves studying the character tables of 3-manifold groups, since in the finite fundamental group case we can reduce the problem to representation theory.

\subsection{Organisation}
In \Cref{sec:background} we briefly recall the definition of pseudo-isotopy and review the smooth invariants introduced by \cite{hatcher_wagoner_1973} and revised by \cite{Igusa}.
In \Cref{sec:def_top_obstructions} we define the topological obstructions $\Sigma^{\TOP}$ and $\Theta^{\TOP}$.  In \Cref{sec:properties} we compare these to the already known smooth invariants; in particular, we prove \Cref{thm:main}.  In \Cref{sec:handle_decompositions} we define and develop a theory of one-parameterfamilies of topological handle decompositions.  Using this theory, in \Cref{sec:realisation_sigma} and \Cref{sec:realisation_theta} we prove \Cref{thm:realisation}, dealing with the realisation problem.  In \Cref{sec:duality} we prove a duality formula for inertial pseudo-isotopies (\Cref{thm:duality}).  Finally, in \Cref{sec:homeos} we use the duality formula from the previous section to study inertial pseudo-isotopies, and hence obtain \Cref{thm:homeos}, producing interesting homeomorphisms of $Y \times S^1$ for many 3-manifolds $Y$.

\subsection{Acknowledgements}
We would like to thank Mark Powell for his pertinent comments and always helpful conversations, many of which were instrumental in this work.  The first author would also like to thank Daniel Hartman, Paula Tru\"{o}l and Simona Vesel\'{a} for useful conversations. The idea behind this paper originated in the workshop \emph{"Algebraic Methods in 4-Manifold Topology"}, which was organized in Glasgow thanks to a focused research grant from the \emph{Heilbronn Institute for Mathematical Research}. The authors would like to thank the \emph{Max Plank Institute for Mathematics in Bonn} for their hospitality and support throughout much of the work.

\section{Background notions}\label{sec:background}
	Throughout this section let $X$ be a compact, connected, $\CAT$ $4$-manifold, potentially with non-empty boundary $\partial X$. 
	
	\begin{definition}\label{definition:pseudo_isotopy}
		Let $f,g\colon X\to X$ be a pair of $\CAT$-isomorphisms.  We say that $f$ is \emph{pseudo-isotopic} to $g$ if there exists a $\CAT$-isomorphism, called a \emph{pseudo-isotopy} \[
		F\colon X\times I\to X\times I
		\] such that 
    \begin{itemize}
        \item $F\vert_{X\times \{0\}}=f\times\Id\colon X\times\{0\}\to X\times\{0\}$
        \item $F\vert_{X\times\{1\}}=g\times\Id\colon X\times\{1\}\to X\times\{1\}$
        \item $F\vert_{\partial X\times I}=f\vert_{\partial X}\times \Id$.
    \end{itemize}  We say that $f$ and $g$ are \emph{isotopic} if they are pseudo-isotopic via a level-preserving pseudo-isotopy.
	\end{definition}

\begin{figure}[htb!]
\resizebox{0.65\linewidth}{!}{
\centering
\begin{tikzpicture}
\node[anchor=south west,inner sep=0] at (0,0){\includegraphics[width=15cm]{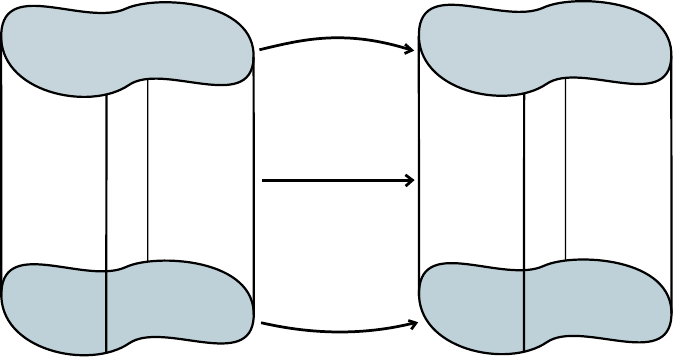}};
\node at (7.5,0.2){f};
\node at (7.5,7.4){g};
\node at (7.5,4.3){F};
\node at (1.3,3.8) {{$X \times I$}};
\node at (10.6,3.8) {{$X \times I$}};
\end{tikzpicture}}
\caption{A visual representation of a pseudo-isotopy}
\label{figure:pseudoisotopy_representation}

\end{figure}
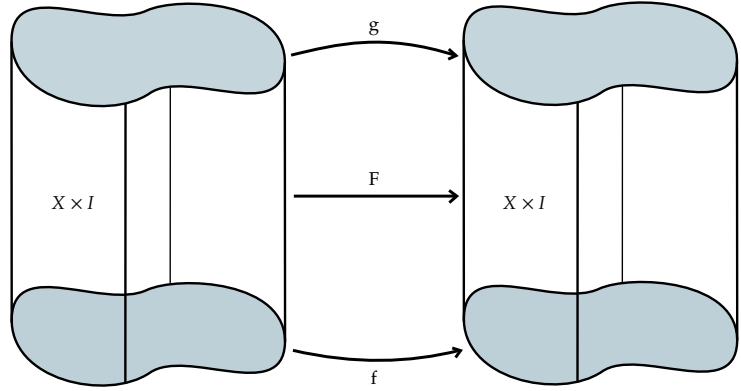

	Of course, if $f$ is isotopic to $g$ then it is also pseudo-isotopic to $g$.

\begin{definition}[Pseudo-isotopy space]\label{definition:pseudo_isotopy_space}

    $\mathcal{P}^{\CAT}(X,\partial X)$ is the topological subgroup of $\Diff(X\times I)$ or $\Homeo(X\times I)$ consisting of all pseudo-isotopies $F\colon X\times I\to X\times I$ so that $F|_{X \times 0}=\Id_X$. The group operation is given by map composition. If $\CAT=\TOP$, $\mathcal{P}^{\CAT}(X,\partial X)$ is endowed with the compact open topology (see \cite[Appendix p.594]{hatcher_2002}). If $\CAT=\DIFF$ the space is endowed with the $C^{\infty}$-topology (see e.g.\ \cite[Chapter 2]{hirsch_1994}). 
\end{definition}
   
\begin{remark}\label{remark:pi_0}
    A path in $\mathcal{P}^{\CAT}(X,\partial X)$ corresponds to a $\CAT$ isotopy between pseudo-isotopies, and so $\pi_0(\mathcal{P}^{\CAT}(X,\partial X))$ is the set of all pseudo-isotopies $X\times I\to X\times I$, based at $\Id$ considered up to  $\CAT$ isotopy.
   \end{remark} 

In this paper we will be mainly interested in $\pi_0(\mathcal{P}^{\TOP}(X,\partial X))$. In particular we are searching for topological invariants that obstruct a (topological) pseudo-isotopy from being trivial in $\pi_0(\mathcal{P}^{\TOP}(X, \partial X))$.

Next, we briefly recall the definitions of the smooth Hatcher-Wagoner invariants, which we will use at the last stage of our topological definition in \Cref{sec:def_top_obstructions}.  For the full details, see \cite{hatcher_wagoner_1973} or \cite{singh2022pseudoisotopies}.

\subsection{Definition of $\Sigma$}\label{sbs:def_sigma}

We are going to define the following map. \[\Sigma\colon \pi_0(\mathcal{P}^{\DIFF}(X,\partial X))\to \Wh_2(\pi_1(X)).\]
We begin by defining the codomain, $\Wh_2(\pi_1(X))$.
\begin{definition}[Steinberg group]
    Let $\Lambda$ be a ring.  Then we define the \emph{Steinberg group} of $\Lambda$, denoted $\St(\Lambda)$ to be group generated by symbols $x_{i,j}^{\lambda}$ where $i,j$ are distinct positive integers and $\lambda\in \Lambda$, subject to the relations:
    \begin{enumerate}
        \item $x_{i,j}^{\lambda}x_{i,j}^{\lambda'}=x_{i,j}^{\lambda+\lambda'}$,
        \item $[x_{i,j}^{\lambda},x_{k,l}^{\lambda'}]=1$ provided $j\neq k$ and $i\neq l$,
        \item $[x_{i,j}^{\lambda},x_{j,k}^{\lambda'}]=x_{i,k}^{\lambda\lambda'}$ provided $i\neq k$.
    \end{enumerate}
Then there is a natural map $\St(\Lambda)\to E(\Lambda)$, where $E(\Lambda)$ denotes the group of elementary matrices with coefficients in $\Lambda$, given by sending the symbol $x_{i,j}^{\lambda}$ to the elementary matrix with $(i,j)$-th entry $\lambda$, denoted $e_{i,j}^{\lambda}$.  Then we define the second algebraic $K$-group $K_2(\Lambda)$ to be the kernel of this map.  Now fix $\Lambda=\Z[\pi]$ where $\pi$ is a finitely presented group.  Let $w_{i,j}^{\pm g}:=x_{i,j}^{\pm g}x_{j,i}^{\mp g^{-1}}x_{i,j}^{\pm g}$.  Then define \[
\Wh_2(\pi):= K_2(\Z[\pi])/\langle w_{i,j}^{\pm g}\rangle.
\]
\end{definition}

\begin{figure}[h]
  \centering
  \begin{minipage}[t]{0.49\textwidth} 
    \centering
\begin{tikzpicture}
    \node[anchor=south west,inner sep=0] at (0,0){
    \includegraphics[width=0.9\textwidth]{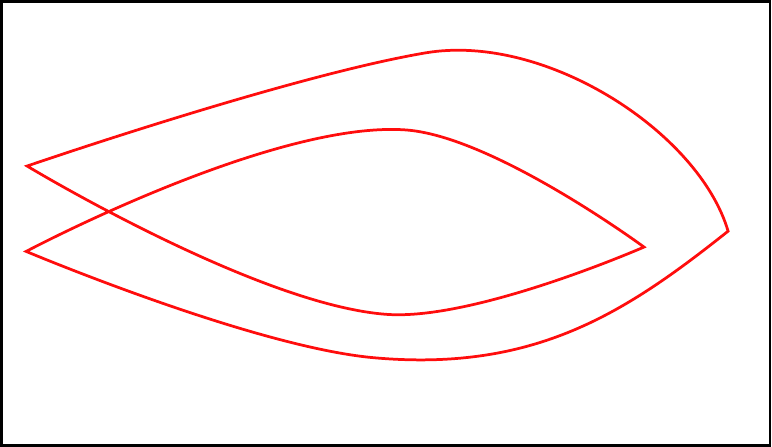}};
    \node at (-0.3,2){I};
    \node at (6,-0.2){t};
    \end{tikzpicture} 
    \caption{An example of a Cerf graphic.}
    \label{fig:example_Cerf_Graphic}

 \end{minipage}%
  \hfill
  \begin{minipage}[t]{0.49\textwidth} 
    \centering
    \begin{tikzpicture}
     \node[anchor=south west,inner sep=0] at (0,0)
     {\includegraphics[width=0.9\textwidth]{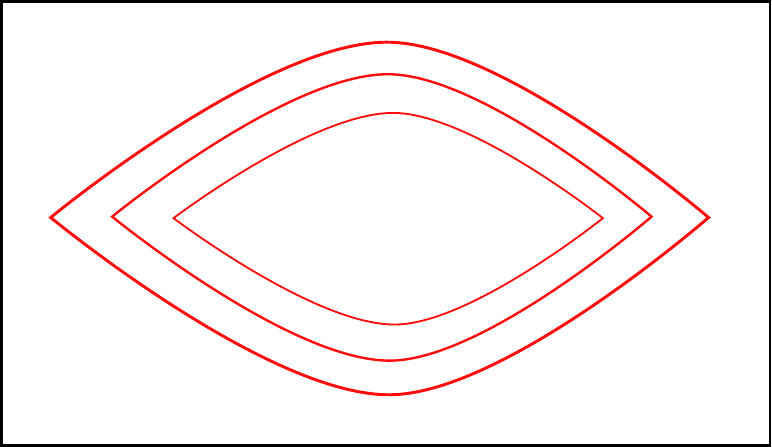}};
      \node at (-0.3,2){I};
    \node at (6,-0.2){t};
    \end{tikzpicture}
    \caption{An example of a Cerf graphic in nested eyes position.}
    \label{fig:example_Cerf_Graphic_Nested_Eyes}
  \end{minipage}
\end{figure}

We now seek to define $\Sigma$, following the discussion in \cite{hatcher_wagoner_1973}.  Start with a pseudo-isotopy $[F] \in \pi_0(\mathcal{P}^{\DIFF}(X,\partial X))$.  The pseudo-isotopy $F$ produces a one-parameter family of generalised Morse functions $g_t\colon X\times I \to I$, and, together with a choice of a one-parameter family of gradient-like vector fields $\eta_t$, this gives rise to a one-parameter family of handle decompositions for $X\times I$, starting and ending at the trivial handle decomposition.  We plot the critical values of $g_t$ in a Cerf graphic (see Figure \ref{fig:example_Cerf_Graphic}). We deform the one-parameter family of generalised Morse function and potentially obtain a different Cerf graphic.
Hatcher-Wagoner show that we can deform $(g_t,\eta_t)$ such that there are only critical points of index 2 and 3.  The obstruction $\Sigma$ will then be the obstruction to deform $(g_t,\eta_t)$ to a family of nested eyes (see Figure \ref{fig:example_Cerf_Graphic_Nested_Eyes}).

\begin{definition}
    Choose once and for all a basepoint $z\in X\times I$ and choose an ordered indexing $k$  of the critical points of index 2 and 3  which we denote by $b^k_t$ and $z^k_t$, respectively.  Choose orientations for the descending manifolds associated to $z^k_t$ and one-parameter families of basing arcs $\gamma_{z^k_t}$.  Now make the corresponding choices for $b^k_t$ such that the (equivariant) intersection matrix for the belt spheres for the index 2 handles and the attaching spheres for the index 3 handles is given by the identity matrix just after all of the births have occurred (we can assume that nothing else interesting occurs in $(g_t,\eta_t)$ until after all of the birth/deaths).  As we scan across the one-parameter family the intersection matrix changes as we pass handle slides.  Passing a 3/3 handle slide at time $t$ causes the matrix to change by multiplication on the left by $e_{i,j}^{\pm g}$ where $i$ and $j$ are determined by the ordering of the critical points and $g=\gamma_{z^i_t}\circ \varphi \circ \gamma_{z^j_t}^{-1}$, with $\varphi$ the arc that the handle slide occurs over.  Similarly, passing a 2/2 handle slide at time $t$ multiplies the matrix on the right by $e_{i,j}^{\pm g}$, where $i$ and $j$ are again determined by the ordering of the critical points and $g=\gamma_{b^i_t}\circ \varphi \circ \gamma_{b^j_t}^{-1}$, with $\varphi$ again the handle slide arc.

    Write $\prod_\ell e_{i_{\ell},j_{\ell}}^{\pm g_{\ell}}$ for the result of these multiplications.  At the end of this process (since the critical points must cancel) our matrix is of the form $PD$, where $P$ is a permutation matrix and $D$ is a diagonal matrix with elements $d_i=\pm g_i$ for some $g_i\in \pi_1(X)$.  Hatcher and Wagoner show that this matrix can be written in the form $PD = \prod_m e_{p_{m},q_{m}}^{\pm h_{m}}e_{q_{m}.p_{m}}^{\mp h^{-1}_{m}}e_{p_{m},q_{m}}^{\pm h_{m}}$. Define  \[\Pi(F):= \prod_\ell e_{i_{\ell},j_{\ell}}^{\pm g_{\ell}}\left( \prod_m e_{p_{m},q_{m}}^{\pm h_{m}}e_{q_{m},p_{m}}^{\mp h^{-1}_{m}}e_{p_{m},q_{m}}^{\pm h_{m}}\right)^{-1} \in K_2(\Z[\pi_1(X)]).\]
    We then define $\Sigma(F):= [\Pi(F)]\in \Wh_2(\pi_1(X))$.  Hatcher and Wagoner show that this does not depend on the various choices that were made.
\end{definition}

\subsection{Definition of $\Theta$}\label{sbs:def_theta}

We will briefly recall the definition of the obstruction $\Theta\colon \ker\Sigma \to \Wh_1(\pi_1(X);\Z/2 \times \pi_2(X))$.  Hatcher and Wagoner show that the kernel of $\Sigma$ is the `unicity of death' subgroup in high dimensions, which corresponds to pseudo-isotopies whose Cerf graphic can be deformed to that of a single `eye'.  As Quinn notes in \cite[Section 4.1]{quinn_1986}, the final step of the proof wherein the Cerf graphic is reduced from a family of `nested eyes' to a single eye does not apply in dimension four.  Hence, this $\Theta$ will be defined on pseudo-isotopies whose Cerf graphic can be deformed to consist only of a nested family of eyes (see \ref{fig:example_Cerf_Graphic_Nested_Eyes}). By a Cerf graphic with only "a nested family of eyes" we mean a Cerf graphic describing a one-parameter family of handle decompositions that has independent  births and deaths of handles.

\begin{definition}\label{def:wh_1}
    We follow Singh's exposition \cite[7.1]{singh2022pseudoisotopies} and the original exposition \cite[Part I, Chapter VII]{hatcher_wagoner_1973} to define $\Wh_1(\pi_1(X);\Z/2\times \pi_2(X))$.  Let $\pi$ be a group with an action on an abelian group $\Gamma$.  We define a ring $R:=\Gamma[\pi]\times \Z[\pi]$ and describe the ring structure by showing the multiplication of two generic elements:
    \[
    \left(\sum_i(\gamma_i+n_i)g_i\right)\left(\sum_j(\gamma'_j+n'_j)g'_j\right):=\sum_{i,j}(n_i(g_i\cdot \gamma'_j)+n'_j\gamma_i+n_in'_j)g_ig'_j.
    \]
    Above we have that $\gamma_i,\gamma'_j\in \Gamma$, $n_i,n'_j\in\Z$ and $g_i,g'_j\in\pi$.
    We then define \[\GL(\Gamma[\pi])=\ker\left(\GL(R)\to \GL(R/\Gamma[\pi])\right)\] and define \[K_1(\Gamma[\pi]):=\GL(\Gamma[\pi])/[\GL(R),\GL(\Gamma[\pi])].\]
    If we denote the trivial group as $\mathbb{1}$, there is then a natural map $\Gamma[\mathbb{1}]\to K_1(\Gamma[\pi])$ which sends $\gamma$ to the equivalence class of the $1\times 1$ matrix $(\gamma)$.  We then define $\Wh_1(\pi;\Gamma)$ to be the cokernel of this map.
    We remark that Hatcher showed that: \[\Wh_1(\pi;\Gamma)\cong \Gamma[\pi]/(\gamma g_1 - (g_2\cdot \gamma) (g_2 g_1 g_2^{-1}),\gamma')\] for all $\gamma,\gamma'\in \Gamma$ and $g,g'\in \pi$.
\end{definition}
We will apply this construction specifically when $\pi=\pi_1(X)$ and $\Gamma=\Z/2\times \pi_2(X)$ where the action of $\pi$ on $\Gamma$ is the product of the trivial action on the $\Z/2$-factor and the standard action on the $\pi_2(X)$-factor.

In the original definition, $\Theta$ lived in $\Wh_1(\pi_1(X);\Z/2\times \pi_2(X))$, but Igusa \cite{Igusa} showed that this only works if the first $k$-invariant $k_1(X)$ vanishes.  In general we have a map
\[
\chi \colon K_3(\Z[\pi_1(X)])\to \Wh_1(\pi_1(X);\Z/2\times \pi_2(X))
\]
where $K_3$ denotes the third algebraic $K$-group, and Igusa showed that $\Theta$ is only well-defined in $\left(\Wh_1(\pi_1(X);\Z/2\times \pi_2(X))\right)/(\Image \chi)$.  If the first $k$-invariant of $X$ vanishes then $\chi=0$.  In \Cref{sec:homeos} we will only consider examples where $k_1(X)=0$ and hence this subtlety will not matter for us.

The following definition will be rather terse.  For details, see \cite[Part I, Chapter VII]{hatcher_wagoner_1973} or \cite[Section 7.2]{singh2022pseudoisotopies}.

\begin{definition}\label{def:Theta}
    Let $F\colon X\times I\to X\times I$ such that $F$ lies in the unicity of death subgroup, i.e.\ its Cerf graphic can be deformed to consist only of nested eyes.  We now describe how to produce an element in $\Gamma[\pi]$ associated to $F$.  This will not be well-defined, but the equivalence class of this element in $\Wh_1(\Gamma;\pi)/\chi$ will be, and this will be the element $\Theta([F])$.  We will not touch on the well-definedness of $\Theta$ here and instead direct the reader to \cite[Part II]{hatcher_wagoner_1973}.

    Assume that we have already chosen a deformation of the one-parameter family associated to $F$ to a one-parameter family $(g_t,\eta_t)$ whose Cerf graphic consists of $n$ nested eyes, with all births and deaths at height $1/2$.  Assume that there exist times $0<t_0<t_1<t_2<t_3<1$ such that the following is true.  The $n$ births occur between time $t_0$ and $t_1$ and  for $t_0\leq t \leq t_1$ the 3-handles are still in cancelling position with their respective 2-handles; the $n$ deaths occur between time $t_2$ and $t_3$ and that for $t_2\leq t\leq t_3$ the 3-handles are in cancelling position with their respective 2-handles (see \cite[Section 2]{gabai_2022}).  Now we fix a diffeomorphism that identifies the middle level with \[
    \bigcup_{t\in[t_1,t_2]}g_t^{-1}(1/2)\cong (X\# n(S^2\times S^2))\times I.\]
    Let $B_i$ denote the trace of the belt sphere of the $i$-th 2-handle inside the middle level (note that via our identification this corresponds to some $(\{\pt\}\times S^2)\times I\subset X\# n(S^2\times S^2))\times I$), and let $A_i$ denote the trace of the attaching sphere of the $i$-th 3-handle inside the middle level (which looks non-standard).  The intersection $T_i:= A_i\cap B_i$ is a 1-manifold which consists of a single arc component and an unknown number of circle components denoted $C^j_i$.  For each circle component we will obtain an element $\gamma^j_i\in \pi_1(X)$ and an element $\alpha^j_i\in \Z/2\times \pi_2(X)$, and then we will define $\Theta((g_t,\eta_t))$ to be the sum $\sum_{i,j} \alpha^j_i \gamma^j_i\in (\Z/2\times \pi_2(X))[\pi_1(X)]$.
    We first describe how to obtain $\gamma^j_i$.  The circle $C^j_i$ comes with two basing arcs given by the attaching sphere and the belt sphere associated to it and together they form a loop $\gamma^j_i$ (since the circle $C^j_i$ is null-homotopic in $X\times I$, this is well-defined).  Now we describe how to obtain $\alpha^j_i$.  Since $A_i$ and $B_i$ are simply-connected, null-homotopies of $C_i^j$ in them determine two discs which glue together to give a $\pi_2$ element in $X$.  Similarly, $A_i$ and $B_i$ determine two different framings for $C$ in $A_i$ and the difference in these framings is a well-defined integer, whose value modulo two we record and combine with the $\pi_2$ element to produce $\alpha^j_i$.
    Finally, we define \[\Theta([F]):=[\Theta((g_t,\eta_t))]\in \Wh_1(\pi_1(X);\Z/2\times \pi_2(X))/\Image\chi.\]
\end{definition}

Some remarks on the above definition.

\begin{remark}
    \begin{enumerate}[(i)]
        \item Hatcher and Wagoner show that, in higher dimensions, circles $C^{j}_i$ and $C^{\ell}_i$ such that $\gamma^j_i=\gamma^{\ell}_i$ may be surgered to a single circle, hence geometrically justifying the fact that the invariant is valued in the group ring.  This surgery step does not hold in dimension four, but we nevertheless choose to lose information by passing to the group ring.
        \item In higher dimensions, there is only a modulo two choice for the framing of the circle in $A_i$.  In dimension four, there is an integers worth of framings and so we could choose to record the difference as an integer.  However, the well-definedness of $\Theta$, proven in \cite[Part II]{hatcher_wagoner_1973}, relies on suspending the given $(g_t,\eta_t)$, and so only the version of $\Theta$ defined using the $\Z/2$ framing is known to be well-defined, even in dimension four.   
    \end{enumerate} 
\end{remark}
    
\subsection{Suspension}
The aim of this subsection is to introduce the \emph{suspension} operation for pseudo-isotopies.
 We will directly follow Hatcher-Wagoner (see \cite[Part I, Chapter I, \S 5]{hatcher_wagoner_1973}). Note that the original definition in \cite{hatcher_wagoner_1973} was done for $\CAT=\DIFF$. Here we will work only in the topological category. As such, we will omit the superscript $\TOP$.

We need to make a small technical modification. For convenience in defining the suspension map we will first replace $\mathcal{P}(X,\partial X)$ with $\mathcal{P}'(X,\partial X)$, which is the subspace of pseudo-isotopies $F$ satisfying:
 \begin{itemize}
 \item $F\vert_{X \times [0,\epsilon] \cup \partial X \times I}= F \vert _{X \times 0}$,
 \item $F \vert_{X \times [1-\epsilon,1]}= F|_{X\times 1}$.
 \end{itemize}
 Here $\epsilon$ is a small positive number. In practice we are assuming that the pseudo-isotopy coincides with the identity around the boundary and a small subinterval $[0,\epsilon]$ and it also coincides with the top-end map around a small interval $[1-\epsilon, 1]$.

It is not hard to show these two spaces are homotopy equivalent and hence, in a slight abuse, we will keep writing $\mathcal{P}(X,\partial X)$ to ease the notation.
 
\begin{definition}[Suspension map] \label{definition:suspension_map}

 We define the \emph{suspension map} $S^+ \colon \mathcal{P}(X,\partial X) \to \mathcal{P}(X\times J,\partial (X \times J))$, where $J=[-1,1]$.

 Let $I=[0,1]$. Define $C \subset J \times I$ to be the set $C= J \times [0, 1 - \frac{\epsilon}{2}]$.
 Define an embedding $\varphi: C \to J \times I$ such that:
 \begin{itemize}
     \item $\varphi(1,t)=(1-t, 1)$
     \item $\varphi(-1,t)=(t-1,1)$
     \item $\varphi(s,t)= (s,t+s^2)$ for $|s| \leq \frac{\epsilon}{4}$
     \item for a fixed $t \in [0,1 - \frac{\epsilon}{2}]$ the map $[-1,1]\to [0,1]  \colon s \to \varphi (s,t)$ followed by the projection $J \times I \to I$ has a critical point only at $s=0$.
 \end{itemize}

  Using the embedding $\varphi$ we can introduce coordinates $u$ and $v$ on the image $\varphi(C)$. The lines $v=\constant$ are of the form $u=\varphi(s,t)$ with $t=\constant$, and the lines $u=\constant$ are of the form $v= \varphi(s,t)$ with $s=\constant$.

  \begin{figure}[h]
      \centering      \includegraphics{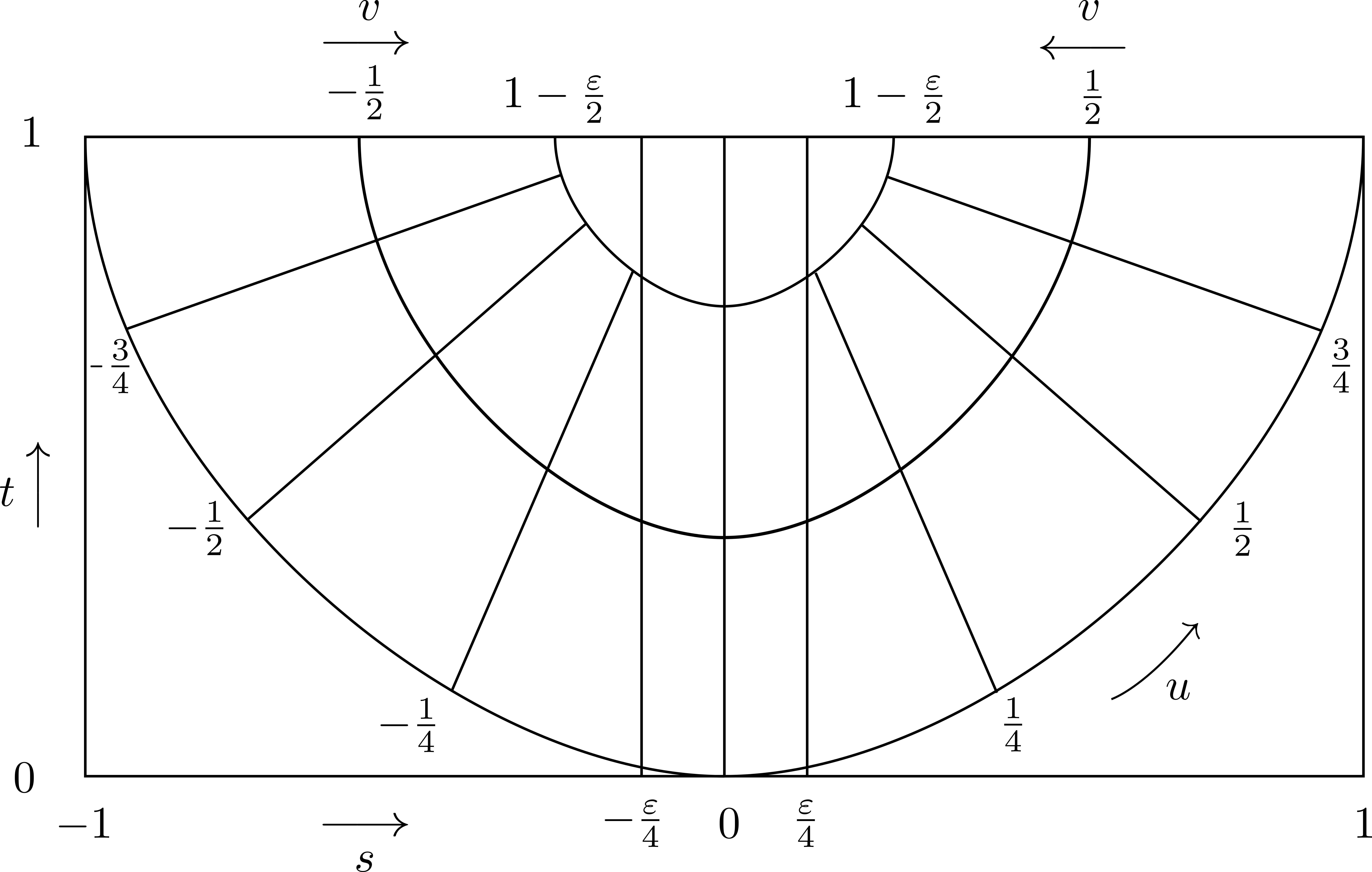}
      \caption{Pictorial description of the embedding of $\varphi\colon C\to J\times I$, taken from \cite[Part I, Chapter I, \S 5]{hatcher_wagoner_1973}.}
      \label{fig:suspension}
  \end{figure}

 Let $F \in \mathcal{P}(X,\partial X)$. We define $S^+(F) \in \mathcal{P}(X \times J, \partial (X \times J))$ to be:
 \begin{itemize}
     \item $F\vert_{X \times [0,1- \frac{\epsilon}{2}]}$ on each of the sections $ X \times \{u= \constant, 0 \leq v \leq 1- \frac{\epsilon}{2}\}$.
     \item $F\vert_{X \times 0}$ for each of the points $X \times \{(s,t)\}$ below the line $v=0$.
     \item $F|_{X \times 1}$ for each of the points $X \times \{(s,t)\}$ above the line $v=1-\frac{\epsilon}{2}$.

 \end{itemize}
\end{definition}
\begin{remark}\label{remark:well_behaved_suspension}
   The suspension map $S^+$ is defined so that it is ``well-behaved'' in terms of critical points of the pseudo-isotopy $F$. The square $J \times I$ is bent so that for each critical point of $F$, the critical point is preserved and no additional critical points are created. Moreover we stress that increasing the dimension in the trivial way (meaning taking the product with the identity on the square $J^2$) not only would introduce a line of critical points for every critical point of $F$ (this would be not ideal for the smooth functional approach used in \cite{hatcher_wagoner_1973}), but would actually not even produce a pseudo-isotopy, since the result would not be constant on the "vertical" boundary of $X \times J^2 \times I$.
   
   The behaviour with respect to critical points of course only makes sense in $\CAT = \DIFF$, and plays an important role in the definition of the Hatcher-Wagoner invariants in \cite{hatcher_wagoner_1973}.
   As for now, we only need this to be a well-defined map. The smooth properties will mostly on be useful for us for comparing with suspension in the smooth category, which we will do in \Cref{sec:properties} and \Cref{sec:duality}.
\end{remark}

 We now show that the map $S^+$ descends to a well-defined map on $\pi_0(\mathcal{P}(X, \partial X))$.
 \begin{lem}\label{lem:well-definedness_suspension}
     $\pi_0(\mathcal{P}(X, \partial X)) \xrightarrow{S^+_*}\pi_0(\mathcal{P}(X \times J, \partial( X\times J)))$ is a well-defined map.
 \end{lem}

 \begin{proof}
     Let $F$ and $G$ be two pseudo-isotopies in $\mathcal{P}(X, \partial X)$ such that $[F]=[G] \in \pi_0(\mathcal{P}(X, \partial X))$. Then there exists an isotopy $H_t : X \times I \to X \times I$ so that $H_0=F$ and $H_1=G$. We now build an isotopy $\widehat{H}_s$ that connects $S^+(F)$ and $S^+(G)$. To do so, we suspend the path $H$ by taking the suspension of the isotopy $H$.

This is defined on each time $s$ to be $\hat{H}_s: (X \times J) \times I \to (X \times J) \times I$ as:
\begin{itemize}
    \item $S^+(H_s)_0$ below the line $v=0$.
    \item $S^+(H_s)$ along radial lines $X \times \{u=\constant, 0\le v \le 1-\frac{\epsilon}{2}\} $.
    \item $S^+(H_s)_1$ above the "horizontal" line $v=1- \frac{\epsilon}{2}$.\qedhere
\end{itemize}

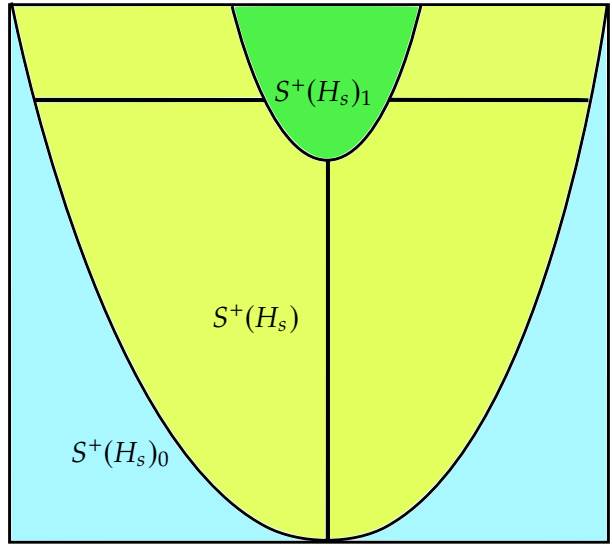
\begin{figure}[htb]
\centering
\begin{tikzpicture}
\node[anchor=south west,inner sep=0] at (0,0){\includegraphics[width=8cm]{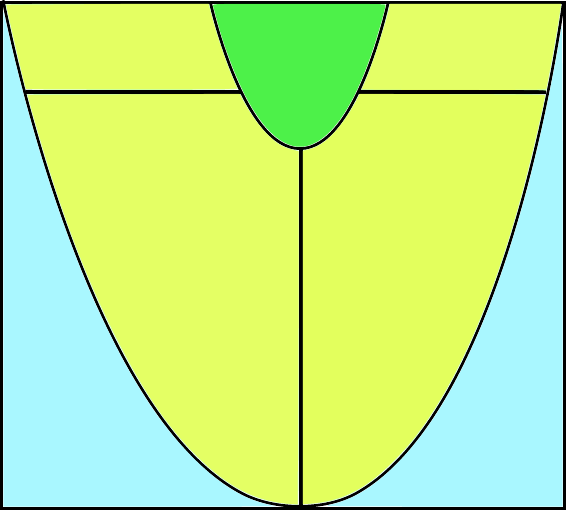}};
\node at (1.5,1.2) {$S^+(H_s)_0$};
\node at (4.2,6){$S^+(H_s)_1$};
\node at (3.3,3){$S^+(H_s)$};
\end{tikzpicture}
\caption{A time slice $s$ of the isotopy $\hat{H}_s$.}
\label{figure:stabilized_descends_to_well_defined}

\end{figure}

\end{proof}

 \begin{remark}
     In Section \ref{sec:homeos} we will work quite extensively with the suspension map. As in \cite{hatcher_wagoner_1973} we will slightly modify the embedding $\phi: C \to J \times I$ so that the "bent" rectangle does not touch $\{0\} \times I$. This is useful to work with smooth properties of the suspension.
 \end{remark}

 We will need the property that the suspension maps are compatible with the smooth Hatcher-Wagoner invariants.

 \begin{lem}\label{lrm:suspension_preserves_Sigma_Theta}
     The following diagrams, concerning $\Sigma$, $\Theta$ and the suspension $S^+$, commute.
     \[
     \begin{tikzcd}
         \pi_0\mathcal{P}^{\DIFF}(X,\partial X) \arrow[r, "\Sigma"] \arrow[d, "S^{+}"] & \Wh_2(\pi_1(X)) \arrow[d] \\
         \pi_0\mathcal{P}^{\DIFF}(X\times J,\partial (X \times J)) \arrow[r,"\Sigma"] & \Wh_2(\pi_1(X\times J))
     \end{tikzcd}
     \]
     and 
     \[
     \begin{tikzcd}
         \ker\Sigma \arrow[r, "\Theta"] \arrow[d, "S^{+}"] & \Wh_1(\pi_1(X);\Z/2\times \pi_2(X))/\chi \arrow[d] \\
         \ker\Sigma \arrow[r,"\Theta"] & \Wh_1(\pi_1(X\times J);\Z/2\times \pi_2(X\times J))/\chi
     \end{tikzcd}
     \]
     In both cases the rightmost vertical map is induced by the natural identification $\pi_k(X)\cong \pi_k(X\times J)$ for $k=1,2$.
 \end{lem}

 \begin{proof}
    This fact is used often and implicitly in \cite{hatcher_wagoner_1973}.  The key observation is in \cite[Part I, Chapter I, \S 5]{hatcher_wagoner_1973} where it is shown that one can define a compatible suspension of the one-parameter family $(g_t,\eta_t)$ associated to $F$, and that the suspension operation preserves all critical points, all handle slides and all intersections between the various ascending and descending manifolds.  This means that the data that goes into defining $\Sigma$ and $\Theta$ is entirely preserved by the suspension.
 \end{proof}

\subsection{Naturality of the smooth Hatcher-Wagoner invariants}
Since it will be useful later, we establish that the smooth Hatcher-Wagoner invariants satisfy naturality under the inclusion of codimension zero submanifolds.  This will be needed when showing the well-definedness of our topological invariants, and also in proving that they satisfy the analogous property.

\begin{lem}\label{lemma:commutativity_after_extending}
Let $X$ be an $n$-dimensional smooth compact manifold $(n \ge 4)$ and let $N$ be a compact codimension zero submanifold of $X$ . Let $i: N\to X$ be defined as the inclusion. We have a natural map from $\pi_0(\mathcal{P}^{\DIFF}(N,\partial N))$ to $\pi_0(\mathcal{P}^{\DIFF}(X,\partial X))$ given by extending via the identity over $X \setminus N$. Let $F$ be a pseudo-isotopy of $N$ and $F\#\Id$ the extension over $X$. Then we have $\Sigma([F\#\Id]) 
=i_*(\Sigma([F]))$ and, if $F\in \ker \Sigma$, then $\Theta([F\#\Id]) 
=i_*(\Theta([F]))$.
\end{lem}
\begin{proof}
    We first deal with $\Sigma$, so we start by showing that the following diagram commutes.
    \[\begin{tikzcd}
	{\pi_0(\mathcal{P}^{\DIFF}~(X,\partial X))} && {\Wh_2(\pi_1(X))} \\
	{\pi_0(\mathcal{P}^{\DIFF}(N,\partial N))} && {\Wh_2(\pi_1(N))}
	\arrow["\Sigma"{description}, from=1-1, to=1-3]
	\arrow[from=2-1, to=1-1]
	\arrow["\Sigma"{description}, from=2-1, to=2-3]
	\arrow[ from=2-3, to=1-3]
\end{tikzcd}\]
Let $(g^N_t,\eta^N_t)$ be the one-parameter family for $F$, which we already assume to have been deformed such that it is in the position for defining $\Sigma$ (see \Cref{sbs:def_sigma}).  The $g^N_t\colon N\times I \to I$ can be extended to a one-parameter family of generalised Morse functions $g_t\colon X\times I\to I$ via extending by the projection map.  In particular, since we extended by the projection map, they have exactly the same critical points, critical values, births and deaths as $g^N_t$.  We can also extend $F$ via the identity to a pseudo-isotopy $\widetilde{F}$, and it is clear that $g_t$ is a one-parameter family for $\widetilde{F}$ (in particular, $g_1=\pr_2\circ \widetilde{F}$).  Similarly, the one-parameter family of gradient-like fields $\eta^N_t$ can be extended to $\eta_t$ such that the one-parameter family of handle decompositions induced by $(g_t,\eta_t)$ is the trivial one-parameter family outside of $N$.

Let \[\prod_\ell e_{i_{\ell},j_{\ell}}^{\pm g_{\ell}}\] be the matrix corresponding to the handle slides for $(g^N_t,\eta^N_t)$, with $g_\ell\in \pi_1(N)$.  Then from the above observation it follows that the matrix corresponding to the handle slides for $(g_t,\eta_t)$ is \[\prod_\ell e_{i_{\ell},j_{\ell}}^{\pm i_*(g_{\ell})}.\]  From this and the definition of $\Sigma$ we see that $i_*(\Sigma([F])=\Sigma([\widetilde{F}])$, which proves that the stated square commutes.

We now deal with $\Theta$ by showing the following diagram commutes.

\[\begin{tikzcd}
	{\ker \left(\Sigma\colon \pi_0(\mathcal{P}^{\DIFF}(X,\partial X))\to \Wh_2(\pi_1(X))\right) } && {\Wh_1(\pi_2(X);\Z/2\times \pi_1(X))/\chi} \\
	{\ker \left(\Sigma\colon \pi_0(\mathcal{P}^{\DIFF}(N,\partial N))\to \Wh_2(\pi_1(N))\right) } && {\Wh_1(\pi_2(N);\Z/2\times \pi_1(N))/\chi}
	\arrow["\Theta"{description}, from=1-1, to=1-3]
	\arrow[from=2-1, to=1-1]
	\arrow["\Theta"{description}, from=2-1, to=2-3]
	\arrow[ from=2-3, to=1-3]
\end{tikzcd}\]

Again let $(g^N_t,\eta^N_t)$ be the one-parameter family for $F$, we already assume we have deformed such that it is in the position for defining $\Theta$ (see \Cref{sbs:def_theta}).  Again, as we did when we considered $\Sigma$, we may extend $(g^N_t,\eta^N_t)$ to a one-parameter family $(g_t,\eta_t)$ which is associated to the pseudo-isotopy $\widetilde{F}\colon X\times I\to X\times I$ which is obtained by extending $F$ via the identity.  

We now use the notation in \Cref{sbs:def_theta}.  Let $A^N_j$ and $B^N_j$ be the traces of the attaching spheres and the belt spheres in the middle level for $(g^N_t,\eta^N_t)$ and let $T^N_j$ denote their intersection.  By the way we defined the extension $(g_t,\eta_t)$, the induced one-parameter family of handle decompositions for $X\times I$ is constant and trivial away from $N$.  This means that the traces of the attaching spheres and the belt spheres in the middle level for $(g_t,\eta_t)$ are simply $i(A^N_i)$ and $i(B^N_i)$, and their intersection is $i(T^N_j)$.  It follows that the circles used for computing $\Theta(\widetilde{F})$ are produced by including the circles used to compute $\Theta(F)$ into $X\times I$.  By choosing the basing arcs for the attaching spheres and belt spheres in the middle level for $(g_t,\eta_t)$ to be compatible with the ones used for $(G^N_t,\eta_t)$ it can be seen that the $\pi_2$ element, the $\pi_1$ element and the framing difference associated to a given circle component in $i(T^N_j)$ are obtained from the corresponding elements for the corresponding circle component in $T^N_j$ by the inclusion induced map (or, in the case of the framing difference, this is simply the same framing difference).  This proves that the above diagram involving $\Theta$ commutes, finishing the proof.
\end{proof}

\section{Defining the topological obstructions}\label{sec:def_top_obstructions}
Using the suspension map $S^+$, we are able to increase the dimension of the topological manifold we are working with. This is crucial in our approach, since we need to apply the machinery from  Burghelea-Lashof-Rothenberg \cite{burghelea_lashof_1974,BergheleaLashofBook}, which is only available in dimensions $\geq 6$ (c.f.\ \Cref{theorem:berghelea_lashof} and \Cref{thm:top_version_3.1}). Hence we need to suspend our pseudo-isotopies twice. 
    
\subsection{Smoothing pseudo-isotopies}
Our goal is to manipulate our topological pseudo-isotopy until we reach a smooth pseudo-isotopy on which we can evaluate the Hatcher-Wagoner invariants. In order to do so, we need to find a way to pass from $\CAT=\TOP$ to $\CAT=\DIFF$. In this process we also have to make sure not to lose any relevant information about the connectivity of the space $\mathcal{P}^{\TOP}(X,\partial X)$.  We start by suspending our pseudo-isotopy twice, and will denote the suspended $X$ by $X\times J^2$.  We need the following definition.

\begin{definition}\label{def:handle_skeleton}
    Let $M$ be a $n$-manifold with a fixed handle decomposition.  Then we denote the $k$-handle-skeleton of $M$ by $M^{(k)}$. This is the union of all handles in the given decomposition of index $\leq k$.
\end{definition}

First note that $X\times J^2$ is a $6$-manifold and hence it has a handle decomposition by the work of \cite[Essay III, Theorem 2.1]{kirby_siebenmann_1977}. If $\partial X \neq \emptyset$, we always take the handle decomposition to be relative to $\partial X \times J^2$. Fix one, and then we can consider the $3$-handle-skeleton $(X\times J^2)^{(3)}$. Let $N$ be a neighbourhood of  $(X\times J^2)^{(3)}$ inside $(X \times J^2)$.

We now need the following lemma.

\begin{lem}\label{lem:3_skeleton_isomorphism}
    The inclusion map $N\to X\times J^2$ induces an isomorphism  \[ \pi_0\left(\mathcal{P}^{\TOP}(N, \partial N)\right) \xrightarrow{\cong} \pi_0\left(\mathcal{P}^{\TOP}(X\times J^2, \partial (X \times J^2) )\right).\]
\end{lem}
To prove the result we start from the work of \cite{BergheleaLashofBook}. We adapt for this proof part of their notation. 
Let $W \subset \mathring{V}$ be a compact submanifold with $\dim W= \dim V$. Then we have an injection $\gamma:\mathcal{P}(W,\partial W) \to \mathcal{P}(V,\partial V)$ that takes a pseudo-isotopy $F$ of $W$ and extends it via the identity to the rest of $V$. We have the following result. Note that this is stated in terms of $\CAT= \DIFF$ or $\CAT=\PL$ in \cite{BergheleaLashofBook}.

\begin{thm}[{\cite[Theorem $3.1'$]{BergheleaLashofBook}}]\label{theorem:berghelea_lashof}
Let $V^n$, $n \ge 5$, and $W^n \subset \mathring{ V}^n$. Assume:
\begin{enumerate}
    \item $\pi_i( \partial W) \cong \pi_i(W)$, $i=0,1$,
    \item $\pi_i(W,V)=0$ for $i \le r$ \text{ where }$r \le n-4$ $($ or $r \le n-3$ if $n >5$ and $\pi_i(\partial V) \cong \pi_i(V)$, $i=0,1$ $)$,
    \item $W$ is $k$-connected, for some  $0\le k \le r$.
    
\end{enumerate}
Then $\pi_j(\gamma)=0$ for:
\begin{itemize}
    \item $j \le \inf(2r-3,r+k-2)$ if $\CAT=\DIFF$,
    \item $j \le \inf(2r-3,r+k-2,r+2)$ if $\CAT=\PL$.
\end{itemize}
\end{thm}
For the purposes of this paper and specifically for the proof of Lemma \ref{lem:3_skeleton_isomorphism} we need a $\CAT=\TOP$ version of Theorem \ref{theorem:berghelea_lashof}.

This result in fact holds in the $\TOP$ category as well. In fact, this is implicitly stated in \cite{BergheleaLashofBook} but not exactly spelled out. We describe how to obtain the topological version.

Theorem \ref{theorem:berghelea_lashof} ($\CAT=\PL$ version) follows from \cite[Theorem 3.1, $\CAT=\PL$ version]{BergheleaLashofBook}.
This result is proved using \cite[Lemma $b)$]{BergheleaLashofBook}. This lemma is in turn a consequence of \emph{Morlet's Disjunction Lemma} (due to Morlet, proven also in \cite{BergheleaLashofBook}). Using Pedersen's work -- which appears briefly as an appendix in \cite{BergheleaLashofBook} -- on \emph{concordance straightening} and a topological notion of transversality \footnote{Pedersen's approach uses a local version of $PL$ transversality. At the time topological transversality from \cite{kirby_siebenmann_1977} was not yet available---we believe it can now  be used to simplify Pedersen's approach.} one can adapt the proof of the Disjunction Lemma appearing in \cite{BergheleaLashofBook} to $\CAT=\TOP$. Hence all the results directly depending on Morlet's result presented in \cite{BergheleaLashofBook} -- more specifically, we focus on \Cref{theorem:berghelea_lashof} -- can be restated by changing $\PL$ to $\TOP$ (and hence all the relevant categorical objects have to be turned into the corresponding $\TOP$ version). 
Note that we also have to assume that the ambient manifold has to have dimension at least 6 -- this is necessary to use Pedersen's result of concordance implies isotopy \cite{pedersen_1977}.

We can thus obtain the analogue of Theorem \ref{theorem:berghelea_lashof}.
\begin{thm}\label{thm:top_version_3.1}
  Let $V^n$, $n \ge 6$, and $W^n \subset \mathring{V}^n$ be an embedded handlebody. Assume:
\begin{enumerate}
     \item $\pi_i( \partial W) \cong \pi_i(W)$, $i=0,1$,
    \item $\pi_i(W,V)=0$ for $i \le r$ \text{ where }$r \le n-4$ $($ or $r \le n-3$ if $n >5$ and $\pi_i(\partial V) \cong \pi_i(V)$, $i=0,1$ $)$,
    \item $W$ is $k$-connected, for some  $0\le k \le r$.
    
\end{enumerate}
Then $\pi_j(\gamma)=0$ for
\begin{equation}\label{equation:range}
    j \le \inf(2r-3,r+k-2,r+2).
\end{equation}
\end{thm}

\begin{remark}
    Note that the first two conditions in Theorem \ref{thm:top_version_3.1} are satisfied when \[V= W \cup (n-a)-\text{handles},\] where $n-a>r$. 
\end{remark}
Now we can proceed to prove Lemma \ref{lem:3_skeleton_isomorphism}.
\begin{proof}[Proof of \cref{lem:3_skeleton_isomorphism}]

Let $V$ be our twice suspended manifold $X\times J^2$ and $W$ be the neighbourhood $N$ of $(X\times J^2)^{(3)}$. Our goal is to show that $\pi_1(\gamma)=0$, where $\gamma$ is the inclusion map: $$\gamma:\mathcal{P}^{\TOP}(W) \to \mathcal{P}^{\TOP}(V).$$

We apply Theorem \ref{thm:top_version_3.1} to $X\times J^2$ and $N$. Set $r=3$. The conditions $1)$ and $2)$ in  \cref{thm:top_version_3.1} are satisfied since $X\times J^2$ is built from $N$ by attaching handles of index $>r=3$. In general we do not have more information regarding the connectivity of $N$, so $k$ will be 0. We thus obtain that $\pi_j(\gamma)=0$ for:
\begin{equation*}
     j \le \inf(3,1,5)=1
\end{equation*}
i.e. $\pi_0(\gamma)=0$ and $\pi_1(\gamma)=0$.

In particular, $\gamma$ is 1-connected and using the long exact sequence in (relative) homotopy groups we obtain that $\pi_0(\mathcal{P}(X\times J^2,\partial (X \times J^2))) \cong \pi_0(\mathcal{P}(N, \partial N)) $, concluding our proof.
\end{proof}

\begin{notation}\label{not:inclusion_inverse}
    We will denote the inverse of the inclusion-induced map from \Cref{lem:3_skeleton_isomorphism} by \[ \mathfrak{i}\colon\pi_0\left(\mathcal{P}^{\TOP}(X\times J^2, \partial (X \times J^2) )\right) \xrightarrow{\cong} \pi_0\left(\mathcal{P}^{\TOP}(N, \partial N)\right).\]
\end{notation}

\begin{remark}
     Note that the above result requires $r=3$. If we had chosen a lower index we would have obtained less information on the connectivity of $\gamma$.  In particular, if we were to instead take the 2-handle skeleton and a neighbourhood $N'$ of it, \Cref{thm:top_version_3.1} would only give us that $ \pi_0(\mathcal{P}(N', \partial N'))  \to \pi_0(\mathcal{P}(X \times J^2, \partial (X \times J^2)))$ is onto.  This is not surprising, however, as the presence of $\pi_2$ in the definition of $\Theta$ implies that relevant data is encoded in the 3-handles.

     We will, however, make use of the surjectivity of  $\pi_0(\mathcal{P}^{\TOP}(N',\partial N'))  \to \pi_0(\mathcal{P}^{\TOP}(X \times J^2, \partial (X\times J^2)))$ in \Cref{lemma:independence_on_smoothing}.
 \end{remark}

With Lemma \ref{lem:3_skeleton_isomorphism} in our hands we restrict to pseudo-isotopies of a neighbourhood $N$ of the 3-handle-skeleton of $X \times J^2$. The next step is to pass to $\CAT=\DIFF$. In order to do so, we first need to smooth the 3-handle-skeleton of $X \times J^2$. 

 \begin{lem}\label{lem:3_skeleton_smoothing}
     The 3-handle-skeleton $(X\times J^2)^{(3)}$ can be given a smooth structure $\mathcal{S}$.  The 2-handle-skeleton $(X\times J^2)^{(2)}$ can be given a smooth structure which is unique up to isotopy.
 \end{lem}
 
\begin{proof} 
The existence statements follow from simple applications of the smoothing theory of Kirby-Siebenmann, and we prove that first.  Note that it suffices to show that the 3-handle-skeleton is smoothable, since the 2-handle-skeleton inherits a smooth structure from it.

Since our handle decomposition is relative to $\partial X \times J^2$, we always assume that we already have smoothed such portion of the manifold -- there is a unique way to do so \cite{kirby_siebenmann_1977}. 

We will prove this by finite induction on the index of handles. Let $\mathcal{C}$ be the collection of the handles we have already smoothed. At the beginning, $\mathcal{C}=\emptyset$. 

    Start from the collection $C_0$ of $0$-handles. This clearly has a smooth structure obtained by pulling back the standard structure on $\bigsqcup_i D^6$ via the corresponding homeomorphism -- one for each handle; all of them are disjoint. Hence, we can say that $\mathcal{C}=C_0$.
    
    Now assume we have extended this smooth structure up until the $(k-1)$-handles, i.e.\ $\mathcal{C}=C_{k-1}$.  We begin by extending the smooth structure onto a small neighbourhood of the $k$-handles attaching regions, canonically smoothing corners as necessary.  Denote the $k$-handles as $\mathcal{H}_i,  i \in I$ and denote the attaching region of these handles by $\partial_A \mathcal{H}_i$. Using smoothing theory \cite[Essay IV, Theorem 10.1]{kirby_siebenmann_1977}, we know that the first obstruction to extending this smooth structure onto the rest of the handles lies in $\bigoplus_i H^4 (\mathcal{H}_i, \partial_A \mathcal{H}_i; \Z/2)$.  A simple computation gives that this cohomology group vanishes for handles of index three or lower; the further obstructions lie in even higher degree cohomology groups, and hence also vanish. 
    Hence we obtain obtain a smooth structure on $C_k$ as well. Applying this to $k=1,2,3$ we have $\mathcal{C}=C_0 \cup C_1 \cup C_2 \cup C_3$.

    Now we deal with the uniqueness statement for the 2-handle-skeleton.  This follows from the computation the homotopy groups of $\TOP/O$ \cite[Essay IV, Chapter 10]{kirby_siebenmann_1977}, which means that the only obstruction for the smooth structures to be concordant relative to the chosen structure on $(X \times J^2)^{(2)}$ lies in $ \bigoplus _iH^3(\mathcal{H}_i, \partial_A \mathcal{H}_i; \Z/2)$, and hence vanishes for handles of index two or lower (higher degree obstructions also appear but these similarly vanish in our case).  Applying the concordance implies isotopy theorem \cite[Essay I, Theorem 4.1]{kirby_siebenmann_1977} finishes the proof.
\end{proof}

\begin{remark}
    The uniqueness obstructions lying in $\bigoplus_iH^3(\mathcal{H}_i, \partial_A \mathcal{H}_i; \Z/2)$, mentioned at the end of the proof of \Cref{lem:3_skeleton_smoothing} means that the 3-handle-skeleton does not in general admit a unique smooth structure. In fact, each different element of $\bigoplus_iH^3(\mathcal{H}_i, \partial_A \mathcal{H}_i; \Z/2)$ corresponds to a (potentially) distinct (up to isotopy) smooth structure on $(X \times J^2)^{(3)}$, and we will have to take this non-uniqueness into account later in \Cref{lemma:independence_on_smoothing}.
\end{remark}

Once $N$ is smoothed, we can pass to $\pi_0\left(\mathcal{P}^{\DIFF}(N_{\mathcal{S}})\right)$. We have to make sure that this step does not cause loss of information with respect to the path-connectedness in the $\TOP$ category.
In order to do so, we appeal to Burghelea-Lashof \cite[Theorem 6.1, Theorem 6.2]{burghelea_lashof_1974} which gives:
\begin{equation*}
 \pi_0\left(\mathcal{P}^{\DIFF}(N_{\mathcal{S}},\partial N_{\mathcal{S}})\right) \xrightarrow{\cong}\pi_0\left(\mathcal{P}^{\TOP}(N,\partial N)\right) 
\end{equation*} via the forgetful map.  

\begin{notation}\label{not:smoothing_inverse}
    We denote the inverse of this isomorphism by $\mathfrak{f}_{\mathcal{S}}$.
\end{notation}

\subsection{Defining the topological obstruction}

We briefly recall what we have obtained so far. Starting from the manifold $X^4$, we twice-suspended it to $X \times J^2$. Using the $\TOP$ Theorem \ref{thm:top_version_3.1} we showed that the space of path components of $\mathcal{P}^{\TOP} (X\times J^2,\partial (X \times J^2))$ depends only on a neighbourhood $N$ of the 3-handle skeleton. We then smoothed $N$ and used the work of \cite{burghelea_lashof_1974} to transition into the $\DIFF$ category, i.e.\ smoothed the pseudo-isotopy, without losing information with respect to the connected-components of the pseudo-isotopy spaces.

The last step is to use Hatcher-Wagoner as in \cite{hatcher_wagoner_1973} to define the obstruction maps $\Sigma$ and $\Theta$ for $\CAT=\DIFF$. These maps have image in $\Wh_2\left(\pi_1(X \times J^2)\right)$ and \\$\Wh_1\left(\pi_1(X \times J^2);\Z/2 \times \pi_2(X\times J^2)\right)/\chi$.

Showing that the choice of smoothing of $N$ does not affect the result of the composition map that is involved in our definition, which is a key step to have an actual unique definition of the invariant, will be the content of \cref{lemma:independence_on_smoothing}. We postpone the proof for now.

\begin{definition}[Topological Hatcher-Wagoner invariants]\label{def:top_hw}
Let $X$ be a compact, topological 4-manifold.  Let $S^2:= S^+\circ S^+ \colon \pi_0\mathcal{P}^{\TOP}(X,\partial X)\to \pi_0\mathcal{P}^{\TOP}(X\times J^2,\partial (X\times J^2))$ be the two-fold suspension homomorphism and let $\mathcal{S}$ denote a smooth structure on $N$, the neighbourhood of the 3-handle-skeleton of $X\times J^2$.  Then we define:
\[
\Sigma^{\TOP}:= \Sigma \circ \mathfrak{f}_{\mathcal{S}} \circ \mathfrak{i} \circ S^2\colon \pi_0\mathcal{P}^{\TOP}(X,\partial X) \to \Wh_2(\pi_1(X)) 
\]
where $\mathfrak{i}$ is from \Cref{not:inclusion_inverse}, $\mathfrak{f}_{\mathcal{S}}$ is from \Cref{not:smoothing_inverse} and $\Sigma$ is the smooth, first Hatcher-Wagoner invariant from \Cref{sbs:def_sigma}.

Similarly, we define
\[
\Theta^{\TOP}:= \Theta \circ \mathfrak{f}_{\mathcal{S}} \circ \mathfrak{i} \circ S^2\colon \ker\Sigma^{\TOP} \to \Wh_1(\pi_1(X);\Z/2 \times \pi_2(X))/\chi,
\]
where $\Theta$ is the smooth, second Hatcher-Wagoner invariant from \Cref{sbs:def_theta}.
\end{definition}

\begin{lemma}\label{lem:homomorphism}
    The maps $\Sigma^{\TOP}$ and $\Theta^{\TOP}$ are group homomorphisms.
\end{lemma}
\begin{proof}
    By definition, both $\Sigma^{\TOP}$ and $\Theta^{\TOP}$ are compositions of homomorphisms.
\end{proof}
\Cref{def:top_hw} a priori depends on the choice of smooth structure $\mathcal{S}$ on $N$, but this is not reflected in the notation.  We will justify this notation in \Cref{lemma:independence_on_smoothing} by showing that our invariants do not depend on the choice of $\mathcal{S}$.

\begin{remark}
    Note that there could be elements $[F]$ in $\ker \Sigma ^{\TOP}$ that vanish because they lie in $\ker S^2$. For such elements, $\Theta^{\TOP}$ also automatically vanishes. This however does not imply that $[F]$ is isotopic to the identity (e.g. Budney-Gabai's examples \cite{budney_gabai_2023}). It is also important to notice that there could be elements for which $\Sigma^{\TOP}$ vanishes  and whose $\Theta$ invariant is non-trivial. For such elements, it is evident by the definitions that they cannot lie in $\ker S^2$. 
\end{remark}

\begin{thm}
    If either $\Sigma^{\TOP}([F])\neq 0$ and $\Theta^{\TOP}([F]) \neq0$ then $F$ is not topologically isotopic to an isotopy.
\end{thm}
\begin{proof}
   Following \cref{lem:homomorphism}, the maps $\Sigma ^{\TOP}$ and $\Theta^{\TOP}$ are homomorphisms. Hence if $[F]=[\Id]$ then  $\Sigma ^{\TOP}([F])=0$ and $\Theta^{\TOP}([F])=0$. 
\end{proof}

\section{Properties of the topological Hatcher-Wagoner invariants}\label{sec:properties}

We start this section by showing that the choice of smoothing of $N$ does not affect the result of the composition map in \Cref{def:top_hw}. 
We then proceed to prove the main properties that our topological invariants satisfy. Namely, they possess 
naturality under the inclusion of certain codimension zero submanifolds (see Proposition \ref{lemma:proposition_naturality_topological}) and that the topological and smooth Hatcher-Wagoner invariants coincide for smooth pseudo-isotopies, which was stated in the last part of \Cref{thm:main} in the introduction.

\begin{lem}\label{lemma:independence_on_smoothing}
    The maps $\Sigma^{\TOP}$ and $\Theta^{\TOP}$ are independent of the choice of smooth structure $\mathcal{S}$ on $N$.
\end{lem}

\begin{proof}

We only prove this for $\Sigma^{\TOP}$. The proof for $\Theta^{\TOP}$ is analogous.
The general idea is that we will combine \Cref{def:top_hw} with an analogous `definition' using the 2-handle-skeleton instead.  As in \Cref{def:top_hw}, choose a smooth structure $\mathcal{S}$ on $N$, the neighbourhood of the 3-handle-skeleton, and denote a neighbourhood of the 2-handle-skeleton by $N'$.  We now draw the relevant diagram.  

\[\begin{tikzcd}
	& {\pi_0 \left( \mathcal{P}^{\TOP}(X,\partial X)\right)} \\
	& {\pi_0\left(\mathcal{P}^{\TOP}(X \times J^2, \partial (X \times J^2))\right)} \\
	{\pi_0\left(\mathcal{P}^{\TOP}(N',\partial N') \right)} && {\pi_0\left(\mathcal{P}^{\TOP}(N,\partial N) \right)} \\
	{\pi_0\left(\mathcal{P}^{\DIFF}(N',\partial N') \right)} && {\pi_0\left(\mathcal{P}^{\DIFF}(N_{\mathcal{S}},\partial N_{\mathcal{S}}) \right)} \\
    \Wh_2(\pi_1(N')) && \Wh_2(\pi_1(N)) \\
    && \Wh_2(\pi_1(X))
	\arrow["{S^2}"{pos=0.3}, from=1-2, to=2-2]
	\arrow["a", from=3-1, to=2-2, two heads]
    \arrow["\mathfrak{i}'"', from=2-2, to=3-1, bend right, start anchor={[xshift=0,yshift=0]west}, end anchor={north}, dotted]
    \arrow["\mathfrak{i}", "\cong"', from=2-2, to=3-3, bend left, start anchor={[xshift=0,yshift=0]east}, end anchor={north}]
    \arrow["\cong", "b"', from=3-3, to=2-2]
	\arrow["c"{description}, two heads, from=3-1, to=3-3]
    \arrow["\cong", from=4-3, to=3-3, shift left=5ex]
    \arrow["\cong", "\mathfrak{f}"',from=3-1, to=4-1, shift right=5ex]
	\arrow["d"{description}, two heads, from=4-1, to=4-3]
	\arrow["\cong", "e"', from=4-1, to=3-1, shift right=5ex]
    \arrow["\mathfrak{f}_{\mathcal{S}}","\cong"', from=3-3, to=4-3, shift left=5ex]
    \arrow["\cong"', from=5-1, to=5-3]
    \arrow["\Sigma"', from=4-1, to=5-1]
    \arrow["\Sigma", from=4-3, to=5-3]
    \arrow["\cong"', from=5-3, to=6-3]
\end{tikzcd}\]

We describe the maps in the diagram.  The maps $a, b, c$ and $d$ are all inclusion induced maps.  By appealing to \cite[Theorem 6.1, Theorem 6.2]{burghelea_lashof_1974}, we have the existence of the inverse map $\mathfrak{f}$ to $e$, which does not depend on a choice of smooth structure, since, by \Cref{lem:3_skeleton_smoothing}, $N'$ is uniquely smoothable.  As in \Cref{def:top_hw}, we have the existence of $\mathfrak{i}$, the inverse to $b$.  Similarly, $\mathfrak{f}_{\mathcal{S}}$ is the inverse to the forgetful map, as in \Cref{def:top_hw}.  In the above diagram we can see the map $\Sigma^{\TOP}$ for the smooth structure $\mathcal{S}$: map down the right hand side of the diagram.  We now aim to show that the diagram commutes.

We first consider the top section of the diagram, then the middle rectangle and eventually the bottom rectangle.

\subsubsection{The upper triangle}\label{sbsbs:upper}
\[\begin{tikzcd}
	& {\pi_0\left(\mathcal{P}^{\TOP}(X \times J^2, \partial (X \times J^2))\right)} \\
	{\pi_0\left(\mathcal{P}^{\TOP}(N',\partial N') \right)} && {\pi_0\left(\mathcal{P}^{\TOP}(N,\partial N) \right)}
	\arrow["a"{pos=0.4}, from=2-1, to=1-2]
	\arrow["c"{description}, two heads, from=2-1, to=2-3]
	\arrow["b"'{pos=0.4}, from=2-3, to=1-2]
\end{tikzcd}\]

Start with a class $[F] \in \pi_0\left(\mathcal{P}^{\TOP}(N',\partial N') \right)$. 
The map $a$ is given by extending via the identity outside the neighbourhood $N'$ of the 2-handle-skeleton. We have: \[a([F])= [\hat{F}],\] where $\hat{F}|_{N'}=F$ and $\hat{F}|_{X\times J^2 \setminus N'}=\Id$. Surjectivity of this map is guaranteed by Theorem \ref{thm:top_version_3.1}. 
The map $b$ is defined in a similar way---the only difference is that we are extending over the complement of a neighbourhood of the $3$-skeleton. Start with $[G] \in \pi_0\left(\mathcal{P}^{\TOP}(N,\partial N) \right)$. This is sent to: \[b([G])= [\overline{G}],\] where $\overline{G}|_{N}=G$ and $\overline{G}|_{X\times J^2 \setminus N}=\Id$. The map
$b$ is an isomorphism thanks to 
\cref{lem:3_skeleton_isomorphism}. 

The map $c$ is defined by extending via the identity from $N'$ to $N$. In particular if $[F]$ is in $\pi_0\left(\mathcal{P}^{\TOP}(N',\partial N') \right)$, then: \[ c([F])=[G],\] where $G|_{N'}=F$ and $G|_{N\setminus N'}=\Id$. This map is surjective by an application of Theorem \ref{thm:top_version_3.1}. 

It is immediate from the definitions of the maps to show that $a= b \circ c$: \[ b \circ c ([F])= b ([G]) = [\overline{G}],\] where $\overline{G}|_{N'}=F$ and $\overline{G}|_{X \times J^2\setminus N'}=\Id$, which is exactly the same as $[\widehat{F}]=a([F])$.

\subsubsection{The middle rectangle}\label{sbsbs:middle_rectangle}

We now prove commutativity of the middle rectangle. 
\[\begin{tikzcd}
	{\pi_0\left(\mathcal{P}^{\TOP}(N',\partial N') \right)} && {\pi_0\left(\mathcal{P}^{\TOP}(N,\partial N) \right)} \\
	{\pi_0\left(\mathcal{P}^{\DIFF}(N',\partial N') \right)} && {\pi_0\left(\mathcal{P}^{\DIFF}(N_{\mathcal{S}},\partial N_{\mathcal{S}}) \right)}
	\arrow["c"{description}, two heads, from=1-1, to=1-3]
	\arrow["\cong", from=2-1, to=1-1]
	\arrow["d"{description}, two heads, from=2-1, to=2-3]
	\arrow["\cong"', from=2-3, to=1-3]
\end{tikzcd}\]

A class $[F] \in\pi_0\left(\mathcal{P}^{\DIFF}(N',\partial N') \right)$ (resp. in $\pi_0\left(\mathcal{P}^{\DIFF}(N_{\mathcal{S}},\partial N_{\mathcal{S}}) \right)$) is sent by forgetting the smooth structure to $[F] \in \pi_0\left(\mathcal{P}^{\TOP}(N',\partial N') \right)$ (resp. in $\pi_0\left(\mathcal{P}^{\TOP}(N_{\mathcal{S}},\partial N_{\mathcal{S}}) \right)$). 
The map $c$ was previously defined in \cref{sbsbs:upper}. The map $d$ is the smooth analogue of $c$, i.e.\  extension via the identity from $N$ to $N'$. 
Note that this map is also surjective (using \cref{theorem:berghelea_lashof}).

We start with $[F] \in \pi_0\left(\mathcal{P}^{\DIFF}(N',\partial N') \right)$. This is sent to $[F] \in \pi_0\left(\mathcal{P}^{\TOP}(N',\partial N') \right)$ via the vertical isomorphism. Next we have $c([F])=[G]$, where $G|_{N'}=F$ and $G|_{N\setminus N'}=\Id$.

On the other hand, following the bottom part of the diagram, we have: \[ d([F])=[G],\] with $G$ now smooth and with $G|_{N'}=F$ and $G|_{N\setminus N'}=\Id$. This is then sent to $[G] \in \pi_0\left(\mathcal{P}^{\TOP}(N_{\mathcal{S}},\partial N_{\mathcal{S}}) \right)$. The properties of $G$ are not lost.
It is now clear that the diagram commutes.

\subsection{The main diagram}
\[\begin{tikzcd}
	& {\pi_0 \left( \mathcal{P}^{\TOP}(X,\partial X)\right)} \\
	& {\pi_0\left(\mathcal{P}^{\TOP}(X \times J^2, \partial( X \times J^2))\right)} \\
	{\pi_0\left(\mathcal{P}^{\TOP}(N',\partial N') \right)} && {\pi_0\left(\mathcal{P}^{\TOP}(N,\partial N) \right)} \\
	{\pi_0\left(\mathcal{P}^{\DIFF}(N',\partial N') \right)} && {\pi_0\left(\mathcal{P}^{\DIFF}(N_{\mathcal{S}},\partial N_{\mathcal{S}}) \right)} \\
    \Wh_2(\pi_1(N')) && \Wh_2(\pi_1(N)) \\
    && \Wh_1(\pi_1(X))
	\arrow["{S^2}"{pos=0.4}, from=1-2, to=2-2]
	\arrow["\mathfrak{i}'"'{pos=0.6}, dashed, from=2-2, to=3-1]
	\arrow["\mathfrak{i}"{pos=0.6}, from=2-2, to=3-3]
	\arrow["c"{description}, two heads, from=3-1, to=3-3]
	\arrow["\mathfrak{f}"', from=3-1, to=4-1]
	\arrow["\mathfrak{f}_\mathcal{S}", from=3-3, to=4-3]
	\arrow["d"{description}, two heads, from=4-1, to=4-3]
    \arrow["\cong"', from=5-1, to=5-3]
    \arrow["\Sigma"', from=4-1, to=5-1]
    \arrow["\Sigma", from=4-3, to=5-3]
    \arrow["\cong"', from=5-3, to=6-3]
\end{tikzcd}\]

We now return to the main diagram, noting that the commutativity of the bottom rectangle follows from \Cref{lemma:commutativity_after_extending}. We have only included the arrows in the direction for defining the topological Hatcher-Wagoner invariants, i.e.\ going downwards.  It is a simple exercise to see that the commutativity proved above for the more natural maps implies the commutativity for the above diagram (note that $\mathfrak{i}'$ denotes a choice of lift, not a map).

We can similarly draw the same diagram for a different choice of smooth structure, say $\mathcal{S}'$, and let $F$ be a topological pseudo-isotopy of $X$.  Then the commutativity of the diagrams involving $\mathcal{S}$ and $\mathcal{S}'$ give
\[
\Sigma (\mathfrak{f}_{\mathcal{S}}\circ \mathfrak{i}\circ S^2(F)) = \Sigma(\mathfrak{f}\circ \mathfrak{i}' \circ S^2(F)) = \Sigma (\mathfrak{f}_{\mathcal{S}'}\circ \mathfrak{i}\circ S^2(F))
\] and this implies that our definition of $\Sigma^{\TOP}(F)$ does not depend on the choice of $\mathcal{S}$.
\end{proof}

\subsection{The topological and smooth obstructions coincide for smooth pseudo-isotopies}
We now prove that our topological invariants are compatible with the smooth ones.  This will complete the proof of \Cref{thm:main}.
\begin{lem}\label{lemma:smooth_is_equal_top}
    Let $X$ be a smooth 4-manifold i.e.\ a topological manifold $X$ with a given choice of smooth structure $\mathcal{S}$, and let $F\colon X\to X$ be a smooth pseudo-isotopy, $[F]\in \pi_0\left(\mathcal{P}^{\DIFF}(X,\partial X)\right)$. 
    Then \[
    \Sigma(F) = \Sigma^{\TOP}(F),
    \]
    and \[\Theta(F)=\Theta^{\TOP}(F).\]
\end{lem}

\begin{proof}[Proof of Lemma \ref{lemma:smooth_is_equal_top}]
    We begin by showing that the following diagram commutes.
    \begin{center}
    \begin{tikzcd}
{\pi_0\left(\mathcal{P}^{\TOP}(X,\partial X)\right)} \arrow[d]                                           &  & {\pi_0\left(\mathcal{P}^{\DIFF}(X_{\mathcal{S}},\partial X_{\mathcal{S}})\right)} \arrow[ll]    \arrow[d]                             \\
{\pi_0\left(\mathcal{P}^{\TOP}(X\times J^2, \partial (X \times J^2))\right)}                               &  & {\pi_0\left(\mathcal{P}^{\DIFF}(X_{\mathcal{S}}\times J^2, \partial (X_{\mathcal{S}} \times J^2))\right)} \arrow[ll] \\
{\pi_0\left(\mathcal{P}^{\TOP}(N,\partial N)\right)} \arrow[u]                                           &  &                                                                                      
{\pi_0\left(\mathcal{P}^{\DIFF}(N_{\mathcal{S}},\partial N_{\mathcal{S}})\right)} \arrow[ll] \arrow[u] &  &                                                                                      
\end{tikzcd}
\end{center}
where in $\pi_0 \left(\mathcal{P}^{\DIFF}(N_{\mathcal{S}}, \partial N_{\mathcal{S}})\right)$ we choose the smoothing that coincides with the restriction of $\mathcal{S}$ to $N$, and the horizontal maps are just forgetful maps. We chose the suspension map so that it matches the one used in \cite{hatcher_wagoner_1973} in the smooth case. Thus the upper square commutes.

    Commutativity of the lower square is equivalent to forgetting the smooth structure and then extending the pseudo-isotopy via the identity, which yields the same result as extending the smooth pseudo-isotopy via the identity and the forgetting the smooth structure (this is almost identical to the proof in \ref{sbsbs:middle_rectangle}, so we omit the details).
    Now that we have established commutativity of the above diagram, the proof of the lemma  is finished by showing that the following diagrams commute.\footnote{In the second diagram (given $Y$) we write $\ker(\Sigma^{\DIFF}_Y)$ to indicate $\ker \Sigma : \pi_0(\mathcal{P}^{\DIFF}(Y,\partial Y))\to \Wh_2(\pi_1(Y))$.}
\begin{center}
    \begin{tikzcd}
        \pi_0\left(\mathcal{P}^{\DIFF}(X_{\mathcal{S}},\partial X_{\mathcal{S}})\right) \arrow[d] \arrow[r] & \Wh_2\left(\pi_1(X) \right) \arrow[d,"\cong"] \\
        \pi_0\left(\mathcal{P}^{\DIFF}(X_{\mathcal{S}}\times J^2, \partial (X_{\mathcal{S}}\times J^2))\right)\arrow[r] & \Wh_2\left(\pi_1(X \times J^2)\right)  \\
        \pi_0\left(\mathcal{P}^{\DIFF}(N_{\mathcal{S}},\partial N_{\mathcal{S}})\right) \arrow[r] \arrow[u,"\cong"] & \Wh_2\left( \pi_1(N_{\mathcal{S}}) \right)\arrow[u,"\cong"']
    \end{tikzcd}
    \end{center}
and
    \begin{center}

    \begin{tikzcd}
	\ker \Sigma_{X_{\mathcal{S}} }&& {\Wh_1(\pi_2(X_{\mathcal{S}}),\Z/2 \times \pi_1(X_{\mathcal{S}}))/ \chi} \\
	\ker \Sigma_{X_{\mathcal{S}}\times J^2} && {\Wh_1(\pi_2(X_{\mathcal{S}}\times J^2),\Z/2 \times \pi_1(X_{\mathcal{S}}\times J^2))/ \chi} \\
	\ker \Sigma_{N_{\mathcal{S}}} && {\Wh_1(\pi_2(N_{\mathcal{S}}),\Z/2 \times \pi_1(N_{\mathcal{S}}))/ \chi}
	\arrow[from=1-1, to=1-3]
	\arrow[from=1-1, to=2-1]
	\arrow["\cong", from=1-3, to=2-3]
	\arrow[from=2-1, to=2-3]
	\arrow["{\cong }", from=3-1, to=2-1]
	\arrow[from=3-1, to=3-3]
	\arrow["\cong"', from=3-3, to=2-3]
\end{tikzcd}
\end{center}

The commutativity of the top squares follow from \Cref{lrm:suspension_preserves_Sigma_Theta}. The commutativity of the bottom squares follow from \Cref{lemma:commutativity_after_extending}. This finishes the proof of \Cref{lemma:smooth_is_equal_top}, which finally finishes the proof of \Cref{thm:main}. \end{proof}

\subsection{Naturality of the topological Hatcher-Wagoner invariants}

This subsection is devoted to proving that our invariants are natural under certain inclusions of codimension zero submanifolds, i.e.\ \Cref{lemma:proposition_naturality_topological}, which we repeat here for the reader's convenience.

\begin{proposition}
    Let $X = Y \cup_{W} Z$, where $W$ is a (connected) codimension-0 submanifold of $\partial Y$. Let $F$ be a pseudo-isotopy of $X$ that satisfies $F|_{Z \times I}=\Id$.  Let $i_{Y,X}\colon Y\to X$ be the inclusion map. Then $\Sigma^{\TOP}(F)=(i_{Y,X})_*\Sigma^{\TOP}(F|_Y)$. If $F$ lies in the kernel of $\Sigma^{\TOP}$, then we have $\Theta^{\TOP}(F)=(i_{Y,X})_*\Theta^{\TOP}(F|_Y)$. 
\end{proposition}

\begin{proof}
    We start with $\Sigma^{\TOP}$.  We can write the following diagram:
    \begin{center}
        \begin{tikzcd}
\Wh_2(\pi_1(X))                                                                     &  & \Wh_2(\pi_1(Y)) \arrow[ll]                                                              \\
{\pi_0(\mathcal{P}^{\DIFF}(N,\partial N))} \arrow[d, "\cong"'] \arrow[u, "\Sigma"] &  & {\pi_0(\mathcal{P}^{\DIFF}(N',\partial N'))} \arrow[ll] \arrow[d, "\cong"] \arrow[u, "\Sigma"'] \\
{\pi_0(\mathcal{P}^{\TOP}(N,\partial N))} \arrow[d, "\cong"']                      &  & {\pi_0(\mathcal{P}^{\TOP}(N',\partial N'))} \arrow[d, "\cong"] \arrow[ll]                       \\
{\pi_0(\mathcal{P}^{\TOP}(X \times J^2, \partial(X\times J^2)))}                               &  & {\pi_0(\mathcal{P}^{\TOP}(Y \times J^2, \partial(Y\times J^2)))} \arrow[ll]                                  \\
{\pi_0(\mathcal{P}^{\TOP}(X,\partial X))} \arrow[u]                                &  & {\pi_0(\mathcal{P}^{\TOP}(Y, \partial Y))} \arrow[ll] \arrow[u]                                
\end{tikzcd}
    \end{center}

Here $N$ is a neighbourhood of the 3-handle-skeleton in $X \times J^2$ and $N'$ is a neighbourhood of the 3-handle skeleton in $Y \times J^2$, where we choose the handle decompositions such that the handle decomposition of $Y\times J^2$ extends to the chosen handle decomposition of $X\times J^2$.  This is possible because the handle decomposition of $Y\times J^2$ needs to be relative to $\partial Y \times J^2$ and hence relative to $W \times J^2$.
The point is to show that the above diagram commutes. 
We start with the bottom square.
    \begin{center}
        \begin{tikzcd}
{\pi_0(\mathcal{P}^{\TOP}(X\times J^2,\partial(X \times J^2)))}   & {\pi_0(\mathcal{P}^{\TOP}(Y \times J^2, \partial(Y\times J^2)))} \arrow[l]  \\
{\pi_0(\mathcal{P}^{\TOP}(X,\partial X))} \arrow[u, "S^2"] & {\pi_0(\mathcal{P}^{\TOP}(Y, \partial Y))} \arrow[l] \arrow[u,"S^2"]
\end{tikzcd}
    \end{center}

The upward pointing maps are given by the suspension maps $S^2$ on $X$ and $Y$ respectively.
 The bottom map $ \pi_0( \mathcal{P}^{\TOP}(Y,\partial Y)) \xrightarrow{}\pi_0( \mathcal{P}^{\TOP}(X,\partial X))$ is given by extension via the identity over $Z$. This makes sense since the pseudo-isotopies in $\pi_0( \mathcal{P}^{\TOP}(Y,\partial Y))$ fix $W$.
 The map $ \pi_0(\mathcal{P}^{\TOP}(Y \times J^2, \partial(Y \times J^2) )) \xrightarrow{} \pi_0(\mathcal{P}^{\TOP}(X \times J^2, \partial(X\times J^2)))$ is given by the suspended version of the extension via the identity.
 It is clear that this diagram commutes by definition of the maps.
We consider the next square.
 \begin{center}
     \begin{tikzcd}
{\pi_0(\mathcal{P}^{\TOP}(N,\partial N))} \arrow[d] & {\pi_0(\mathcal{P}^{\TOP}(N',\partial N'))} \arrow[l] \arrow[d] \\
{\pi_0(\mathcal{P}^{\TOP}(X\times J^2,\partial(X\times J^2)))}   & {\pi_0(\mathcal{P}^{\TOP}(Y \times J^2, \partial(Y\times J^2)))} \arrow[l]  
\end{tikzcd}
 \end{center}
The arrows pointing downwards are given by extending via the identity outside a neighbourhood $N$ (resp. $N'$) of the 3-handle-skeleton of $X \times J^2$ (resp. $Y \times J^2$). 
The bottom horizontal arrow is given by the suspended-identity extension. 
The top horizontal map is given by extending via the identity. Note that this is well defined thanks to our choice for the handle decompositions on $X\times J^2$ and on $Y\times J^2$ to be compatible---plus, we know that all pseudo-isotopies have to fix $W \times J^2$. Thus the top horizontal map is defined by extending via the identity over the 3-handle skeleton of $Z$. It is clear now that the diagram commutes.
The next square is easier to study.
\begin{center}
    \begin{tikzcd}
{\pi_0(\mathcal{P}^{\DIFF}(N,\partial N))} \arrow[d] & {\pi_0(\mathcal{P}^{\DIFF}(N',\partial N'))} \arrow[d] \arrow[l] \\
{\pi_0(\mathcal{P}^{\TOP}(N,\partial N))}            & {\pi_0(\mathcal{P}^{\TOP}(N',\partial N'))} \arrow[l]           
\end{tikzcd}
\end{center}
The downward pointing arrows are just given by forgetful maps. The horizontal maps are extensions (topological and smooth) via the identity. Commutativity of the diagram is immediate (again, the argument is just like in \ref{sbsbs:middle_rectangle}).

The only thing left to check is the topmost square, which can be factored into the following diagram:
\[\begin{tikzcd}
	{\Wh_2(\pi_1(X))} & {\Wh_2(\pi_1(Y))} \\
	{\Wh_2(\pi_1(N))} & {\Wh_2(\pi_1(N'))} \\
	{\pi_0(\mathcal{P}^{\DIFF}(N,\partial N))} & {\pi_0(\mathcal{P}^{\DIFF}(N',\partial  N'))}
	\arrow[from=1-2, to=1-1]
	\arrow[from=2-1, to=1-1]
	\arrow[from=2-2, to=1-2]
	\arrow[from=2-2, to=2-1]
	\arrow[from=3-1, to=2-1]
	\arrow[from=3-2, to=2-2]
	\arrow[from=3-2, to=3-1]
\end{tikzcd}\]
The bottom part commutes thanks to Lemma \ref{lemma:commutativity_after_extending}. The topmost portion commutes by analysing the definition of $\Wh_2(\pi_1)$ (in particular, note that the corresponding square on $\pi_1$ commutes). We now analyse the $\Theta$ case. We write the following commutative diagram:
  \[\begin{tikzcd}[cramped]
	{\Wh_1(\pi_1X;\Z/2\times\pi_2X)/\chi} && {\Wh_1(\pi_1Y;\Z/2\times \pi_2Y)/\chi} \\
	{\ker (\Sigma_N^{\DIFF})} && {\ker(\Sigma_{N'}^{\DIFF})} \\
	{\pi_0(\mathcal{P}^{\DIFF}(N,\partial N))} && {\pi_0(\mathcal{P}^{\DIFF}(N',\partial N'))} \\
	{\pi_0(\mathcal{P}^{\TOP}(N,\partial N))} && {\pi_0(\mathcal{P}^{\TOP}(N',\partial N'))} \\
	{\pi_0(\mathcal{P}^{\TOP}(X \times J^2, \partial(X\times J^2)))} && {\pi_0(\mathcal{P}^{\TOP}(Y \times J^2, \partial(Y\times J^2)))} \\
	{\pi_0(\mathcal{P}^{\TOP}(X,\partial X))} && {\pi_0(\mathcal{P}^{\TOP}(Y, \partial Y))}
	\arrow[from=1-3, to=1-1]
	\arrow[from=2-1, to=1-1]
	\arrow[hook, from=2-1, to=3-1]
	\arrow[from=2-3, to=1-3]
	\arrow[from=2-3, to=2-1]
	\arrow[hook', from=2-3, to=3-3]
	\arrow["\cong"', from=3-1, to=4-1]
	\arrow[from=3-3, to=3-1]
	\arrow["\cong", from=3-3, to=4-3]
	\arrow["\cong"', from=4-1, to=5-1]
	\arrow[from=4-3, to=4-1]
	\arrow["\cong", from=4-3, to=5-3]
	\arrow[from=5-3, to=5-1]
	\arrow[from=6-1, to=5-1]
	\arrow[from=6-3, to=5-3]
	\arrow[from=6-3, to=6-1]
\end{tikzcd}\]

As before, $N$ is a neighbourhood of the 3-handle skeleton in $X \times J^2$ and $N'$ is a neighbourhood of the 3-handle skeleton in $Y\times J^2$, where we choose the handle decomposition such that the handle decomposition of $Y \times J^2$ extends to the chosen handle decomposition of $X \times J^2$.

 We will consider the above diagram. Note that, thanks to the above steps we know that if $\Sigma(F|_Y)=0$ then $\Sigma(F)=0$. As such, $F$ lies inside $\ker ( \Sigma: \mathcal{P}^{\TOP}(X,\partial X) \to \Wh_2(\pi_1X))$.
 The goal is to show the diagram is commutative. 
 We will use it to compute $\Theta^{\TOP}$ for maps in $\ker (\Sigma_X)$.
Hence for our purposes we know that once we land in $\pi_0(\mathcal{P}^{\DIFF}(N,\partial N))$ (resp. $\pi_0(\mathcal{P}^{\DIFF}(N',\partial N'))$) we actually land in $\ker (\Sigma_N)$ (resp. $\ker (\Sigma_{N'})$).

For most of the squares the proofs are identical to the ones for $\Sigma^{\TOP}$.  We only focus on the other squares.
The square:
\[\begin{tikzcd}[cramped]
	{\ker (\Sigma_N^{\DIFF})} && {\ker(\Sigma_{N'}^{\DIFF})} \\
	{\pi_0(\mathcal{P}^{\DIFF}(N,\partial N))} && {\pi_0(\mathcal{P}^{\DIFF}(N',\partial N'))}
	\arrow[hook, from=1-1, to=2-1]
	\arrow[from=1-3, to=1-1]
	\arrow[hook', from=1-3, to=2-3]
	\arrow[from=2-3, to=2-1]
\end{tikzcd}\]
commutes trivially since inclusion of the kernel commutes by extension with the identity---note that the top map is well defined thanks to Lemma \ref{lemma:commutativity_after_extending}.
Finally we need to analyse the following diagram.
\[\begin{tikzcd}[cramped]
	& {\Wh_1(\pi_1X;\Z/2\times \pi_2X)/\chi} & {\Wh_1(\pi_1Y;\Z/2\times \pi_2Y)/\chi} \\
	{} & {\Wh_1(\pi_1N;\Z/2\times \pi_2N)/\chi} & {\Wh_1(\pi_1N';\Z/2\times \pi_2N')/\chi} \\
	& {\ker (\Sigma_N^{\DIFF})} & {\ker(\Sigma_{N'}^{\DIFF})}
	\arrow[from=1-3, to=1-2]
	\arrow[from=2-2, to=1-2]
	\arrow[from=2-3, to=1-3]
	\arrow[from=2-3, to=2-2]
	\arrow["\Theta", from=3-2, to=2-2]
	\arrow["\Theta"', from=3-3, to=2-3]
	\arrow[from=3-3, to=3-2]
\end{tikzcd}\]

As for $\Sigma^{\TOP}$, the bottom part of the diagram commutes thanks to Lemma \ref{lemma:commutativity_after_extending}. The topmost part commutes by analysing the definition of $\Wh_1$ and by noting that the corresponding squares involving $\pi_1$ and $\pi_2$ also commute.\qedhere

\end{proof}

	\section{One-parameter families of handle decompositions}\label{sec:handle_decompositions}
We now shift our attention to the realisation problem for $\Sigma$ and $\Theta$. The first step in that direction is to define a potential candidate pseudo-isotopy (or, more precisely, isotopy class of pseudo-isotopies) that \emph{might} realise the desired element $x \in  \Wh_2(\pi_1(X))$ or, if $x=0$, the desired $ y \in \Wh_1(\pi_1(X);\Z/2 \times \pi_2(X))/ \chi$. We do this by carefully emulating the smooth functional approach in a purely topological setting.

In this section we describe the framework needed to describe the candidate pseudo-isotopies. We give a definition of an \emph{allowed} one-parameter family of topological handle decompositions on $X \times I$, where $X$ is a topological 4-manifold, and we will see that from these objects we obtain a well defined topological pseudo-isotopy. This construction will, roughly speaking, mimic  the ``one-parameter family of Morse functions in nested eyes position" discussion in \cite{hatcher_wagoner_1973}.  Later, in \Cref{sec:realisation_sigma} and \Cref{sec:realisation_theta}, we will then show how to construct a specific allowed one-parameter family of handle decompositions which will correspond to an element $x\in\Wh_2(\pi_1(X))$ or $y\in \Wh_1(\pi_1(X);\Z/2\times \pi_2(X))/\chi$.

It is important to stress that this is an ``ad hoc" construction. We do not claim any genericity of our families of handle structures and we do not claim that there is a strong correspondence between a space of handle structures and our invariants. In particular, we are not describing anything that could be described as \emph{topological Cerf theory} and, at the current time, such a thing does not exist in the literature. We will satisfy ourselves with producing some concrete examples of one-parameter families of handle decompositions that will suffice to realise our invariants.

\subsection{Allowed one-parameter families of topological handle decompositions.}  

\begin{notation}
    Let \[D^+:=\{(x_1,x_2,x_3,x_4,x_5)\in \R^5 \mid (x_1^2+x_2^2+x_3^2+x_4^2+x_5^2\leq 1), x_5\geq 0 \}\] denote the standard 5-dimensional half-disc of radius one in $\R^5$ centred at the origin.  There is a natural decomposition $\partial D^+ = \partial_- D^+ \cup \partial_+ D^+$ where $\partial_-$ denotes the region given by setting $x_5=0$, and $\partial_+$ denotes the remaining portion of the boundary.  We fix once and for all a decomposition $D^+=\mathfrak{h}_2\cup \mathfrak{h}_3$, where \[\mathfrak{h}_3:= \{(x_1,x_2,x_3,x_4,x_5)\in D^+\mid x_1^2+x_2^2+x_5^2\leq 3/5\}\] and $\mathfrak{h}_2:= D^+\sm \mathring{\mathfrak{h}}_3$.  This decomposes $D^+$ into a cancelling 2-handle and 3-handle pair, after fixing parameterisations $\mathfrak{h}_2= D^2\times D^3$ and $\mathfrak{h}_3 = D^3\times D^2$, which we do now (and use implicitly throughout).
\end{notation}

\begin{definition}[Allowed one-parameter families of topological handle decompositions]\label{def:one-parameter_families}
    Let $X$ be a compact topological 4-manifold.  Suppose we have the following input data.

\begin{itemize}
    
    \item A collection of distinct times $\mathcal{B}:=\{b_1, \dots, b_n\in [0,1]\}$ called the \emph{birth times};
    \item A collection  of distinct times $\mathcal{D}:=\{d_1, \dots, d_n\in [0,1]\sm \mathcal{B}\}$ called the \emph{death times} such that $b_k<d_k$ for all $k$;
    \item A collection of times $\mathcal{A}_2:=\{a^2_1, \dots, a^2_m\in [0,1]\sm \left( \mathcal{B}\cup\mathcal{D}\right)\}$ called the \emph{2-handle slide times} such that for each time there is an associated \emph{sliding 2-handle index pair} $k(a^2_i)= (\vv{k}(a_i^2), \dot{k}(a_i^2))$ and such that $b_{\vv{k}(a^2_i)}<a^2_i<d_{\vv{k}(a^2_i)}$ and $b_{\dot{k}(a^2_i)}<a^2_i<d_{\dot{k}(a^2_i)}$;
    \item A collection of times $\mathcal{A}_3:=\{a^3_1, \dots, a^3_p\in [0,1]\sm \left( \mathcal{B}\cup\mathcal{D}\right)\}$ called the \emph{3-handle slide times} such that for each time there is an associated \emph{sliding 3-handle index pair} $k(a^3_i)= (\vv{k}(a_i^3),\dot{k}(a_i^3))$ and such that $b_{\vv{k}(a^3_i)}<a^3_i<d_{\vv{k}(a^3_i)}$ and $b_{\dot{k}(a^3_i)}<a^3_i<d_{\dot{k}(a^3_i)}$;
\end{itemize}

Then an \emph{allowed one-parameter family of handle decompositions} $\mathcal{F}_t$, $t\in [0,1]$, for $X\times I$ is given by the following data.
\begin{itemize}
    \item A continuous $1$-parameter family of embeddings $F^{-}_t\colon X\times [0,1/2]\to X\times I$ such that $F^{-}_t\vert_{X\times\{0\}\cup \partial X\times I}$ is the inclusion map.
    \item A continuously varying one-parameter family of collections of 2-handle embeddings $\mathcal{E}_2^t:=\{\prescript{2}{}{h}^{k}_{t}: D^2 \times D^3 \hookrightarrow X \times I$\} such that $\prescript{2}{}{h}^{k}_{t}\vert_{S^1\times D^3}$ maps into $F^{-}_t(X\times\{1/2\})$ (except for times near $t\in \mathcal{A}_2$, as explained below) where for each $k$ the time parameter lives in the range $t\in (b_k,d_k)$.  In a small neighbourhood of each $a\in \mathcal{A}_2$, the embedding $\prescript{2}{}{h}^{\vv{k}(a)}_{t}\vert_{S^1\times D^3}$ is allowed to map into \[\left(F_t^{-}(X\times\{1/2\})\sm \prescript{2}{}{h}^{\dot{k}(a)}_{t}(S^1\times D^3)\right) \cup \prescript{2}{}{h}^{\dot{k}(a)}_{t}(D^2\times S^2)\] (this is a slide of the $\vv{k}(a)$-th 2-handle over the $\dot{k}(a)$-th 2-handle).
    \item A continuously varying one-parameter family of collections of 3-handle embeddings $\mathcal{E}_3^t:=\{\prescript{3}{}{h}^{k}_{t}: D^3 \times D^2 \hookrightarrow X \times I$\} such that $\prescript{3}{}{h}^{k}_{t}\vert_{S^2\times D^2}$ maps into 
    \[
    M_t^{-}:=\left(F_t^{-}(X\times\{1/2\})\sm \bigcup_{k}\left(\prescript{2}{}{h}^{k}_{t}(S^1\times D^3)\right)\right) \cup \bigcup_k \prescript{2}{}{h}^{k}_{t}(D^2\times S^2)
    \] 
    (except for times near $t\in \mathcal{A}_3$, as explained below) where for each $k$ the time parameter lives in the range $t\in (a_k,d_k)$.  In a small neighbourhood of each $a\in \mathcal{A}_3$, the embedding $\prescript{3}{}{h}^{\vv{k}(a)}_{t}$ is allowed to attach along \[\left(M_t^{-}\sm \prescript{3}{}{h}^{\dot{k}(a)}_{t}(S^2\times D^2)\right) \cup \prescript{3}{}{h}^{\dot{k}(a)}_{t}(D^3\times S^1)\] (this is a slide of the $\vv{k}(a)$-th 3-handle over the $\dot{k}(a)$-th 3-handle).
    \item A one-parameter family of homeomorphisms 
    \[F^{+}_t\colon X\times [1/2,1] \to X \times I \setminus \left(F^{-}_t(X\times [0,1/2])\cup\bigcup_k \prescript{2}{}{h}^{k}_{t} \cup \bigcup_k \prescript{3}{}{h}^{k}_{t}\right)\] 
    such that $F_t^{+}\vert_{X\times \{1/2\}}$ glues together with $F_t^{-}\vert_{X\times \{1/2\}}$ wherever they meet, which is continuous everywhere except at times $t\in \mathcal{B}\cup \mathcal{D}$, where the the limits from below and above exist separately.  We now define the behaviour at such a birth point, e.g.\ $t=b_j$. Fix notation for the limit from below as $\prescript{\shortuparrow}{}{F}^{+}_{t}:=\lim_{t\nearrow b_j}F_t^{+}$ (respectively, the limit from above as $\prescript{\shortdownarrow}{}{F}^{+}_{t}:=\lim_{t\searrow b_j}F_t^{+}$).

    \begin{enumerate}
        \item We require that there is a given embedding  $D_j\colon D^+\hookrightarrow X\times[1/2,1]$ such that $D_j(\partial_-D^+)$ lies on $X\times\{1/2\}$ and such that $\prescript{\shortuparrow}{}{F}^{+}_{t}(D_j(\partial_-D^{+}))$ lies on $F^{-}_t(X\times \{1/2\})$.
        \item We require that the $j$-th 2-handle embedding $\prescript{2}{}{h}^{j}_{t}$ can be described immediately after its birth as \[\lim_{t\searrow b_j}\left(\prescript{2}{}{h}^{j}_{t}\right)=\prescript{\shortuparrow}{}{F}^{+}_{b_j}(D_j(\mathfrak{h}_2))\]
        and similarly for the $j$-th 3-handle embedding we require that \[\lim_{t\searrow b_j}\left(\prescript{3}{}{h}^{j}_{t}\right)=\prescript{\shortuparrow}{}{F}^{+}_{b_j}(D_j(\mathfrak{h}_3)).\]
        \item The embedding $D_j(D^+)$ guides an ambient isotopy of $D_j(\partial_- D^+)$ to $D_j(\partial_+D^+)$.  Applying isotopy extension to this and post-composing with $\prescript{\shortuparrow}{}{F}^{+}_{b_j}$ describes an isotopy of $\prescript{\shortdownarrow}{}{F}^{+}_{b_j}$ relative to most of the boundary (all but the piece corresponding to $D_j(\partial_- D^+)$, which clearly moves) to a new homeomorphism which we denote by $G$.  We require that $\prescript{\shortdownarrow}{}{F}^{+}_{b_j}=G$.
    \end{enumerate}

    Similarly for a death point, e.g.\ $t=d_j$ we require the same but in the reverse time direction; for brevity we leave the details to the reader.

\end{itemize}
\end{definition}

\begin{remark}
    The condition on the homeomorphisms $F^{+}_t$ means that for times $t$ before any births have occured or after all deaths have occured $\mathcal{F}_t$ consists of a \emph{single} homeomorphism $F_t := F^{-}_t\cup F^{+}_t\colon X\times I\to X\times I$.
\end{remark}

\begin{definition}
    Let $\mathcal{F}_t$ be an allowed one-parameter family of handle decompositions of $X\times I$.  Then the \emph{pseudo-isotopy associated to} $\mathcal{F}_t$ is given by $F_1\colon X\times I\to X\times I$.
\end{definition}

We also have the analogous notion of \emph{isotopy} for allowed one-parameter families of handle decompositions, i.e.\ a separate isotopy of each piece of constituent data of $\mathcal{F}_t$ such that for each time slice of the isotopy, the data still fits together to form a one-parameter family, as in \Cref{def:one-parameter_families}.  We will omit the details for this.

We will need the following compatibility result between smooth one-parameter families and our allowed one-parameter families of topological handle decompositions.

\begin{lem}\label{lem:smooth_top_handle_decomposition}
    Let $X$ be a smooth 4-manifold and let $(g_t,\eta_t)$ be a one-parameter family for a smooth pseudo-isotopy $F\colon X\times I\to X\times I$.  Then there exists a one-parameter family of topological handle decompositions $\mathcal{F}_{t}(g_t,\eta_t)$ such that $F$ is equal to the pseudo-isotopy associated to $\mathcal{F}_t$.
\end{lem}

\begin{proof}
    We sketch the proof.  Deform $(g_t,\eta_t)$ to one with only index two and three critical points and such that first all births occur, then all handle slides, then all deaths \cite[Part I, Chapter V, Theorem 3.1]{hatcher_wagoner_1973}.  Choose a one-parameter family of Morse coordinate charts $\prescript{2}{}{\varphi}^{k}_t$ and $\prescript{3}{}{\varphi}^{k}_t$ for the critical points of $g_t$ as in \cite[Section 3]{laudenbach_2014}, as well as a one-parameter family of collars for $X\times\{0\}$ which we denote by $F_t^{-}$.  At a time slice $t$ where no births, deaths or handle slides occur, the Morse coordinate neighbourhoods for the index two critical points determine smooth handle embeddings $\prescript{2}{}{h}^{k}_t$ after flowing down to the collar using $\eta_t$.  Similarly, the Morse coordinate neighbourhoods for the index three critical points then determine smooth handle embeddings $\prescript{3}{}{h}^{k}_t$ after flowing with $\eta_t$.  Since all of the data was chosen in a one-parameter fashion, this yields not just handle embeddings at each $t$, but actually a smooth one-parameter family of handle embeddings.  To complete the handle decomposition, we produce the final collar by flowing down from $X\times\{1\}$ using $\eta_t$ to form another collar embedding $F_t^{+}$.

    If we are at time $a$ when a handle slide occurs (see \cite[Part I, Chapter II, Proof of Lemma 1.2]{hatcher_wagoner_1973}) then we append the time $a$ to $\mathcal{A}_2$ or $\mathcal{A}_3$, depending on the index of the handle slide.  Assume the index of the handle slide is $2$, for simplicity.  By independence of trajectories \cite[Part I, I \S 7]{hatcher_wagoner_1973}, we can assume that all of the other data stays fixed in a small time neighbourhood of the handle slide.  Let the label of sliding handle be $\vv{k}(a)$ and the label of the handle that is being slid over be $\dot{k}(a)$ and then the sliding 2-handle index pair is $k(a):= (\vv{k}(a),\dot{k}(a))$.  By arrangement, during a small time neighbourhood of $a$, the construction in the above paragraph means that the handle $\prescript{2}{}{h}^{\vv{k}(a)}_t$ attaches along not just the collar, but also the belt region of the handle $\prescript{2}{}{h}^{\dot{k}(a)}_t$.  However, this is precisely what is allowed in \Cref{def:one-parameter_families}.

    Since births and deaths are inverses to each other, we will only describe what happens at a time $b$ when a birth occurs, involving index $2$ and $3$ critical points with label $k$.  First we append $b$ to $\mathcal{B}$.  At time $b+\varepsilon$ the corresponding pair of handles are in cancelling position.  We can assume that the collar of $X\times\{0\}$ stays constant throughout the handle birth (again by independence of trajectories), and so the Morse coordinate charts $\prescript{2}{}{\varphi}^{k}_{b+\varepsilon}$ and $\prescript{3}{}{\varphi}^{k}_{b+\varepsilon}$ and the handle cancellation lemma \cite[Lemma 5.3]{milnor_1965} give that there exists a diffeomorphism
    \[
    f\colon \prescript{2}{}{\varphi}^{k}_{b+\varepsilon}\cup \prescript{3}{}{\varphi}^{k}_{b+\varepsilon} \to D^{+}_5
    \]
    and this diffeomorphism describes how to absorb the union of the handles $\prescript{2}{}{h}^{k}_{b+\varepsilon}$ and $\prescript{3}{}{h}^{k}_{b+\varepsilon}$ into the collar $F^{+}_t$; this is what it needed to describe the data for the one-parameter family of handle decompositions immediately before the handle birth.  Again using independence of trajectories, we can assume that all of the other data stays fixed in a small time neighbourhood of the birth and hence contracting the data of the birth for times $b<t<b+\varepsilon$ creates the situation described in \Cref{def:one-parameter_families}.  Throwing away the smoothness data then produces the given allowed one-parameter family of handle decompositions.
\end{proof}

We call any such $\mathcal{F}_t(g_t,\eta_t)$ the allowed one-parameter family of handle decompositions induced by $(g_t,\eta_t)$, despite it not being unique by any means.

\subsection{Constructions}

\Cref{def:one-parameter_families} gives an extension of the smooth one-parameter families considered in Cerf and Hatcher-Wagoner, but so far the only way we have to build one is to take the one induced by a smooth family (as in \Cref{lem:smooth_top_handle_decomposition}).  This will almost suffice for our methods; we will need to have a way to arrange for handle deaths that does not rely on the smooth machinery.  In contrast, all handle births and handle slides will be constructed smoothly.  One could in principle construct these topologically.  Roughly, one should use a result of Freedman-Quinn \cite[8.7D]{freedman_quinn_1990} which says that $X\times I\sm F^{-}_t(X\times[0,1/2])$ is smoothable away from a proper 1-dimensional submanifold, and then one should perform all of the smooth constructions (such as handle births and handle slides) in the complement of this submanifold.  We will not discuss this idea further since we will not need it.

 \subsubsection{Deaths of cancelling handle pairs}\label{sbsbs:death}

 Assume we are at a time $t$ where we have all of the data used in \Cref{def:one-parameter_families} but no deaths have occurred. To simplify the notation we will assume that the embedding $F^{-}_t$ simply given by the inclusion map in a time-neighbourhood where the deaths will occur.  The extra data is then a collection of handle embeddings $\prescript{2}{}{h}^{k}_{t}$ and $\prescript{3}{}{h}^{k}_{t}$.
 
 Our aim here is to describe the geometric data necessary for these handles can be cancelled, i.e.\ such that at some future time $t'$ our data consists only of an embedding $F_{t'}\colon X\times I\to X\times I$, which will then be the pseudo-isotopy that has been created by this whole process.

Consider the submanifold 
\[
M_t^{-}\approx X\# n(S^2\times S^2)
\]
from \Cref{def:one-parameter_families}.  The homeomorphism is given by an identification that we choose and fix throughout this discussion, which has the property that the $n$ 2-spheres $B_k:= \prescript{2}{}{h}^{k}_{t}(D^2\times S^2)=\{\pt\}\times S^2\subset X\# k(S^2\times S^2)$ are the framed images of the belt spheres of the 2-handles $\prescript{2}{}{h}^{k}_{t}$.  The embeddings of the 3-handles $\prescript{3}{}{h}^{k}_{t}$ when restricted to the attaching region gives embeddings $\prescript{3}{}{h}^{k}_{t}\vert_{S^2\times D^2}$ which determines framed embedded 2-spheres $A_k\subset M_t^{-}\approx X\#k(S^2\times S^2)$.  In the smooth category a necessary and sufficient condition for 2- and 3-handles to cancel is given by the Morse cancellation theorem \cite[Theorem 5.4]{milnor_1965}.  We will make do without this at the expense of needing a stronger hypothesis on the intersection between the attaching spheres and the belt spheres.

Let $A_k^+$ (resp.\ $A_k^-$) denote the upper (resp.\ lower) hemisphere of the framed attaching sphere of the $k$-th 3-handle, and let $B_k^+$ (resp.\ $B_k^-$) denote the upper (resp.\ lower) hemisphere of the framed belt sphere of the $k$-th 2-handle.

\begin{lemma}\label{lem:top_handle_cancellation}
    Assume that the attaching spheres $A_k$ are all disjoint and $A_k \cap B_j=\emptyset$ if $k\neq j$.  Further assume that the $A^k_+$ and $B^k_-$ completely coincide (as subsets) inside $M_t^{-}$ and that no other part of $A^k$ intersects $h_2^k$.  Then there are $n$ embeddings $d_{2,3}^k\colon D_+ \to X\times I$ such that $d_{2,3}^k\vert_{\mathfrak{h}_2}=\prescript{2}{}{h}^{k}_{t}$ and similarly $d_{2,3}^k\vert_{\mathfrak{h}_3}=\prescript{3}{}{h}^{k}_{t}$.
\end{lemma}

\begin{proof}
    
    The region where the embeddings $\prescript{2}{}{h}^{k}_{t}$ and $\prescript{3}{}{h}^{k}_{t}$ of $\mathfrak{h}_2$ and $\mathfrak{h}_3$, respectively, coincide is homeomorphic to a copy of $D^2\times D^2 \approx D^4$ and hence (relative to the boundary $S^3$ after using that $\pi_0 \Homeo(S^3)$ is trivial) we can make the embeddings match using the Alexander trick.  Now we can glue together the embeddings and form \[d_{2,3}:=h_2^k\cup h_3^k\colon \mathfrak{h}_2\cup \mathfrak{h}_3(=D^+)\to X\times I,\] which has the desired properties.
\end{proof}

We now describe the handle death procedure.  We say that a pair of handles which satisfy the conditions of \Cref{lem:top_handle_cancellation} are in \emph{topological cancelling position}.  Assume that our two handles $h_2^k$ and $h_3^k$ are in topological cancelling position.  Then by \Cref{lem:top_handle_cancellation} we can combine these two embeddings into a single embedding $d_{2,3}^k$.  We now absorb this embedding into the upper collar $F^{+}_{t+\delta}$, as in \Cref{def:one-parameter_families}, which can be done after a potential isotopy of the embedding $d_{2,3}^k$ to ensure that the embedding glues to the upper collar.  By a further isotopy of the new collar $F^{+}_{t+\delta}$, it can be arranged that $F^{+}_{t+\delta}$ and $F^{-}_{t+\delta}$ match along the new region where they meet. 

Before we end this subsection, we prove a lemma which shows how we can achieve the hypotheses of \Cref{lem:top_handle_cancellation} in practice.  This lemma (and its applications) are used implicitly in the proof of the smooth to topological h-cobordism theorem \cite[Theorem 7.1D]{freedman_quinn_1990}. We provide a statement and a proof of such a lemma here.

\begin{lemma}\label{lem:smooth_to_top_handle_cancellation}
    Let $X$ be a compact, smooth $4$-manifold with good fundamental group.  Assume that we have a smooth handle decomposition of $X\times I$ which consists $n$ 2-handles and $n$ 3-handles.  Denote the belt spheres of the 2-handles by $B_k$ and the attaching spheres of the 3-handles by $A_k$ inside the middle level $X\# n(S^2\times S^2)$.  Further assume that $\lambda(A_k,B_k)=1$ and $\lambda(A_k,B_j)=0$ if $k\neq j$, where $\lambda$ here denotes the equivariant intersection form, and that there exist smoothly immersed, framed Whitney discs pairing up all of the excess intersections between each $A_k$ and $B_j$ for all $k$ and $j$, whose interiors are disjoint from $A_k$ and $B_j$.  Then we can arrange via a topological isotopy that the conditions of \Cref{lem:top_handle_cancellation} are satisfied.
\end{lemma}

\begin{proof}
    First we assume that $n=1$, i.e.\ there is only one 2-handle/3-handle pair (we drop the $k$ subscript in the notation).  To begin with, the belt sphere $B$, the attaching sphere $A$ and Whitney discs pairing up the intersections are all smooth.  Denote the framed attaching map of $A$ by $f$.  This data determines the non-excess intersection point $p\in A\cap B$, and we isotope $A$ such that $p=(0,1)\in A= D^2\times S^2\subset S^2\times S^2$.  Pick a sufficiently small neighbourhood $U\subset S^2$ of the preimage of $p$ such that $U$ is mapped to $D^2_{\varepsilon}(0)\times\{1\}\subset D^2\times S^2$, isotoping $A$ if required.  Further isotope the Whitney discs such that the (framed) Whitney arcs lying on $B$ do not intersect $f(U)$.  All of the smooth setup is now complete.

    Now apply the procedure detailed in Freedman-Quinn \cite[Proof of theorem 7.1D]{freedman_quinn_1990} to replace the smoothly immersed Whitney discs by topologically, disjointly embedded Whitney discs with the same (framed) boundaries.  Applying simultaneous Whitney moves across these Whitney discs then changes $A$ to a topological embedding which has a single point of intersection, $p$, with $B$.

    By an application of the classical 2-dimensional annulus theorem there is an isotopy of $f(U)$ to $D^2\times \{1\}$ which fixes the point $p$ throughout.  Ambient isotopy extension yields an isotopy of $f$ which pushes all of $A\sm U$ off this $D^2\times S^2$.  Taking $f(U)$ to be $A_+$ (the upper hemisphere of $A$) and $D^2\times \{1\}$ to be $B_-$ (the lower hemisphere of $B$), we have made those portions of the embeddings entirely coincide as subsets inside $M_t^{-}$.
\end{proof}

\subsection{Producing pseudo-isotopies from one-parameter families of handle decompositions}\label{sbs:producing_PI}

The purpose of our allowed one-parameter families of topological handle decompositions is to be a sufficient stand-in for one-parameter families $(g_t,\eta_t)$.  In the smooth category one can freely go between these one-parameter families and smooth pseudo-isotopies, but we are only able to go one way.  Namely, we can produce topological pseudo-isotopies from allowed one-parameter families.  Given an allowed one-parameter family of topological handle decompositions $\mathcal{F}_t$ a conditition is that for $t=1$ the family consists only of a single piece of a data: a homeomorphism $F_1\colon X\times I\to X\times I$ and this is the pseudo-isotopy produced by $\mathcal{F}_t$.  There is no clear way to build a one-parameter family of topological handle decompositions from a given pseudo-isotopy and so our approach is not a priori helpful for determining if a given pseudo-isotopy can be straightened, as was the original purpose of Cerf theory.  However, we hope that our work concerning the realisation side of the story will eventually be useful in future developments.

\section{Realisation for $\Sigma^{\TOP}$}\label{sec:realisation_sigma}

\begin{thm}\label{thm:realisation_sigma}
     Let $X$ be a compact, topological 4-manifold with good fundamental group.  Then given $x\in \Wh_2(\pi_1(X))$ there exists a pseudo-isotopy $F\colon X\times I\to X\times I$ with $\Sigma^{\TOP}(F)=x$.
 \end{thm}

We will produce $F$ by building an allowed one-parameter family of handle decomposition $\mathcal{F}_t$ that `looks like' it should realise the given element $x\in\Wh_2(\pi_1(X))$, with $F$ produced by $\mathcal{F}_t$ as in \Cref{sbs:producing_PI}.  Given the circuitous nature of our definition of $\Sigma^{\TOP}$, it is then a non-trivial matter to show that $\Sigma^{\TOP}(F)=x$.  We will first describe the initial step, before going on to the computation step.

\subsection{Producing the pseudo-isotopy}

Assume $X$ is a compact smoothable 4-manifold with good fundamental group and let $x\in \Wh_2(\pi_1(X))$, and now pick a smooth structure on $X$ which we fix once and for all.  As in the smooth realisation of $\Sigma$ (see \cite[Part I, Chapter VI, Theorem 2]{hatcher_wagoner_1973} and \cite[Chapter 6]{singh2022pseudoisotopies} we choose a word $\lambda$ in the Steinberg group representing $x$ and then begin building a one-parameter family $(g_t,\eta_t)$ which has index 2 and 3 critical points and handle slides which describe the word $\lambda$.  Note that by \Cref{lem:smooth_top_handle_decomposition} we have all of the data necessary to describe our one-parameter family of topological handle decompositions $\mathcal{F}_t$ except the data necessary to describe the deaths.  We will explain how to produce this now.  After performing all of these handle slides, the picture in the middle level has the all of the belt spheres of the 2-handles and all of the attaching spheres of the 3-handles in algebraically cancelling position, i.e.\ the equivariant intersection matrix is a permutation matrix multiplied by the identity matrix.  In dimensions $\geq 5$, as in the proof of the s-cobordism theorem, one can now use the Whitney trick repeatedly to ensure that the belt spheres and attaching spheres are \emph{geometrically} cancelling.  By the standard argument in the smooth to topological s-cobordism theorem (see \cite[Theorem 7.1D]{freedman_quinn_1990}) there exist a system of smoothly immersed, framed Whitney discs $\mathcal{W}$, whose interiors are disjoint from the belt spheres and attaching spheres, pairing up all of the excess intersections between the belt spheres and attaching spheres and hence we can apply \Cref{lem:smooth_to_top_handle_cancellation} to put our handles into topological cancelling position.  Then an application of \Cref{lem:top_handle_cancellation} allows us to `cancel' the handles and by the construction described in \Cref{sbsbs:death} we can complete our one-parameter family $\mathcal{F}_t$ `corresponding' to $x$.  Note that this is by no means well-defined, in particular since the construction described in \Cref{sbsbs:death} has no uniqueness statement attached to it: there may exist many ways to `kill' a 2-3 handle pair topologically, but we ignore these differences.  Similarly, we will not make any comment on the choice of smooth structure on $X$ picked at the beginning.

We now describe how to produce the corresponding $\mathcal{F}_t$ when $X$ is not smoothable.  By Quinn \cite[Theorem 2.3.1]{quinn_1982} (c.f.\ \cite[Theorem 9.1]{freedman_quinn_1990}) $X\times I$ has a handle decomposition; pick such a handle decomposition.  As in \Cref{sec:def_top_obstructions}, let $N$ be an open neighbourhood of the 3-handle-skeleton of $X\times I$.  By a similar argument to that used in \Cref{sec:def_top_obstructions}, $N$ can be smoothed, but we will be more careful in how we build the smooth structure on it.

By \cite[Theorem 8.6, Proof of Theorem 10.1]{freedman_quinn_1990}, $X\# l E_8\# k (S^2\times S^2)$ is smoothable for some $l\in\{0,1\}$ and $k\geq 0$ (here $l$ is given by the Kirby-Siebenmann invariant of $X$). We pick a smooth structure on $X\# l E_8\# k (S^2\times S^2)$ that we fix once and for all.  By topological transversality\footnote{Apply topological transversality \cite[Essay III, Theorem 1.5]{kirby_siebenmann_1977} to the cores of the relevant handles in the handle decomposition for $N$ (note that these all satisfy the relevant dimensions restrictions in \cite[Essay III, Theorem 1.5]{kirby_siebenmann_1977}).  These are of index 0, 1, 2 and 3 and hence all of these admit normal vector bundles---the rest of the handles---see the discussion in \cite[Section 9.4]{freedman_quinn_1990}.  By shrinking the handles using the vector bundle structures we can make the handles themselves disjoint from the arc, and performing this operation handle by handle gives the required disjointness.} \cite[Essay III, Theorem 1.5]{kirby_siebenmann_1977} we can assume that $N\subset X\times I$ is disjoint from a neighbourhood of $\{\pt\}\times I\subset X\times I$ and then we form the connected sum along an interval $X\# l E_8\# k (S^2\times S^2) \times I$.  This can be equipped with the product smooth structure that we can then restrict to $N$.  This is the smooth structure we will use on $N$.

Since we picked $N$ to be a neighbourhood of the 3-handle-skeleton, this means that we have an inclusion induced isomorphism $\pi_1(N)\cong \pi_1(X)$.  So, after picking $x\in \Wh_2(\pi_1(X))$, on $N$ we can build a one-parameter family just as in the smoothable case which `corresponds' to $x$.  Note that since $N$ is not a product of a 4-manifold with an interval we are not producing a pseudo-isotopy of $N$, but we are producing a one-parameter family of topological handle decompositions $\mathcal{F}_t$ for $X\times I$, since we can build the data for the one-parameter family of homeomorphisms $F_t\colon X\times I\to X\times I$ by extending the one-parameter family on $N$ via the trivial handle decomposition.  In particular, this means that all births, handle slides and deaths occur inside $N\subset X\times I$.

\begin{con}\label{con:sigma_realisation_PI}
    Given a compact 4-manifold $X$ with good fundamental group and $x\in\Wh_2(\pi_1(X))$, we carry out the above procedure, making choices to produce an allowed one-parameter family of handle decompositions which we denote by $\mathcal{F}_t(x)$.  By \Cref{sbs:producing_PI} this produces a pseudo-isotopy, which we denote by $F_x$.
\end{con}

\subsection{An alternative viewpoint on stable surjectivity of $\Sigma$}

To prove \Cref{thm:realisation_sigma} we will compare our realisation $\mathcal{F}_t$ to the stable smooth realisation result for $\Sigma$ due to Singh \cite[Theorem E]{singh2022pseudoisotopies} using the stable surface smoothing theorem of Cha-Kim \cite[Theorem D]{cha_kim_2023}.  Singh builds a smooth pseudo-isotopy realising $x\in \Wh_2(\pi_1(X))$ after $X$ has been stabilised, i.e.\ after taking connected-sums with $S^2\times S^2$.  We will briefly describe an alternative way to achieve the same stable smooth realisation, which more easily compares with our one-parameter family of handle decompositions $\mathcal{F}_t$.  

\begin{lem}\label{lem:stable_smooth_top_comparison_sigma}
    Let $X$ be a compact, smooth 4-manifold, let $x\in \Wh_2(\pi_1(X))$ and let $\mathcal{F}_{t}(x)$ be a one-parameter family of topological handle decompositions associated to $x$ (from \Cref{con:sigma_realisation_PI}), with induced pseudo-isotopy $F_x\colon X\times I\to X\times I$.  Then there exists a smooth pseudo-isotopy $G_x\colon X\#k(S^2\times S^2)\times I\to X\# k(S^2\times S^2) \times I$ with $\Sigma(G_x)=x$ and $F_x$ stably isotopic to $G_x$.
\end{lem}

\begin{proof}
    Assume we are in the same situation as in \Cref{sbs:producing_PI} at the point where we are trying to close the eyes: we have immersed Whitney discs $\mathcal{W}$ pairing the excess intersections between the belt spheres and the attaching spheres. In Singh's proof \cite[Section 6.1]{singh2022pseudoisotopies} he uses the Norman trick \cite{norman_1969} to make the Whitney discs embedded after stabilisation.  We will refine this argument to show that we can stably smooth the topologically embedded Whitney discs given by the disc embedding theorem, though we refer the reader to \cite[Section 6.1]{singh2022pseudoisotopies} for many of the details.

    Singh assumes we have as many stabilisations as there are intersections between Whitney discs $W$, $W'\in\mathcal{W}$ (possibly $W=W'$).  For each intersection Singh performs the Norman trick by tubing $W$ into $S^2\times \{\pt\}\subset S^2\times S^2$ and then using the $\{\pt\}\times S^2$ to tube off the intersection.  Instead of doing this, apply the procedure detailed in Freedman-Quinn \cite[Proof of theorem 7.1D]{freedman_quinn_1990} (c.f.\ the proof of \Cref{lem:smooth_to_top_handle_cancellation}) to build topologically embedded, framed Whitney discs $\mathcal{W}^{\TOP}$ with the same framed boundaries.  Note that applying \Cref{lem:smooth_to_top_handle_cancellation} and \Cref{lem:top_handle_cancellation}, as we did in \Cref{sbs:producing_PI}, is how we produce our one-parameter family $\mathcal{F}_t$.  Now we appeal to the stable surface smoothing theorem of Cha-Kim \cite[Theorem D]{cha_kim_2023} which says that, after some number of stabilisations, there exists a topological isotopy $H$, relative to the boundary of each disc, taking each topologically embedded disc $W''\in \mathcal{W}^{\TOP}$ to a smoothly embedded one $W'''$.  These can be arranged to still be all disjoint , and so the rest of Singh's proof proceeds as written to produce $G_x$, still realising $\Sigma(G_x)=x$.  Note that the number of stabilisations is now the number required by \cite{cha_kim_2023}, rather than the number of intersections between Whitney discs, which may not be the same.

    We now show that $F_x$ is stably isotopic to $G_x$.  Let $\mathcal{G}_t$ be the one-parameter family of topological handle decompositions induced by the one-parameter family $(g_t,\eta_t)$ associated to $G_x$.  First, choose an isotopy of $\mathcal{F}_t$ such that the homeomorphism $F^{-}_t$ fixes a $D^4\times I\subset X\times I$ throughout (the $D^4\times I$ is then away from all of the handle births, handle deaths and handle slides).  This allows us to define a stabilisation $\mathcal{F}^{\#}_t$ which induces a stabilised pseudo-isotopy \[F^{\#}_x\colon (X\# k(S^2\times S^2))\times I\to (X\# k(S^2\times S^2))\times I\] that $F_x$ is stably isotopic to.  One can arrange that the data for $\mathcal{F}^{\#}_t$ and $\mathcal{G}_t$ are completely identical up until the point that the eyes are closed.  Via isotopy extension, $H$ then provides the isotopy between the the rest of the data for $\mathcal{F}^{\#}_t$ and $\mathcal{G}_t$.  In particular, restricting this isotopy to the end of the one-parameter families produces the isotopy between $F^\#_x$ and $G_x$.
\end{proof}

\subsection{Proof of \Cref{thm:realisation_sigma}}

We now give the proof of \Cref{thm:realisation_sigma}.  We will first prove the case for smoothable manifolds and then reduce the non-smoothable case to the smoothable one.

\begin{lem}\label{lem:realisation_sigma_smooth}
    Let $X$ be a compact, smooth 4-manifold and $x\in \Wh_2(\pi_1(X))$.  Let $F_x$ be as in \Cref{con:sigma_realisation_PI}.  Then $\Sigma^{\TOP}(F_x)=x$.
\end{lem}

\begin{proof}
    Let $\mathring{X}$ denote $X$ with a small open ball removed.  Consider the following diagram. 
    \[\begin{tikzcd}
	{\pi_0\mathcal{P}^{\TOP}(X, \partial X)} & {\Wh_2(\pi_1(X))} \\
	{\pi_0\mathcal{P}^{\TOP}(\mathring{X}, \partial \mathring{X})} & {\Wh_2(\pi_1(\mathring{X}))} \\
	{\pi_0\mathcal{P}^{\TOP}(X \# k (S^2 \times S^2), \partial X)} & {\Wh_2(\pi_1(X\#k(S^2\times S^2)))} \\
	{\pi_0\mathcal{P}^{\DIFF}(X \# k (S^2 \times S^2), \partial X)} & {\Wh_2(\pi_1(X\#k(S^2\times S^2)))}
	\arrow[from=1-1, to=1-2, "\Sigma^{\TOP}"]
	\arrow[from=2-1, to=1-1]
	\arrow[from=2-1, to=2-2, "\Sigma^{\TOP}"]
	\arrow[from=2-1, to=3-1]
	\arrow[from=2-2, to=1-2]
	\arrow[from=2-2, to=3-2]
	\arrow[from=3-1, to=3-2,"\Sigma^{\TOP}"]
	\arrow[from=4-1, to=3-1]
	\arrow[from=4-1, to=4-2,"\Sigma"]
	\arrow[from=4-2, to=3-2]
\end{tikzcd}\]
The vertical maps are either induced by inclusions or are forgetful maps.  The top two squares commute by our naturality for certain inclusions of codimension zero submanifolds \Cref{lemma:proposition_naturality_topological}.  The bottom square commutes by \Cref{lemma:smooth_is_equal_top}.  

Using \Cref{lem:stable_smooth_top_comparison_sigma} and a diagram chase involving the above diagram gives that \[\Sigma^{\TOP}(F_x)=\Sigma^{\TOP}(F^{\#}_x)=\Sigma(G_x)=x,\]
where $F^{\#}_x$ is the stabilisation of $F_x$ compatible with $G_x$, as in the proof of \Cref{lem:stable_smooth_top_comparison_sigma}.
\end{proof}
 
\begin{proof}[Proof of \Cref{thm:realisation_sigma}]

\Cref{lem:realisation_sigma_smooth} provides the proof in the case that $X$ is smoothable, so we will now make no assumption that $X$ admits a smooth structure.  We will instead show that we can smooth $X$ after performing connected-sums with specific simply-connected 4-manifolds and then argue using \Cref{lemma:proposition_naturality_topological} that we can still carry through the calculation in this case.

Let $F_x$, $k$ and $l$ be as in the non-smoothable case in \Cref{con:sigma_realisation_PI} and note that in this construction we chose a smooth structure on $X\# k (S^2\times S^2)\# l E_8$ that we will use throughout.  Consider the following diagram.

\[\begin{tikzcd}
	{\pi_0\mathcal{P}^{\TOP}(X,\partial X)} & {\Wh_2(\pi_1X)} \\
	{\pi_0\mathcal{P}^{\TOP}(\mathring{X},\partial \mathring{X})} & {\Wh_2(\pi_1X)} \\
	{\pi_0\mathcal{P}^{\TOP}(X\#k(S^2 \times S^2) \# lE_8,\partial X)} & {\Wh_2(\pi_1(X\#k(S^2\times S^2)\#lE_8))}
	\arrow[from=1-1, to=1-2, "\Sigma^{\TOP}"]
	\arrow[from=2-1, to=1-1]
	\arrow[from=2-1, to=2-2, "\Sigma^{\TOP}"]
	\arrow[from=2-1, to=3-1]
	\arrow[from=2-2, to=1-2]
	\arrow[from=2-2, to=3-2]
	\arrow[from=3-1, to=3-2, "\Sigma^{\TOP}"]
\end{tikzcd}\]

All of the vertical maps are either induced by inclusions or are forgetful maps, and $k$ and $l$ are as determined in \Cref{con:sigma_realisation_PI}.  The diagram commutes by \Cref{lemma:proposition_naturality_topological}.  Note that by the definition of our pseudo-isotopy $F_x$ we can lift/map it down to bottom-left of the diagram in such a way that we obtain a pseudo-isotopy $F^{\#}_x$ as in \Cref{con:sigma_realisation_PI} for the \emph{smooth} case.  By \Cref{lem:realisation_sigma_smooth} we have that $\Sigma^{\TOP}(F^{\#}_x)=x$ and hence, by the commutativity of the diagram, we conclude that $\Sigma^{\TOP}(F_x)=x$.
\end{proof}

\section{Realisation theorem for $\Wh_1(\pi_1(X); \Z/2 \times \pi_2(X))$}\label{sec:realisation_theta}

\begin{thm}\label{thm:realisation_theta}
    Let $X$ be a compact topological 4-manifold with good fundamental group. Then given $y \in \Wh_1(\pi_1X; \Z/2 \times \pi_2 X)$ there exists a pseudo-isotopy $F : X \times I \to X \times I$ with vanishing $\Sigma^{\TOP}$ obstruction and $\Theta^{\TOP}(F)=y$. 
\end{thm}
We know that in dimension $\ge 5$ , $$ \Theta: \ker \Sigma \to \Wh_1(\pi_1(X); \Z/2 \times \pi_2(X))/ \chi(K_3\mathbb{Z}[\pi_1 X]) $$
is surjective. 
We are only concerned in the case where $k_1(X)=0$, which is the case treated in \cite{singh2022pseudoisotopies}.
Note that in \cite{singh2022pseudoisotopies} a \emph{stable} realisation theorem for $\Theta$ is proved. Unlike the stable realisation for $\Sigma$, the result in \cite{singh2022pseudoisotopies} requires only a single stabilisation. 
We will prove here a realisation theorem for $\Theta^{\TOP}$ that does not require stabilisation. In order to do so, we will follow the steps in \cite{singh2022pseudoisotopies}. During the process, we will find an immersed geometrical dual sphere $P$; while in the smooth case there is some necessary work to be done to embed the sphere---hence the stabilisation---in the topological case we do not need to find an embedded sphere. Instead we appeal to the disc embedding theorem \cite{behrens_kalmar_kim_powell_ray_2021} using the immersed $P$ to get the desired Whitney disc.
As we did in \Cref{sec:realisation_sigma}, we will produce $F$ by building an allowed one-parameter family of handle decompositions $\mathcal{F}_t$ that 'potentially' realises the given element $y \in \Wh_1(\pi_1 X; \Z/2 \times \pi_2(X))$. This family $\mathcal{F}_t$ is first of all created so that $\Sigma^{\TOP}(F_y)=0$. In parallel with smooth Cerf theory, one can think of this family as a 'nested-eyes' Cerf family.  We will consider two different cases, depending whether $X$ is smoothable or non-smoothable. We begin with the smoothable case, and then move to the non-smoothable case.

\subsection{Producing the pseudo-isotopy}\label{sbs:producing_pseudo_isotopy_theta}

Assume $X$ is a compact smoothable 4-manifold with good fundamental group and let $y \in \Wh_1(\pi_1X;\Z/2 \times \pi_2(X))$. Pick a smooth structure for $X$ and fix it once and for all. As in the smooth realisation for $\Theta$ (see \cite{hatcher_wagoner_1973,singh2022pseudoisotopies}) we choose elements $\alpha\in \Z/2\times \pi_2(X)$ and $ \gamma\in \pi_1(X)$ so that the sum $\alpha\gamma\in (\Z/2\times \pi_2(X))[\pi_1(X)]$ represents $y$ and then we begin building a one-parameter family $(g_t, \eta_t)$ which has one index 2 critical point and one index 3 critical point and no handle slides.\footnote{We will just work with a single handle pair; this is sufficient to realise all possible elements.} Note that following \Cref{lem:smooth_top_handle_decomposition} we have all the data necessary to describe our one-parameter family of handle decompositions $\mathcal{F}_t$ except for how these two critical points cancel, i.e.\ the death.  Our goal is to instead cancel them in a non-trivial way which encodes the data from $y$, by first performing an isotopy of the attaching sphere of the 3-handle.  We show that the smooth realisation process of $y$, which requires (a single) stabilisation to be fully carried out, can be terminated in the $\TOP$ category unstably.

We can represent the $\pi_2$ part of $\alpha$ by an immersed 2-sphere $S \subset X$ (with an arc $\mu$ to the base point) and we can arrange that $S$ has only transverse double points of self-intersection.

    The family starts with the trivial handle structure on $X \times I$ and we generate a cancelling 2-3 handle pair. We make sure the birth region is disjoint from the sphere $S$. We then consider the middle-level $V:=X \# (S^2 \times S^2)$. We denote the 2-handle belt sphere and the 3-handle attaching sphere by $A$ and $B$ in $V$ respectively. At the beginning these are $S^2\times \{p\}$ and $\{q\}\times S^2$ in $X \# (S^2 \times S^2)$. Note we also have dual spheres $A^*$ \and $B^*$ (see \cite[Section 7.3.1]{singh2022pseudoisotopies}. Everything here is disjoint from $S$.
Since $A \cup B$ is $\pi_1 $-negligible in $V$, $\gamma$ determines a finger move from $A$ to $B$. By simple dimension arguments, we can ensure the finger move does not hit $S$ and the dual spheres $A^*$ and $B^*$ . We do this finger move to obtain a new handle structure. Afterwards, it is clear we see the Whitney disc $U$ that undoes the move. The interior of $U$ is disjoint from $A$,$B$, $A^*$ and $B^*$.
    
    This is our setup. Now our goal is to tube $U$ into the sphere $S$ to create a new Whitney disc $W$, make $W$ embedded and perform a Whitney move. We will do this in such a way that $U \cup W$ represents $\alpha$ and  so that $W$ has the correct framing (in \cref{def:Theta} we combined the framing value in $\Z/2$ and the $\pi_2(X)$ element into a single element $\alpha$, which is a slightly different convention from the one used in \cite{singh2022pseudoisotopies}).
    Depending on the values of $\alpha$ we have four different cases to consider. Three of them can be dealt as in \cite[Proposition 8.1]{singh2022pseudoisotopies} without appealing to the disc embedding theorem, since all of the discussion in \cite[Proposition 8.1]{singh2022pseudoisotopies} produces a smooth pseudo-isotopy and then we can then appeal to \Cref{lemma:smooth_is_equal_top}.

    The last case left to realise is when $\omega_2^X(\alpha)=0$ and we have non-trivial framing. In the smooth realm, this requires a single stabilisation to be realised.
    We basically follow the proof in \cite[Proposition 8.2]{singh2022pseudoisotopies} up until when we can make use of the disc embedding theorem, i.e.\ \Cref{lem:smooth_to_top_handle_cancellation} to bypass the smooth problems and avoid the need for an extra $S^2 \times S^2$ in the middle-middle level.

    First of all, to get the correct value for the framing of $\alpha$, we perform a single boundary twist of $W$ along the boundary componenet of $W\cap B$.  We then perform another boundary twist on $W$ with $A$ (an opposite twist), so that $W$ remains a framed Whitney disc. Note that $W$ intersects $A$ and $B$. Using the \emph{Norman trick} on $A^*$ we can resolve the intersection between $A $ and $W$. Similarly we can perform the Norman trick on $B^*$ to resolve the single intersection between $B$ and $W$.
    This procedure adds a single self-intersection of $W$ (as the parallel copy of $A^*$ intersects the parallel copy of $B^*$).

    We now consider $T$, the \emph{Clifford torus} in $V$ for the Whitney disc $W$.  This Clifford torus $T$ intersects $W$ in exactly one point and it is disjoint from $A, B, A^*, B^*$.
    \begin{figure}[htb!]
\centering
\begin{tikzpicture}
\node[anchor=south west,inner sep=0] at (0,0){\includegraphics[width=15cm]{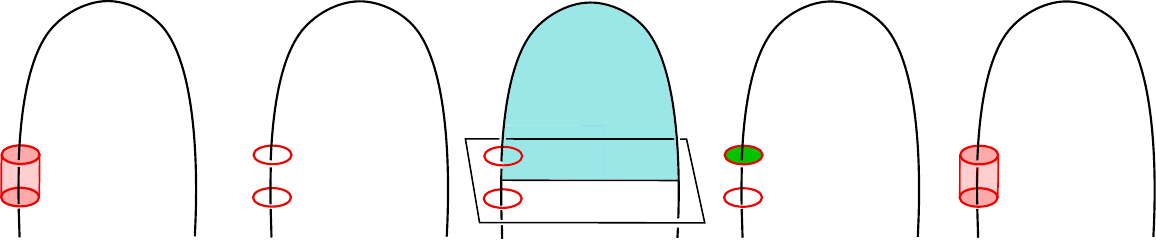}};
\node at (7.6,1.8){W};
\node at (8.5,3.3){A};
\node at (0.8,0.8) {T};
\node at (8,0.5) {B};
\node at (10.2,1){$D_A$};
\end{tikzpicture}
\caption{Movie of the 4-dimensional space. In the middle time slice we see the Whitney disc $W$ and a disc subset of the sphere $B$. In each slice we see a line from $A$ that sweeps a disc subset of $A$. The Clifford torus is denoted by $T$; it intersects $W$ in exactly one point but does not intersect $A$ and $B$. The disc $D_A$ in the fourth slice is the one of the two caps. To see the other one we can just re-draw the picture with the roles of $A$ and $B$ reversed. We credit the figures of this section to \cite{singh2022pseudoisotopies}.}
\label{figure:CliffordTorus}
\end{figure}
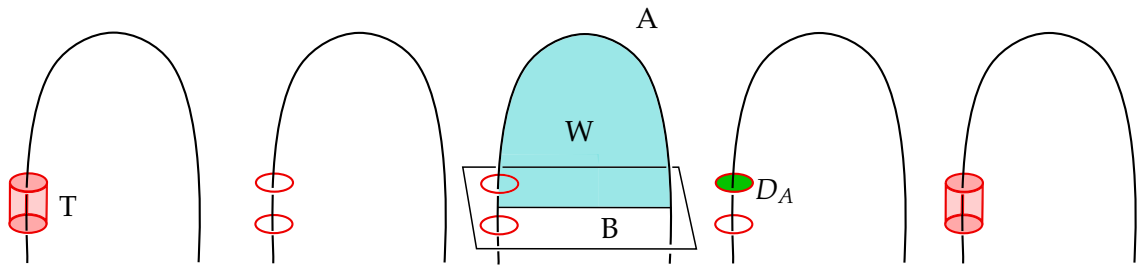

There are two caps for $T$, i.e. embedded discs $D_A$ and $D_B$ which intersect $T$ on $\partial D_A$ and $\partial D_B$. The disc $D_A$ intersects $A$ in exactly one point but it is disjoint from $W, B, A^*,B^*$ and the disc $D_B$ intersects $B$ in exactly one point but it is disjoint from $W,A,A^*,B^*$. Moreover, $D_A,D_B$ intersect in a single point on their boundary.

 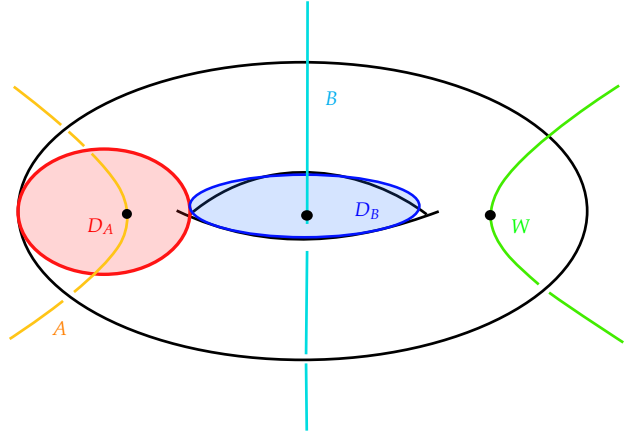
\begin{figure}[htb!]
 \resizebox{0.55\linewidth}{!}{
\centering
\begin{tikzpicture}
\node[anchor=south west,inner sep=0] at (0,0){\includegraphics[width=12cm]{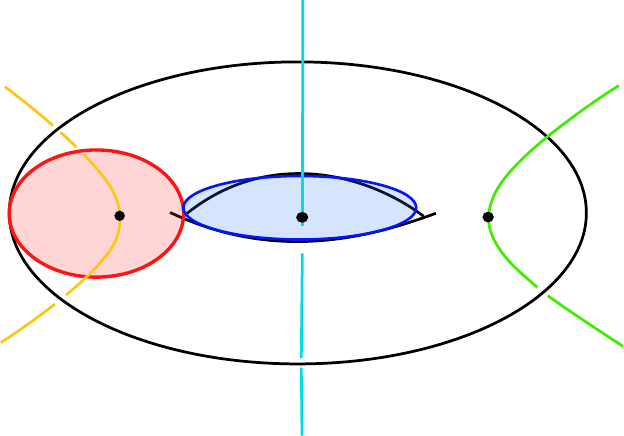}};
\node at (1.8,4) {$\textcolor{red}{D_A}$};
\node at (7,4.3){$\textcolor{blue}{D_B}$};
\node at (6.3,6.5){$\textcolor{cyan}{B}$};
\node at (10,4){$\textcolor{green}{W}$};
\node at (1,2){$\textcolor{orange}{A}$};
\end{tikzpicture}}
\caption{The Clifford torus with the caps. Note that $A$ and $B$ do not intersect $T$ and that $W$ only intersects $T$ in a single point.}
\label{figure:Torus}
\end{figure}

 \begin{figure}[htb!]
\centering
\begin{tikzpicture}
\node[anchor=south west,inner sep=0] at (0,0){\includegraphics[width=12cm]{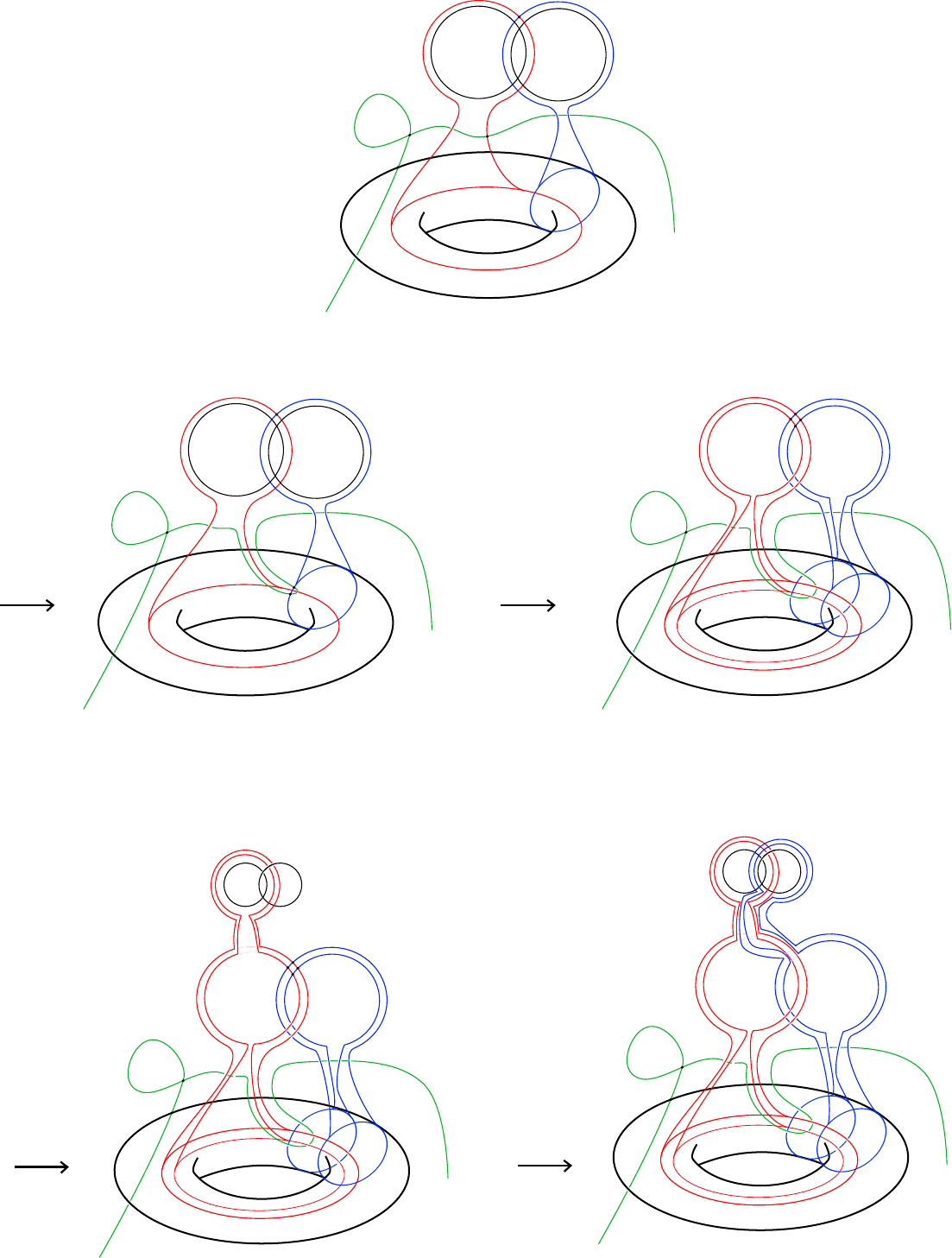}};

\end{tikzpicture}
\caption{We picture the manipulation of the Clifford caps.}
\label{figure:ModifyingCaps}
\end{figure}

We now modify the caps (see Figure \ref{figure:ModifyingCaps}).  We remove the intersection of $D_A$ and $A$ using the Norman trick, tubing it into $A^*$. We similarly remove the intersection of $D_B$ and $B$, tubing into $B^*$, noting that this adds intersections with $W$, and with the dual spheres $A^*$ and $B^*$, and adds a single point of intersection between $D_A$ and $D_B$.

All of this is done so we can now perform the \emph{symmetric capping operation} introduced in \cite[2.3]{freedman_quinn_1990} using the resulting two discs $D_A$, $D_B$. To do so, we remove a neighbourhood of $\partial_A \cup \partial_B$ from $T$ and glue back two parallel copies of $D_A$, two parallel copies of $D_B$ and glue back in a square in $T$ around $\partial D_A \cap \partial D_B \in T$ to fill the resulting hole (see Figure \ref{figure:symmetric_capping}). 
\begin{figure}[htb!]
\centering
\begin{tikzpicture}
\node[anchor=south west,inner sep=0] at (0,0){\includegraphics[width=12cm]{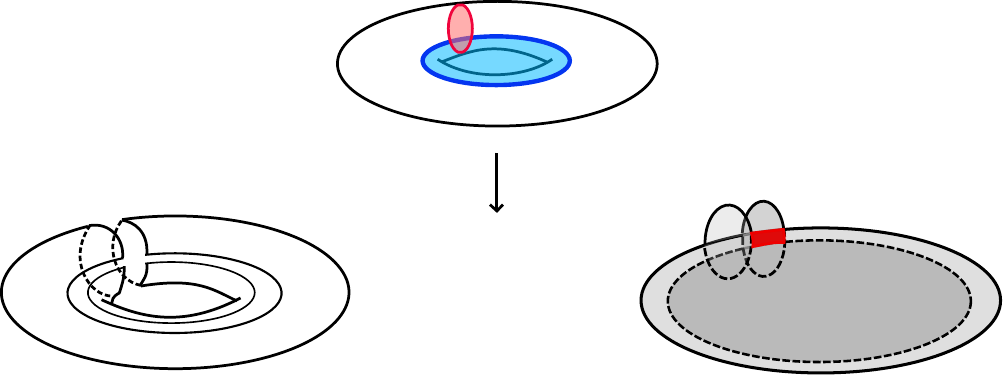}};

\end{tikzpicture}
\caption{The symmetric capping operation}
\label{figure:symmetric_capping}
\end{figure}

We smooth the edges of the resulting sphere and denote the result of this operation $P$.
Note that $P$ is an \emph{immersed} sphere with double points. Moreover, as of now $P$ intersects $W$ in more than one point---we did introduce more intersections during our operations.
However, we can remove the extra intersections between $W$ and $P$ as follows. Start with those between $D_B$ and $W$. Push down all such intersections into $T$ so that the intersection points of $W$ and $T$ lie on $\partial D_A$. Now when we perform symmetric capping to get $P$, $W$ will not intersect $P$ in there (see Figure \ref{figure:removingIntersections}). Same can be done for intersections between $D_A$ and $W$. 
\begin{figure}[htb!]
\centering
\begin{tikzpicture}
\node[anchor=south west,inner sep=0] at (0,0){\includegraphics[width=12cm]{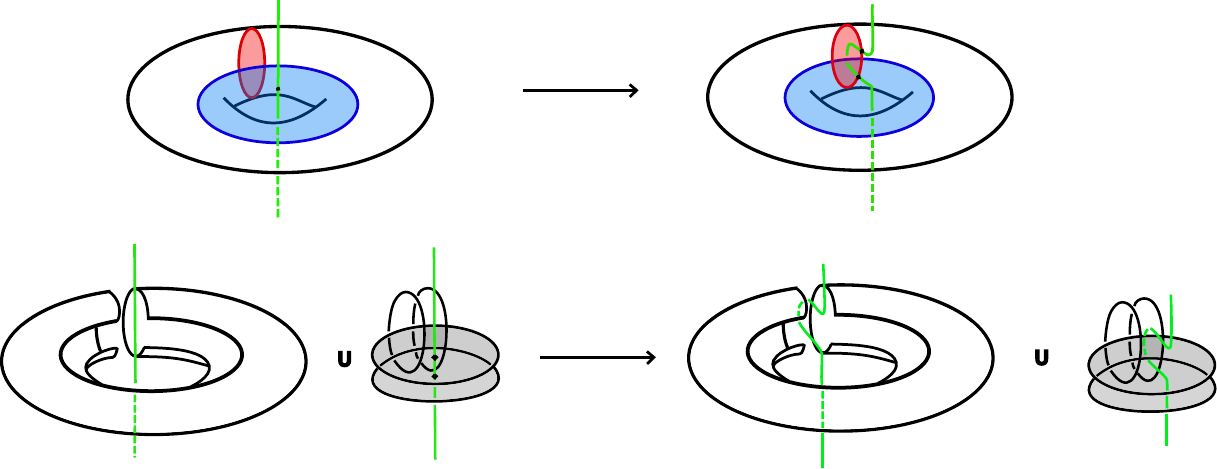}};

\end{tikzpicture}
\caption{Removing intersections between $P$ and $W$. }
\label{figure:removingIntersections}
\end{figure}
At the end, $P$ is an immersed sphere that is geometrically dual to $W$. Here is where our construction diverges from \cite{singh2022pseudoisotopies}. 

We claim that we are in exactly in the position to apply \Cref{lem:smooth_to_top_handle_cancellation}.  In the middle level, the $A$ sphere and the $B$ sphere are algebraically cancelling with precisely two excess points of intersection, and we have a smoothly immersed, framed Whitney disc $W$ which pairs up the double points and $W$ has an immersed geometric dual sphere $P$.  We can now apply \Cref{lem:smooth_to_top_handle_cancellation} and then \Cref{lem:top_handle_cancellation} to complete the one-parameter family of handle decompositions\footnote{This one-parameter family of handle decompositions `looks like' it should realise the correct invariant, since we used the correct $\alpha$ and $\gamma$ and the right framing in the construction, the latter because we performed a
 single boundary twist with $B$.}, as in the construction in \Cref{sbs:producing_PI}.

We now describe how to produce the corresponding $\mathcal{F}_t$ when $X$ is not smooth. This will be done as in \cref{sec:realisation_sigma}. 
By \cite[Theorem 8.6, Proof of Theorem 10.1]{freedman_quinn_1990}, $X\# l E_8\# k (S^2\times S^2)$ is smoothable for some $l=0,1$ and $k\geq 0$. We pick a smooth structure on $X\# l E_8\# k (S^2\times S^2)$ that we fix once and for all.  By topological transversality \cite[Essay III, Theorem 1.5]{kirby_siebenmann_1977} we can assume that $N\subset X\times I$ is disjoint from a neighbourhood of $\{\pt\}\times I\subset X\times I$ and then we form the connected sum along an interval $X\# l E_8\# k (S^2\times S^2) \times I$.  This can be equipped with the product smooth structure that we then restrict to $N$.

Since we picked $N$ to be an open neighbourhood of the 3-handle-skeleton, this means that we have inclusion induced isomorphisms $\pi_1(N)\cong \pi_1(X)$ and $\pi_2(N) \cong \pi_2(X) $.  So, after picking $y\in \Wh_1(\pi_1(X);\Z/2 \times \pi_2 (X))$, on $N$ we can build a one-parameter family just as in the smoothable case which `corresponds' to $y$.  Note that since $N$ is not a product of a 4-manifold with an interval we are not producing a pseudo-isotopy of $N$, but we are producing a one-parameter family of topological handle decompositions $\mathcal{F}_t$ for $X\times I$, since we can build the data for the one-parameter family of homeomorphisms $F_t\colon X\times I\to X\times I$ by extending the one-parameter family via the trivial handle decomposition.  In particular, this means that all births 
and deaths and moves dictated by $y$  occur inside $N\subset X\times I$.

\begin{con}\label{con:allowed_1_for_theta}
    Given a compact 4-dimensional manifold $X$ with good fundamental group and $y \in \Wh_1(\pi_1X; \Z/2 \times \pi_2(X))$ we carry out the above procedure, making choices where necessary, to produce an allowed one-parameter family of handle decompositions which we denote by $\mathcal{F}_t(y)$.  By \cref{sec:handle_decompositions} this defines a pseudo isotopy $F_y$.
\end{con}
\subsection{An alternative viewpoint on stable surjectivity of $\Theta$.}

To prove \Cref{thm:realisation_theta}, we will compare our realisation of $\mathcal{F}_t$ to the stable smooth realisation of $\Theta$. Instead of using Singh's strategy \cite{singh2022pseudoisotopies}, we use again the stable surface smoothing theorem of Cha-Kim \cite{cha_kim_2023}. This approach may require more than one stabilisation, but has the advantage that this stabilized pseudo-isotopy will be much easier to compare with our one-parameter family of handle decompositions $\mathcal{F}_t$ and $F_X$.
\begin{lemma}\label{lem:stably_pseudo_isotopic_theta}
    Let $X$ be a compact, smooth 4- manifold and let $y \in \Wh_1(\pi_1(X);\Z/2 \times \pi_2(X))$. Let $\mathcal{F}(y)$ be the one-parameter family of topological handle decompositions associated to $y$, with induced pseudo-isotopy $F_y: X \times I \to X \times I$. Then there exists a smooth pseudo-isotopy $G_y \colon X \# k(S^2 \times S^2) \times I \to X\# (S^2 \times S^2)$ with $\Theta(G_y)=y$ and $F_y$ is stably isotopic to~$G_y$.
\end{lemma}

\begin{proof}
    The proof is essentially similar to that of \Cref{lem:stable_smooth_top_comparison_sigma}. 
    First we remark that for some elements $y$ we have unstable smooth realisation: the one-parameter family that is created smoothly gives us the candidate $\mathcal{F}_t(y)$, hence the proof of the lemma follows from \Cref{thm:main} for such elements.
    For the elements that require a stable statement, instead of using Singh's method, we work with the stable surface smoothing \cite{cha_kim_2023}. We build $\mathcal{F}_t(y)$ using the disc embedding theorem to produce a topologically embedded Whitney disc $\mathcal{W}$ which is correctly framed. We then perform the Whitney move using this disc. 
    Due to Cha-Kim \cite[Theorem D]{cha_kim_2023}, after stabilisations, there exists a topological isotopy~$H$, relative to the boundary of $\mathcal{W}$ that takes $\mathcal W$ to a smoothly embedded $\mathcal{W}'$. Thus we can terminate the proof by performing the Whitney move smoothly along this disc, producing a pseudo-isotopy $G_y$ so that $\Theta(G_y)=y$.

    We now show that $F_y$ is stably isotopic to $G_y$. Let $G_t$ be the one-parameter family of topological handle decompositions induced by the one-parameter family $(g_t,\eta_t)$ associated to $G_y$. First, choose an isotopy of $\mathcal{F}_t$ that makes the embedding $F^{-}_t$ fixes a $D^4 \times I \subset X \times I$ throughout (the $D^4\times I$ is then away from handle births, deaths and the isotopy in the middle level of the attaching sphere of the 3-handle).
    This allows to define a stabilisation $\mathcal{F}_t^{\#}$ which induces a stabilised pseudo-isotopy $$ F_y^{\#}\colon (X \# k(S^2 \times S^2))\times I \to (X \#k(S^2 \times S^2))\times I$$ stably isotopic to $F_y$. One can arrange the data for $\mathcal{F}_y^{\#}$ and $\mathcal{G}_t$ to be completely identical up until the point that we need to perform the Whitney move in the middle level. Via isotopy extension, $H$ provides the isotopy between the rest of the data for $\mathcal{F}_y^{\#}$ and $\mathcal{G}_t$. In particular, restricting this isotopy to the end of the one-parameter families produces the isotopy between $F^{\#}_y$ and $\mathcal{G}_y$.
\end{proof}
\subsection{Proof of Theorem \ref{thm:realisation_theta}}
We now give the proof of \Cref{thm:realisation_theta}. We will start by working with smoothable manifolds.
\begin{lemma}\label{lem:smoothable_theta}
 Let $X$ be a compact, smooth 4-manifold and $y \in \Wh_1(\pi_1(X);\Z/2 \times \pi_2(X))$. Let $F_y$ be as in \ref{con:allowed_1_for_theta}. Then $\Theta^{\TOP}(F_y)=y$.   
\end{lemma}
\begin{proof}
Let $X$ be a smooth 4-manifold.
We write down the following diagram, which mimics the one in Section \ref{sec:realisation_sigma}.
Let $\mathring{X}$ be $X$ with a small open ball removed. Consider the following.
\begin{center}

\begin{tikzcd}
	{\ker \Sigma^{\TOP}_{X}} && {\Wh_1(\pi_1X;\Z/2 \times \pi_2X)/\chi} \\
	{\ker \Sigma^{\TOP}_{\mathring{X}} } &&{\Wh_1(\pi_1\mathring{X}; \Z/2 \times \pi_2\mathring{X})/\chi} \\
	{\ker \Sigma^{\TOP}_{X\#k(S^2\times S^2)}} && {\Wh_1(\pi_1X\#k(S^2 \times S^2);\Z/2\times \pi_2(X\#k(S^2 \times S^2))/\chi} \\
	{\ker \Sigma_{X\#k(S^2\times S^2)}} && {\Wh_1(\pi_1X\#k(S^2 \times S^2);\Z/2\times \pi_2(X\#k(S^2 \times S^2))/\chi}
	\arrow["\Theta^{\TOP}", from=1-1, to=1-3]
	\arrow[from=2-1, to=1-1]
	\arrow["\Theta^{\TOP}", from=2-1, to=2-3]
	\arrow[from=2-1, to=3-1]
	\arrow[from=2-3, to=1-3]
	\arrow[from=2-3, to=3-3]
	\arrow["\Theta^{\TOP}", from=3-1, to=3-3]
	\arrow[from=4-1, to=3-1]
	\arrow["\Theta", from=4-1, to=4-3]
	\arrow[from=4-3, to=3-3]
\end{tikzcd}
\end{center}
The  vertical arrows on the left are given by extension via the identity or forgetful maps.

To prove commutativity of the top two squares we use \Cref{lemma:proposition_naturality_topological}.
The bottom square commutes thanks to \Cref{lemma:smooth_is_equal_top}.  Using \Cref{lem:stably_pseudo_isotopic_theta} and a diagram chase involving the above diagram we have $$ \Theta^{\TOP}(F_y)=\Theta^{\TOP}(F_y^{\#})=\Theta(G_y)=y,$$
where $F^{\#}_y$ is the stabilisation of $F_y$ compatible with $G_y$, as in the proof of Lemma \ref{lem:stably_pseudo_isotopic_theta}.
\end{proof}
\begin{proof}[Proof of \Cref{thm:realisation_theta}]
Lemma \ref{lem:smoothable_theta} provides the proof when $X$ is smoothable, so we can drop this assumption. 
We will instead show that we can smooth $X$ after some specific connected-sums, as we did in the proof of \Cref{thm:realisation_sigma}.
Let $F_y, k $ and $l$ be as in the non-smoothable case in \Cref{con:allowed_1_for_theta}. Note that in this construction we chose a smooth structure on $X\#k(S^2 \times S^2) \#lE_8$ that we will use throughout.
Consider the following diagram:

\[\begin{tikzcd}
	{\ker \Sigma^{\TOP}_X} & {\Wh_1(\pi_1X;\Z/2 \times \pi_2X)/ \chi(K_3\mathbb{Z}[\pi_1X]}) \\
	{\ker \Sigma^{\TOP}_{\mathring{X}}} & {\Wh_1(\pi_1X)/\chi(K_3\mathbb{Z}[\pi_1\mathring{X}]}) \\
	{\ker \Sigma^{\TOP}_{X\#k(S^2 \times S^2) \# l E_8}} & {\Wh_1(\pi_1(X\#k(S^2\times S^2)\#lE_8))/\chi(K_3\mathbb{Z}[\pi_1 (X \#k(S^2 \times S^2) \# lE_8)])} 
	\arrow[from=1-1, to=1-2]
	\arrow[from=2-1, to=1-1]
	\arrow[from=2-1, to=2-2,]
	\arrow[from=2-1, to=3-1]
	\arrow[from=2-2, to=1-2, "\cong"]
	\arrow[from=2-2, to=3-2]
	\arrow[from=3-1, to=3-2]
\end{tikzcd}\]
All the vertical maps are either inclusions or forgetful maps, and $k$ and $l$ are as fixed in \Cref{con:allowed_1_for_theta}. The diagram commutes by \Cref{lemma:proposition_naturality_topological}. Note that by definition of $F_y$ we can lift/map it down to the bottom left of the diagram so that we obtain a pseudo-isotopy $F_y^{\#}$ as in \Cref{con:allowed_1_for_theta} for the \emph{smooth} case. By Lemma \ref{lem:smoothable_theta} we have $\Theta^{\TOP}(F_y^{\#})=y$ and hence, by commutativity, we conclude $\Theta^{\TOP}(F_y)=y$.\qedhere

\end{proof}

\section{Inertial pseudo-isotopies and duality}\label{sec:duality}

Given a homeomorphism $f \in \Homeo(X, \partial X)$ which is pseudo-isotopic to the identity, we may try to obstruct it from being isotopic to the identity by picking a pseudo-isotopy $F \in \mathcal {P}$ from $\Id$ to $f$ and then computing the invariants $\Sigma^{\TOP}$ and $\Theta^{\TOP}$. However, $\Sigma^{\TOP}(F)$ and $\Theta^{\TOP}(F)$ depend on the choice of the specific pseudo-isotopy. It is perfectly possible that there exists a different pseudo-isotopy, say $G$, for which both obstruction vanish.

We note that the composition $F \circ G^{-1} $ is still a  pseudo-isotopy, and moreover it fixes the whole boundary of $X \times I$. This motivates the following definition:
\begin{definition}[Inertial pseudo-isotopy]\label{def:intertial_pseudo_isotopy}
    Let $F \in \mathcal{P}^{\TOP}(X,\partial X)$ be a pseudo-isotopy with $F\vert_{X\times 1}=id_X$. In fact, this forces $F\vert_{\partial(X \times I)}= \Id\vert_{\partial (X\times I)}$. We say that $F$ is an \emph{inertial} pseudo-isotopy, and denote the set of inertial pseudo-isotopies $\mathcal{J}^{\TOP}(X, \partial X) \subset \mathcal{P}^{\TOP}(X,\partial X)$.
\end{definition}
 To have a well defined obstruction for a homeomorphism pseudo-isotopic to the identity to be isotopic to the identity we need to mod out by the invariants of the inertial pseudo-isotopies. See, for example, \cite{singh2022pseudoisotopies} for a detailed explanation on why the quotient map is well defined. 
In general, the images $\Sigma(\mathcal{J}(X, \partial X))$ and $\Theta(\mathcal{J}(X,\partial X) \cap \ker \Sigma)$ are extremely hard to compute. In some cases, e.g.\ when our manifold is of the form $Y^3\times I$, we can say much about these two subgroups. In order to do so, we recall Hatcher-Wagoner's involution on the space on the space of pseudo-isotopies.

\begin{definition}[Dual pseudo-isotopy]\label{def:dual_pseudo_isotopy}
    Let $F$ be a pseudo-isotopy. Denote the reflection map on $X\times I$ which sends $(p,s)$ to $(p,1-s)$ by $r$. We define the \emph{dual} pseudo-isotopy to $F$ to be: $$ \overline{F}= \left( (F|_{X \times \{1\}})^{-1} \times id_{I}\right) \circ r \circ F \circ r.$$
\end{definition}
One can check that $\overline{F}$ is indeed a pseudo-isotopy and that it connects $\Id$ with $F|_{X \times 1}^{-1}:= F_1^{-1}$.  Furthermore, one may verify that this involution is well-defined on $\pi_0(\mathcal{P}(X,\partial X))$, though we leave this as an exercise.

Our goal is to use this dualisation construction to obtain structural information about $\mathcal{J}^{\TOP}(X,\partial X) \cap \ker \Sigma^{\TOP}$.
Smoothly, this is encoded via a duality formula.  However, the proof of this formula uses Cerf's approach to studying pseudo-isotopies; i.e.\ the strong connection between the pseudo-isotopy space and one-parameter families of generalised Morse functions.  Since we do not have this correspondence, we will need to work harder to prove a corresponding formula.  The central idea will be to go through the definition of our topological invariants and track how the involution behaves throughout the process.

\begin{proposition}{[c.f.\ \cite[I, Chapter VIII; II, \S 4]{hatcher_wagoner_1973}]}\label{proposition:duality_inertial}
   Let $X$ be a compact topological 4-manifold with $k_1(X)=0$, and let $F \in \mathcal{P}(X)$ be a pseudo-isotopy. Then: $$\Sigma^{\TOP}(\ol{F})= \overline{\Sigma^{\TOP}(F)}$$ and, if $F \in \ker \Sigma$, then 
    $$ \Theta^{\TOP}(\ol{F})=\overline{\Theta(F)},$$
    where the involutions on $\Wh_2(\pi_1X)$, $\Wh_1(\pi_1X; \Z/2 \times \pi_2X)$ are defined\footnote{See \cite[Part I, Chapter VIII; Part II, \S 4]{hatcher_wagoner_1973} for more details.} to be those induced by the involution on the group ring $\Z[\pi_1(X)]$ that sends \[\sum_{g\in \pi_1(X)}\lambda_{g}g\mapsto \sum_{g\in \pi_1(X)}w_1(g)\lambda_{g}g^{-1}\] for $\Wh_2(\pi_1(X))$, and induced by the involution on the group ring $(\Z/2\times \pi_2(X))[\pi_1(X)]$ that sends \[\sum_{g\in \pi_1(X)}(\lambda_{g}, \mu_g)g^{-1}\mapsto \sum_{g\in \pi_1(X)}\left(\lambda_{g}+w_2(\mu_g),-w_1(g)(g^{-1}\cdot\mu_{g})\right)g \] (here $w_1$ and $w_2$ denote the first and second Steifel-Whitney classes\footnote{In \Cref{sec:homeos} the manifolds we consider are spin, and hence these involutions greatly simplify due to the vanishing of $w_1$ and $w_2$.}).
    \begin{remark}
        We stress that we are working with 4-manifolds with $k_1 X=0$. It is a priori not clear that the involution respects $\chi$ and so we would need to take a further quotient of $\Theta$.  Conjecturally, this is unnecessary (see \cite[Conjecture 9.7]{singh2022pseudoisotopies} and the surrounding discussion for more details).
    \end{remark}
\end{proposition}

Before we prove \Cref{proposition:duality_inertial}, we establish how the involution interacts with the suspension map.

\begin{lem}{\cite[Appendix I, Lemma]{hatcher_concordances}}\label{lem:suspension_involution}
    Let $X$ be a compact $\CAT$ manifold and let $F\in \mathcal{P}^{\CAT}(X,\partial X)$.  Then $\ol{S^{+}(F)}=-S^{+}(\ol{F})$ in $\pi_0\mathcal{P}^{\CAT}(X\times J,\partial (X\times J))$.
\end{lem}

\begin{proof}
    Although the proof is given in the attribution, we spell out some initial details.  Firstly, the proof proceeds by first passing to a modification of the pseudo-isotopy space.  We denote by $\mathcal{I}^{\CAT}(X)$ the space of $\CAT$ automorphisms of $X\times I$ which send $X\times \{0\}$ to $X\times\{0\}$ and restrict to isotopies along $\partial X\times I$ modulo isotopies of $X\times I$.  Hatcher claims that the inclusion map $\mathcal{P}^{\CAT}(X,\partial X)\to \mathcal{I}^{\CAT}(X)$ is a homotopy equivalence.  We construct the homotopy inverse.

    Let $G\in \mathcal{I}^{\CAT}(X)$.  One can verify that $(G_0^{-1}\times \Id)\circ G \in \mathcal{P}^{\CAT}(X)$.  Since $G$ originally restricted to an isotopy on $\partial X\times I$, we can now `push out' the identity map along $\partial X\times I$ using the reverse of this isotopy, and further use of isotopy extension in a collar results in a $\CAT$ pseudo-isotopy $G'\in \mathcal{P}^{\CAT}(X,\partial X)$.  We have now defined a map $\mathcal{I}^{\CAT}(X)\to \mathcal{P}^{\CAT}(X,\partial X)$ which we claim is the homotopy inverse.  We leave verifying this to the reader.

    To prove the lemma it now suffices to show it for the involution on $\mathcal{I}^{\CAT}(X,\partial X)$.  On $\mathcal{I}^{\CAT}(X,\partial X)$, note that the involution is now equal to the map which conjugates by the reflection map $r$.  Now the interested reader is equipped to read the proof in \cite{hatcher_concordances}.
\end{proof}

\begin{proof}[Proof of \Cref{proposition:duality_inertial}]
We will prove the formula for $\Sigma^{\TOP}$ (the proof for $\Theta^{\TOP}$ is analogous).
Let $F \in \mathcal{P}^{\TOP}(X,\partial X)$. We consider the following diagram:
\[\begin{tikzcd}[cramped]
	{\pi_0(\mathcal{P}^{\TOP}(X \times J^2, \partial (X\times J^2))} & {\pi_0(\mathcal{P}^{\TOP}(X \times J^2, \partial (X\times J^2))} \\
	{\pi_0(\mathcal{P}^{\TOP}(N,\partial N))} & {\pi_0(\mathcal{P}^{\TOP}(N,\partial N))} \\
	{\pi_0(\mathcal{P}^{\DIFF}(N,\partial N))} & {\pi_0(\mathcal{P}^{\DIFF}(N,\partial N))}
	\arrow["{{\ol{\cdot}}}", from=1-1, to=1-2]
	\arrow["{\cong,\ K}"', curve={height=24pt}, dashed, from=1-1, to=2-1]
	\arrow["{\cong,\ \ol{K}}", curve={height=-24pt}, dashed, from=1-2, to=2-2]
	\arrow[from=2-1, to=1-1]
	\arrow["{{\ol{\cdot}}}", from=2-1, to=2-2]
	\arrow[from=2-2, to=1-2]
	\arrow["\cong"', from=3-1, to=2-1]
	\arrow["{{\ol{\cdot}}}", from=3-1, to=3-2]
	\arrow["\cong", from=3-2, to=2-2]
\end{tikzcd}\]
The map $K$ on the left is given by straightening concordances (see \cref{lem:3_skeleton_isomorphism}) and 
the map $\ol{K}$  is defined as an isotopy $\ol{K}\colon (X \times J^2 \times I) \times I \to (X\times J^2 \times I) \times I$ which is given by dualising $K$ on the last $I$ coordinate---treating this isotopy as a pseudo-isotopy---using the $\ol{\cdot}$ construction on $K$.
We now prove commutativity of the diagram. 
The topmost square commutes by definition of $\ol{K}$. 
The bottom square commutes trivially, since forgetting the smooth structure and dualizing is the same as dualizing and forgetting the smooth structure.

Let $G\in \pi_0(\mathcal{P}^{\DIFF}(N,\partial N))$ be the resulting smooth pseudo-isotopy constructed in the definition of $\Sigma^{\TOP}(F)$. We conclude that
\[
\ol{\Sigma^{\TOP}(F)}=\ol{\Sigma(G)}=\Sigma(\ol{G})=\Sigma^{\TOP}(\ol{S^+(S^+(F))})=\Sigma^{\TOP}(\ol{F})
\]
where the first equality is via the definition of $\Sigma^{\TOP}$, the second is from the smooth duality formula \cite[Part I, Chapter VIII]{hatcher_wagoner_1973}, the third is via the commutativity of the above diagram,\footnote{Note the abuse of notation $\Sigma^{\TOP}(\ol{S^+(S^+(F))})$; as defined in \Cref{sec:def_top_obstructions} this would mean first suspend $\ol{S^+(S^+(F))}$ twice, but instead we skip this suspension step since this is already a 6-dimensional pseudo-isotopy.} and finally the last is given by two applications of \Cref{lem:suspension_involution}.\qedhere
\end{proof}

\section{Application to homeomorphisms of $Y^3 \times I$}\label{sec:homeos}

In this section, we will use the realisation theorem (\Cref{thm:realisation}) and the duality formula \Cref{thm:duality} to produce homeomorphisms of smooth 4-manifolds of the form $M^3 \times I$ which are topologically pseudo-isotopic to the identity, but not topologically isotopic to the identity.  We will only consider the second obstruction, $\Theta^{\TOP}$, in this section.  

\subsection{Inertial pseudo-isotopies of $Y^3 \times I$.}
In general, the duality formula \Cref{thm:duality} is not enough to be able to sufficiently control $\Theta^{\TOP}(\pi_0\mathcal{J}^{\TOP}(X,\partial X))$.  Hatcher shows in \cite[Part II, Lemma 5.3]{hatcher_wagoner_1973} that if we restrict to the case that $X=M\times I$, then we can say more.  To describe this,  we note that there is a differential defined on $\Wh_2(\pi_1X) \oplus \Wh_1(\pi_1X;\Z/2\times \pi_2 X)$ given by $d_i(x)= x-(-1)^i \overline{x}$. 
We can define the subgroup $Z_i= \ker d_i$, which we can also split as: \[ Z_i= Z^2_i \oplus Z^1_i\subset \Wh_2(\pi_1X) \oplus \Wh_1 (\pi_1X;\Z/2 \times \pi_2X).\]

For simplicity we write $\mathcal{J}$ for $\pi_0\mathcal{J}(X,\partial X)$ when the $X$ in question is clear.

\begin{lemma}\label{lem:duality_application}
    Let $X=Y^3\times I$ for some compact, orientable 3-manifold $Y$.  Then
    \[
    \Sigma^{\TOP}(\mathcal{J})\subset Z^2_4,\ \text{and}\ \Theta^{\TOP}(\mathcal{J}\cap \ker \Sigma^{\TOP})\subset Z^1_4.
    \]
\end{lemma}

\begin{proof}
    We refer to Hatcher's proof of the corresponding fact in the smooth case: \cite[Part II, Lemma 5.3]{hatcher_wagoner_1973}.  The proof does not use the smooth category, only the duality formula \cite[Part I, Chapter VIII; Part II, \S 4]{hatcher_wagoner_1973}, which we proved for our topological invariants in \Cref{thm:duality}.
\end{proof}

This lemma allows us to significantly control the values of $\Theta^{\TOP}$ realised by inertial pseudo-isotopies, which we now use to produce interesting homeomorphisms.

\subsection{Homeomorphisms of $Y^3 \times I$}\label{sbs:YxI}

In this section we utilize the topological invariant $\Theta^{\TOP}$, together with \Cref{thm:realisation_sigma} and \Cref{lem:duality_application} to construct homeomorphisms of $Y^3 \times I$ (where $Y^3$ is a closed 3-manifold) that are pseudo-isotopic to the identity but not isotopic to it.  Later on, in \Cref{sbs:YxS1}, we will explain how to translate this into the closed case i.e.\ homeomorphisms of $Y^3\times S^1$.  Our discussion mirrors the discussion of Singh \cite[Section 9]{singh2022pseudoisotopies}.

Recall that we have a splitting $\Wh_1(\pi_1X ; \Z/2 \times \pi_2 X) = \Wh_1 (\pi_1X;\Z/2) \oplus \Wh_1(\pi_1 X; \pi_2 X)$. Moreover, we also have $\Wh_1(\pi_1 X; \Gamma)= \Gamma[\pi_1X]/ \langle \alpha \sigma - \alpha^{\tau}\tau \sigma \tau^{-1}, \beta \cdot 1 \mid \alpha, \beta \in \Gamma, \sigma, \tau \in \pi_1 X\rangle$.  In what follows we will only consider the summand $\Wh_1(\pi_1(X);\Z/2)$, which can be described as \[\Wh_1(\pi_1X;\Z/2)\cong\Z/2[\pi_1 X]/\langle \gamma(g_1-g_2g_1g_2^{-1}), \gamma'\rangle\]
where $\gamma,\gamma'$ range over all elements in $\Z/2$ and $g_1,g_2$ range over all elements in $\pi$.  It follows that, as a $\Z/2$-vector space, we have
\[ \Wh_1(\pi_1(X);\Z/2)\cong \bigoplus_c \Z/2\]
where $c$ ranges over all non-trivial conjugacy classes in $\pi_1(X)$.
Clearly the action of $\pi_1$ on $\pi_2$ is trivial and $k_1 X \in H^3(\pi_1 X;\pi_2 X)$ is trivial as well.
We can hence consider $\Theta^{\TOP}(\ker \Sigma^{\TOP}) \subset \Wh_1(\pi_1 X;\Z/2 \times \pi_2 X)=\Wh_1(\pi_1X;\Z/2)$.

We know that \[\Theta^{\TOP}( \mathcal{J}(Y^3 \times I)) \subset Z_4((Y^3 \times I))= \{\Theta^{\TOP} \in \Wh_1(\pi_X;\Z/2)\mid\Theta^{\TOP}=\overline{\Theta^{\TOP}}\}.\]
We would like to show that, for our cases of interest, the quotient $$\left(\bigoplus_{c \in \{conj \setminus {1}\}} \Z/2 \right)/ \left(Z_4 \cap \ker \Sigma^{\TOP}\right)$$ is non-trivial, i.e.\ there exist homeomorphisms that are pseudo-isotopic but not isotopic to the identity.  Before stating the main theorem we need one definition.

\begin{definition}\label{def:ambivalent}
    We say that a group is \emph{ambivalent} if every element is conjugate to its inverse.
\end{definition}

\begin{thm}\label{thm:interesting_homeomorphisms}
Let $Y^3$ be a 3-manifold whose first $k$-invariant $k_1(Y)$ is trivial and with $\pi_1(Y)$ good and not ambivalent.  Then there exists a homeomorphism $f\colon Y\times I\to Y\times I$ which is pseudo-isotopic to the identity but not isotopic to the identity.  In particular, there exists a homeomorphism $Y\times I\to Y\times I$ which is homotopic but not isotopic to the identity.
\end{thm}

\begin{proof}
   By the discussion immediately above we have that $\Wh_1(\pi_1(Y);\Z/2)\cong \bigoplus_c \Z/2$, where $c$ ranges over all non-trivial conjugacy classes in $\pi_1(Y)$.  Since $\pi_1(Y)$ is not ambivalent, we have that there exists at least one non-trivial conjugacy class $c$ such that for every $g\in c$, $g^{-1}\notin c$.  By \Cref{thm:realisation_theta} there exists a pseudo-isotopy $F\colon Y\times I^2\to Y\times I^2$ such that $\Sigma^{\TOP}(F)=0$ and \[\Theta^{\TOP}(F)=1\cdot c\in \Wh_1(\pi_1(Y\times I);\Z/2)\subset \Wh_1(\pi_1(Y\times I);\Z/2\times\pi_2(Y\times I)).\]
    We define $f:=F_{Y\times I\times \{1\}}\colon Y\times I\to Y\times I$.  Since $\overline{c}\neq c$ by assumption, $\Theta^{\TOP}(F)\notin Z_4\cap \ker\Sigma^{\TOP}$ and hence $\Theta^{\TOP}(F)$ cannot be killed by an inertial pseudo-isotopy.  It follows that $f$ is not isotopic to the identity (but it is clearly pseudo-isotopic to the identity via $F$).
\end{proof}

\subsection{Non-ambivalent 3-manifold groups}

We now briefly investigate which 3-manifolds satisfy the condition in \Cref{thm:interesting_homeomorphisms}.  This by no means constitutes a full investigation, but serves to illustrate that our realisation applies to many 3-manifolds.  

We make some brief observations about \cref{def:ambivalent}.

\begin{remark}\label{rem:ambivalence}
    The following two properties hold.
    \begin{enumerate}
        \item A finite group is ambivalent if and only if all of its irreducible representations have real characters \cite[Page 31]{isaacs_1976}.
        \item If a group is ambivalent then its centre is ambivalent (a simple exercise).
    \end{enumerate} 
\end{remark}

We start with the finite 3-manifold groups, where we can fully say for which 3-manifolds our realisation will work.

\begin{proposition}\label{prop:ambivalence}
    All ambivalent, finite 3-manifold groups are listed below.  The rest are not ambivalent.
    \begin{enumerate}[(i)]
        \item The cyclic groups $\Z_m$ for $m=1,2$.
        \item The dicyclic group of order $8n$, $\Dic_{2n}$.
        \item The binary octahedral group $O_{48}$.
        \item The binary icosahedral group $I_{120}$ .
    \end{enumerate}
\end{proposition}

\begin{proof}
    We begin by showing that all of the listed groups are ambivalent.  For (i) the claim is clear.  For the rest we prove this using (1) in \Cref{rem:ambivalence}, i.e.\ we show that in all of the cases these groups have all characters of irreducible representations being real.  For (ii) we use the fact that $\Dic_{2n}$ and $D_{4n}$, the dihedral group of order $8n$, have isomorphic character tables \cite[Page 64]{feit_1967}.  From the standard group presentation of $D_{k}$ it is not hard to see that all dihedral groups are ambivalent, and hence all irreducible representations of $D_{k}$ have real characters.  For (iii) and (iv) we appeal to the character tables of these groups, which can be found on GroupNames at \url{https://people.maths.bris.ac.uk/~matyd/GroupNames/}.

    We now show that all other finite 3-manifold groups are not ambivalent.  To begin with, we appeal to the classification of finite 3-manifold groups which is essentially due to Hopf \cite{hopf_1926} (see also \cite{milnor_1957,orlik_1972}).  The useful fact is the following: every non-cyclic finite 3-manifold group is a central extension of one of the following groups by an even order cyclic group: dihedral groups $D_n$, the tetrahedral group $A_4$, the octahedral group $S_4$ or the icosahedral group $A_5$. 
    By (2) of \Cref{rem:ambivalence}, we see that if the 3-manifold group is a central extension by $\Z/2m$ where $m\geq 2$, then automatically the group is not ambivalent.  Hence we only have to show that the groups obtained via central extension by $\Z/2$ are not ambivalent, or were already included in the statement of the proposition.  The ones not appearing in the statement of the proposition are the following groups (here we again appeal to the classification to know that these are the only $\Z/2$ central extensions which occur). 
    \begin{enumerate}[(a)]
        \item Odd degree dicyclic groups $\Dic_{2n+1}$.
        \item The binary tetrahedral group $T_{24}$.
    \end{enumerate}

    We deal with case (a) first.  A presentation for $\Dic_{\ell}$ is given by
    \[
    \Dic_\ell = \langle a, x \mid a^{2\ell}=1, x^2=a^\ell, x^{-1}ax=a^{-1}\rangle,
    \]
    and one can calculate from this presentation that the element $x$ is order 4 and is not conjugate to $x^{-1}$ provided that $\ell$ is odd.

    For case (d) we simply appeal again to the character table, which again can be found on GroupNames at the url listed above.
\end{proof}

We now consider the general class of Seifert fibred 3-manifolds, where we can give some partial statements.

We follow Orlik \cite[Chapter 5]{orlik_1972}.  Recall that the fundamental group of a Seifert fibred 3-manifold $Y$ is determined by its Seifert invariants\[
Y = Y(b;(\varepsilon,g);(\alpha_1,\beta_1),\dots, (\alpha_r,\beta_r))
\] where $\epsilon\in\{o_1,o_2,n_1,n_2,n_3,_n4\}$ is a ``meta-variable" encoding which of six types $Y$ is, depending on the genus $g$ and orientability of the base orbifold of the Seifert fibration; $b$ is an integer or an integer modulo two depending on $\varepsilon$; and $(\alpha_i,\beta_i)$ are pairs of relatively prime integers, with $r$ the number of exceptional fibres.  We refer the reader to \cite[Chapter 5, Theorem 3]{orlik_1972} for further details. 

The fundamental group then has the following presentation
\begin{align*}
    \pi_1(Y) = \langle &a_1, b_1, \dots, a_g, b_g, q_1, \dots, q_r, h \mid a_i ha_i^{-1}=h^{\varepsilon_i}, b_ihb_i^{-1}=h^{\varepsilon_i}, \\ &q_jhq_j^{-1}=h, q_j^{\alpha_j}h^{\beta_j}=1, q_1\cdots q_r[a_1,b_1]\cdots[a_g,b_g]=h^{b}\rangle,
\end{align*}
in the case that $\varepsilon =o_1$ or $o_2$ and 
\begin{align*}
    \pi_1(Y) = \langle &v_1, \dots v_g, q_1, \dots, q_r, h \mid v_i hv_i^{-1}=h^{\varepsilon_i}, q_jhq_j^{-1}=h, \\ &q_j^{\alpha_j}h^{\beta_j}=1, q_1\cdots q_rv_1^2\cdots v_g^2=h^{b}\rangle,
\end{align*}
in the case that $\varepsilon =n_1, n_2, n_3$ or $n_4$ (here $o$ stands for orientable and $n$ for non-orientable, with these terms referring to the base space of the orbifold for the Seifert fibration).

\begin{lem}\label{lem:seifert_ambivalence}
    Let $Y(b;(\varepsilon,g);(\alpha_1,\beta_1),\dots, (\alpha_r,\beta_r))$ be a Seifert fibred 3-manifold with $\varepsilon =o_1$ or $n_1$ and $\lvert h\rvert > 2$.  Then $\pi_1(Y)$ is not ambivalent.
\end{lem}

\begin{proof}
    As noted in \cite[Section 5.3]{orlik_1972}, $\varepsilon =o_1$ or $n_1$ corresponds precisely to the cases when $h$ is central in $\pi_1(Y)$.  By our assumption on the order of $h$ and (2) of \Cref{rem:ambivalence}, this implies that $\pi_1(Y)$ is not ambivalent.
\end{proof}

We verify a basic set of examples which satisfy this condition.

\begin{proposition}\label{prop:ambivalence_surfaces}
    Let $Y$ be the $S^1$ bundle over the orientable surface bundle of genus $k$ with Euler number $e\neq 1$.  Then $\pi_1(Y)$ is not ambivalent.
\end{proposition}

\begin{proof}
    The Seifert invariants of $Y$ are as follows: $b=e, \varepsilon =o_1$ and $g=k$.  By \Cref{lem:seifert_ambivalence} we only have to verify that $\lvert h \rvert >2$.  This is clear from the group presentation given above.
\end{proof}

Note that \Cref{prop:ambivalence_surfaces} includes the 3-torus, since this is the Euler number zero $S^1$ bundle over the 2-torus. Note that all of the examples in \Cref{prop:ambivalence_surfaces} have trivial $k_1$ since they all have trivial $\pi_2$ except for $S^1\times S^2$ (this can be deduced from the long exact sequence of homotopy groups coming from the fibration).  As a separate case, one can also see that $S^1\times S^2$ has trivial $k_1$, since $K(\Z,1)$ is 1-dimensional.

\begin{remark}\label{rem:goodness}
    For applications involving our realisation theorem \Cref{thm:realisation_theta}, we also need the fundamental group of the 3-manifold to be \emph{good}, in the sense of Freedman-Quinn \cite[Chapter 2.9]{freedman_quinn_1990} (c.f.\ \cite[Definition 12.12]{behrens_kalmar_kim_powell_ray_2021}).  Since all finite groups are good \cite[Section 5.1]{freedman_quinn_1990} (c.f.\ \cite[Theorem 19.2]{behrens_kalmar_kim_powell_ray_2021}), this does not provide any restrictions on the elliptic 3-manifolds considered above.  However, it does provide a restriction to the class of 3-manifolds considered in \Cref{prop:ambivalence_surfaces}; in particular, we have to restrict to $S^1$ bundles over surfaces of genus $k< 2$, since higher genus surface groups are not known to be good.  When $k< 2$, the goodness of $\pi_1(Y)$ follows since goodness is closed under extensions \cite[Exercise 2.9]{freedman_quinn_1990} (c.f.\ \cite[Proposition 19.5]{behrens_kalmar_kim_powell_ray_2021}, and $\Z$ is good \cite[Section 5.1]{freedman_quinn_1990} (c.f.\ \cite[Theorem 19.4]{behrens_kalmar_kim_powell_ray_2021}).
\end{remark}

\begin{remark}
    By \Cref{prop:ambivalence} and \Cref{rem:goodness}, \Cref{thm:interesting_homeomorphisms} applies to all lens spaces except $S^3$ and $\RP^3$, `most' prism manifolds, all tetrahedral manifolds, all but one octahedral manifold and all icosahedral manifolds except the Poincar\'{e} homology 3-sphere.  Since all of the finite fundamental group manifolds have vanishing $\pi_2$, this means that our topological second Hatcher-Wagoner invariant utterly fails to detect any interesting homeomorphisms in these cases.
    
    By \Cref{prop:ambivalence_surfaces} and \Cref{rem:ambivalence}, \Cref{thm:interesting_homeomorphisms} also applies to all $S^1$ bundles over tori, including the 3-torus.
\end{remark}

\subsection{Homeomorphisms of $Y^3\times S^1$}\label{sbs:YxS1}
In this subsection we will describe how to take the homeomorphisms constructed in \Cref{sbs:YxI} and produce interesting homeomorphisms of $Y^3\times S^1$.  It is clear that, given a homeomorphism $f\colon M\times I\to M\times I$ that restricts to the identity on the boundary, we can glue $M\times \{0\}$ to $M\times \{1\}$ via the identity map and produce a new homeomorphism $\ol{f}\colon M\times S^1\to M\times S^1$.  However, it is conceivable for $\ol{f}$ to be isotopic to the identity even if $f$ was not isotopic to the identity.  This is because $M\times \{\pt\}\subset M\times S^1$ may not be fixed throughout the isotopy.

Igusa \cite[Lemma 5.1]{igusa_2021} showed\footnote{At the beginning of \cite[Section 5]{igusa_2021_2}, Igusa says that the lemma is well-known, but we know of no other reference and hence attribute it to him.}, that any isotopy of $\ol{f}$ to the identity could be deformed such that it fixes $M\times \{\pt\}$ throughout.

\begin{lem}[{\cite[Lemma 5.1]{igusa_2021_2}}]\label{lem:clam}
    Let $M$ be a compact, topological $n$-manifold.  Then the map $\pi_0(\mathcal{P}^{\TOP}(M\times I,\partial) \to \pi_0(\mathcal{P}^{\TOP}(M\times S^1)$ given by gluing $M\times \{0\}$ to $M\times \{1\}$ is injective.
\end{lem}

The above lemma is stated in the smooth category in \cite{igusa_2021_2} but the smooth category is not used essentially during the proof and all of the arguments follow through exactly the same in the topological category.

As a corollary, we obtain interesting homeomorphisms of $M\times S^1$ whenever we previously had interesting homeomorphisms of $M\times I$.  In particular, we have the following result which was \Cref{thm:homeos} from the introduction.

\begin{thm}\label{thm:lens_space_homeomorphisms_closed}
    Let $Y^3$ be a 3-manifold whose first $k$-invariant $k_1(Y)$ is trivial and with $\pi_1(Y)$ good and not ambivalent.  Then there exists a homeomorphism $f\colon Y\times S^1\to Y\times S^1$ which is pseudo-isotopic to the identity but not isotopic to the identity.  In particular, $f$ is homotopic but not isotopic to the identity.
\end{thm}

\begin{proof}
    Apply \Cref{lem:clam} to the homeomorphisms produced by \Cref{thm:interesting_homeomorphisms}.
\end{proof}

\clearpage
	\bibliographystyle{alpha}
	\bibliography{bibliography.bib}
	
\end{document}